\documentclass[UTF8,fontset = windows,preprint,review,12pt]{elsarticle}

\usepackage{ctex}
\usepackage{amssymb}
\usepackage{graphicx}
\usepackage{subfigure}
\usepackage{extarrows}
\usepackage{fancyhdr}
\usepackage{amsthm}
\usepackage{extarrows}
\usepackage{caption}
\usepackage{float}
\usepackage{algorithm}
\usepackage{algorithmic}
\usepackage{multirow}
\usepackage{amsmath}
\usepackage{color}

\usepackage[toc,page,title,titletoc,header]{appendix}
\usepackage{appendix}
\theoremstyle{definition}

\theoremstyle{definition}

\theoremstyle{definition}
\usepackage{latexsym, bm}
\usepackage{amsmath}
\usepackage{amssymb}
\usepackage{amsfonts}
\usepackage{amsthm}
\usepackage{fancyhdr}
\usepackage{mathrsfs}
\usepackage{wasysym}
\usepackage{float}
\usepackage{graphicx}
\usepackage{multirow}
\usepackage{pgfplots}
\usepackage{tikz}
\usepackage{amsmath}
\usepackage{amssymb}
\usepackage{amsfonts}
\usepackage{amsthm}
\usepackage{fancyhdr}
\usepackage{latexsym,bm}
\usepackage{mathrsfs}
\usepackage{wasysym}
\usepackage{graphicx}
\usepackage{subfigure}

\usepackage{caption}

\allowdisplaybreaks[4]

\biboptions{numbers,sort&compress}
\usetikzlibrary{arrows,shapes,positioning}
\usetikzlibrary{decorations.markings}
\tikzstyle arrowstyle=[scale=1]
\tikzstyle directed=[postaction={decorate,decoration={markings,
    mark=at position .65 with {\arrow[arrowstyle]{stealth}}}}]
\tikzstyle reverse directed=[postaction={decorate,decoration={markings,
    mark=at position .65 with {\arrowreversed[arrowstyle]{stealth};}}}]
\newtheorem{defn}{Definition}[section]
\newtheorem{lemm}{Lemma}[section]
\newtheorem{thrm}{Theorem}[section]

\newtheorem{rem}{Remark}[section]

\begin{document}

\begin{frontmatter}

\title{Multiphysics mixed finite element method with Nitsche's technique for Stokes-poroelasticity problem\tnoteref{1}}\tnotetext[1]{The work was supported by the National Natural Science Foundation of China under grant No. 11971150.\\
*Corresponding author. Email: zhihaoge@henu.edu.cn.}
\author{Zhihao G${\rm e^{1,*}}$,\ Jin'ge Pan${\rm g^{1}}$,\ Jiwei Ca${\rm o^{2}}$}
\address{$ ^1$School of Mathematics and Statistics, Henan University, Kaifeng 475004, P.R. China\\
$ ^2$College of Mathematics and Information Science, Henan University of Economics and Law, Zhengzhou 450000, P.R. China}

\begin{abstract}
In this paper, we propose a  multiphysics mixed finite element method with Nitsche's technique for Stokes-poroelasticity problem. Firstly, we present a multiphysics reformulation of poroelasticity part of the original problem by introducing two pseudo-pressures to reveal the underlying deformation and diffusion multi physical processes in the Stokes-poroelasticity problem. Then, we prove the existence and uniqueness of weak solution of the reformulated and original problem. And we use Nitsche's technique to approximate the coupling condition at the interface to propose a loosely-coupled time-stepping method-- multiphysics mixed finite element method for space variables, and we decouple the reformulated  problem into three sub-problems at each time step--a Stokes problem, a generalized Stokes problem and a mixed diffusion problem. Also,  we give the stability analysis and error estimates of the loosely-coupled time-stepping method. Finally, we show the numerical tests to verify the  theoretical results, which has a good stability and no ¡°locking" phenomenon. 
\end{abstract}
\end{frontmatter}

\thispagestyle{empty}

\numberwithin{equation}{section}

\section{Introduction}\label{sec1}

Stoke-poroelasticity model is an interaction of a free fluid and  poroelastic, which is  widely applied in reservoir engineering \cite{M.B.Dusseault2001,S.Michael2002,Y.Wang2003}, the groundwater flow \cite{J.Kim,A.Mikelic,R.I.Borja}, the poroelastic structure and crack \cite{T.Wick,M.Lesinigo,M.F.Wheeler}, in biomechanics \cite{M.Prosi,G.A.Holzapfel,S.M.K. Rausch}, and so on. For the the Stokes-poroelasticity problem, there are few theoretical results. As for the numerical results, a domain decomposition (DD) method is proposed in \cite{J.M.Connors,W.J.Layton}. The monolithic multigrid method have been successfully applied for the system of Stokes equations in \cite{S.Vanka} and poroelasticity equations in \cite{F.Gaspar}, the Stokes-poroelasticity problem in \cite{I.Ambartsumyan,S.Muntz}.  The authors of \cite{S.Badia} adopt an extended domain decomposition method for the fluid-poroelastic structure interaction (FPSI) problem. In \cite{E.Burman}, a loosely coupled scheme using a Nitsche type weak coupling together with an interface pressure stabilization in time is proposed. However, the accuracy is relatively low and the computational efficiency of the scheme is reduced by using iterative correction. Recently, the multiphysics finite element method is proposed by  \cite{X.B.Feng2014} for the poroelasticity problem to reveals the underlying deformation and diffusion multi physical processes and overcome the ¡°locking phenomenon¡±.

In this paper, we consider the flow of a viscous fluid in a channel bounded by a poroelastic medium, which is described as follows:
\begin{eqnarray}
	-2\mu_{f}\nabla\cdot\mathbf{D}(\mathbf{v})+\nabla p_{f}&=&\mathbf{f},~~~~~~(\mathbf{x},t)\in\Omega_{f}\times(0,T], \label{1.1'}\\
	\nabla\cdot\mathbf{v}&=&g,~~~~~~(\mathbf{x},t)\in\Omega_{f}\times(0,T],\label{1.2'} \\
	-2\mu_{p}\nabla\cdot\mathbf{D}(\mathbf{U})-\lambda_{p}\nabla(\nabla\cdot\mathbf{U})+\alpha\nabla p_{p} &=&\mathbf{h},~~~~~~(\mathbf{x},t)\in\Omega_{p}\times(0,T], \label{1.5'}\\
	-k\nabla p_{p}&=&\mathbf{q},~~~~~~(\mathbf{x},t)\in\Omega_{p}\times(0,T],\label{1.6'} \\
	(s_{0}p_{p}+\alpha\nabla\cdot\mathbf{U})_{t}+\nabla\cdot\mathbf{q}&=&s,~~~~~~~(\mathbf{x},t)\in\Omega_{p}\times(0,T], \label{1.7'}
\end{eqnarray}
where $\Omega_{f}$ is a Stokes flow domain, and $\Omega_{p}$ is a poroelastic material. $v$ is the fluid velocity, $\sigma_{f}(\mathbf{v},p_{f})=-p_{f}\mathbf{I}+2\mu_{f}\mathbf{D}(\mathbf{v})$ is the fluid stress tensor,~$p_{f}$~ is the fluid pressure,~$\mu_{f}$~is the fluid viscosity and~$\mathbf{D}(\mathbf{v})=\frac{1}{2}(\nabla\mathbf{v}+(\nabla\mathbf{v})^{T})$~ is the rate-of-strain tensor.~$\mathbf{h}$, $s$, $\mathbf{f}$~and~$g$~are the body force. The stress tensor of the poroelastic medium is  $\sigma_{p}=\sigma^{E}-\alpha p_{p}\mathbf{I}$, where~$\sigma^{E}=\lambda_{p}(\nabla\cdot\mathbf{U})\mathbf{I}+2\mu_{p}\mathbf{D}(\mathbf{U})$. $\lambda_{p}$~and~$\mu_{p}$~denote the Lam\'{e} coefficients for the skeleton, and   $\mathbf{D}(\mathbf{U})=\frac{1}{2}(\nabla\mathbf{U}+(\nabla\mathbf{U})^{T})$. The equation of (\ref{1.6'}) is called by  Darcy's law, where the flux~$\mathbf{q}$~is the relative velocity of the fluid within the porous structure and~$p_{p}$~is the fluid pressure. The hydraulic conductivity is denoted by~$k$, which is in general a symmetric positive definite tensor. The coefficient~$s_{0}\in(0, 1)$~is the storage coefficient, and the Biot¨CWillis constant~$\alpha$~is the pressure¨Cstorage coupling coefficient. The fluid domain is bounded by a deformable porous matrix consisting of a skeleton and connecting pores filled with fluid, the following two configurations are considered: (i) the channel extends to the external boundary (Figure \ref{fig11}), and (ii) the channel is surrounded by the poroelastic media (Figure \ref{fig22}).
\begin{figure}[H]
	\centering
	\begin{minipage}[t]{0.49\linewidth}
		\centering
		\includegraphics[height=10cm,width=10cm]{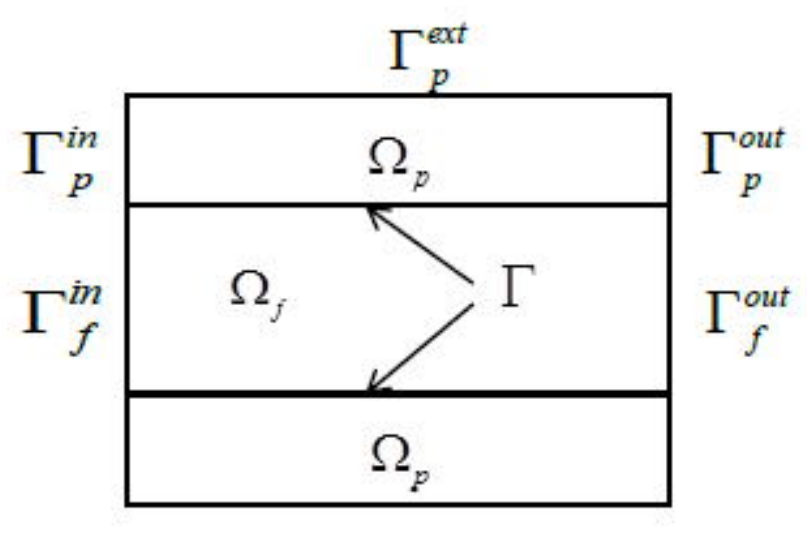}
		\caption{\small channel extends to external boundary}\label{fig11}
	\end{minipage}
	\begin{minipage}[t]{0.49\linewidth}
		\centering
		\includegraphics[height=10cm,width=10cm]{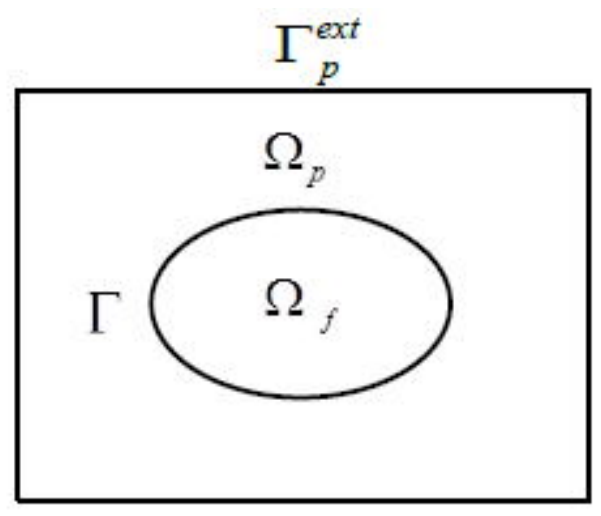}
		\caption{\small channel is surrounded by poroelastic media}\label{fig22}
	\end{minipage}
\end{figure}

Configuration (i) is suitable FPSI in arteries, and the configuration (ii) is often applied to the fractured reservoirs. As for configuration (i),  we prescribe the following boundary and initial conditions:
\begin{eqnarray}
	\sigma_{f}\mathbf{n}_{f}=-p_{in}(t)\mathbf{n}_{f},&&(\mathbf{x},t)\in\Gamma_{f}^{in}\times(0,T],\label{1.3} \\
	\sigma_{f}\mathbf{n}_{f}=0,&&(\mathbf{x},t)\in\Gamma_{f}^{out}\times(0,T], \label{1.4}\\
	\mathbf{U}=0,&&(\mathbf{x},t)\in\Gamma_{p}^{in}\cup\Gamma_{p}^{out}\times(0,T],\label{1.8}\\
	\mathbf{n}_{p}\cdot\sigma^{E}\mathbf{n}_{p}=0,&&(\mathbf{x},t)\in\Gamma_{p}^{ext}\times(0,T],\label{1.9}\\
	\mathbf{U}\cdot\tau_{p}=0,&&(\mathbf{x},t)\in\Gamma_{p}^{ext}\times(0,T],\label{1.10}\\
	p_{p}=0,~~~~~(\mathbf{x},t)\in\Gamma_{p}^{ext}\times(0,T],~~\mathbf{q}\cdot\mathbf{n}_{p}=0,&&(\mathbf{x},t)\in\Gamma_{p}^{in}\cup\Gamma_{p}^{out}\times(0,T],\label{1.11}\\
	\mathbf{v}^{0}=0,&&\mathbf{U}^{0}=0,~~~~~p_{p}^{0}=0.\label{1.12}
\end{eqnarray}


As for the configuration (ii), we can suppose that $\mathbf{U}=0$~and~$\mathbf{q}\cdot\mathbf{n}_{p}=0$~ on~$\partial\Omega_{p}:=\Gamma_{p}^{ext}$.

Next, we give the interface conditions on $\Gamma$. To do that, we denote  the outward normal to the fluid domain by $\mathbf{n}$ and the tangential unit vector on the interface ~$\Gamma$~by $\tau$. The interface conditions are given by
\begin{eqnarray}
	(\mathbf{v}-\frac{\partial\mathbf{U}}{\partial t})\cdot\mathbf{n}=\mathbf{q}\cdot\mathbf{n},&&(\mathbf{x}.t)\in\Gamma\times(0,T],\label{1.13}\\
	\mathbf{v}\cdot\tau=\frac{\partial\mathbf{U}}{\partial t}\cdot\tau,&&(\mathbf{x},t)\in\Gamma\times(0,T],\label{1.14}\\
	-\mathbf{n}\cdot\sigma_{f}\mathbf{n}=p_{p},&&(\mathbf{x},t)\in\Gamma\times(0,T],\label{1.16}\\
	\mathbf{n}\cdot\sigma_{f}\mathbf{n}-\mathbf{n}\cdot\sigma_{p}\mathbf{n}=0,&&(\mathbf{x},t)\in\Gamma\times(0,T], \label{1.17}\\
	\tau\cdot\sigma_{f}\mathbf{n}-\tau\cdot\sigma_{p}\mathbf{n}=0,&&(\mathbf{x},t)\in\Gamma\times(0,T].\label{1.18}
\end{eqnarray}
%
The condition (\ref{1.14}) can be replaced by  Beavers¨CJoseph¨CSaffman condition (cf. \cite{F.Schieweck}):
\begin{equation}\label{1.15}
	\beta(\mathbf{v}-\frac{\partial\mathbf{U}}{\partial t})\cdot\tau=-\tau\cdot\sigma_{f}\mathbf{n},~~~~~(\mathbf{x},t)\in\Gamma\times(0,T],
\end{equation}
where the parameter $\beta$ quantifies the resistance that the porous matrix opposes to fluid flow in the tangential direction.

As for the problem (\ref{1.1'})-(\ref{1.18}), we use the multiphysics mixed finite element method with Nitsche's technique to solve the Stokes-poroelasticity problem. Firstly, we define the weak solution of the reformulated problem and prove the existence and uniqueness of weak solution. Then, we use Nitsche's technique to deal with  the interface condition to propose a fully discrete multiphysics finite element method, which decouples the problem into three sub-problems at each time step: a Stokes problem, a generalized Stokes problem and a mixed diffusion problems. And we give the stability analysis and error estimates of the decoupled finite element method. Finally, we give the numerical tests to verify the effectiveness and feasibility of our method. To the best of our knowledge, it is first time to use the multiphysics finite element method with Nitsche's technique to solve the Stokes-poroelasticity problem. 

The remainder of this paper is organized as follows. In Section \ref{sec2}, we reformulate the original problem based on multiphysics approach and prove the existence and uniqueness of weak solution of the reformulated problem. In Section \ref{sec3}, we use Nitsche's technique to deal with  the interface condition and propose a fully discrete multiphysics mixed finite element method to decouple the problem into three sub-problems at each time step. And we give the stability analysis and optimal error estimates of the proposed numerical method. In Section \ref{sec4}, we show the numerical tests to verify the  theoretical results, which has a good stability and no ¡°locking" phenomenon.  Finally, we draw a conclusion to summary the main results in this paper.

\section{Multiphysics approach and PDE analysis}\label{sec2}

To reveal the multi physics processes underlying in the original problem, we introduce a new variable $\phi:=\nabla\cdot\mathbf{U}$, and define the following two pseudo-pressures
\begin{equation*}
    \xi:=\alpha p_{p}-\lambda_{p}\phi,~~~~\eta:=s_{0}p_{p}+\alpha\phi.
\end{equation*}
It is easy to check that
\begin{equation}\label{1.19}
    p_{p}=k_{1}\xi+k_{2}\eta,~~~~\phi=k_{1}\eta-k_{3}\xi,
\end{equation}
where
\begin{equation}\label{1.20}
    k_{1}=\frac{\alpha}{\alpha^{2}+\lambda_{p}s_{0}},~~~~k_{2}=\frac{\lambda_{p}}{\alpha^{2}+\lambda_{p}s_{0}},~~~~k_{3}=\frac{s_{0}}{\alpha^{2}+\lambda_{p}s_{0}}.
\end{equation}

Then the problem (\ref{1.1'})-(\ref{1.7'}) can be reformulate into 
\begin{eqnarray}
  -2\mu_{f}\nabla\cdot\mathbf{D}(\mathbf{v})+\nabla p_{f}&=&\mathbf{f},~~~~~~~~~~~~~(\mathbf{x},t)\in\Omega_{f}\times(0,T], \label{1.21}\\
  \nabla\cdot\mathbf{v}&=&g,~~~~~~~~~~~~~(\mathbf{x},t)\in\Omega_{f}\times(0,T],\label{1.22} \\
  -2\mu_{p}\nabla\cdot\mathbf{D}(\mathbf{U})+\nabla \xi&=&\mathbf{h},~~~~~~~~~~~~~(\mathbf{x},t)\in\Omega_{p}\times(0,T], \label{1.23}\\
  k_{3}\xi+\nabla\cdot\mathbf{U}&=&k_{1}\eta,~~~~~~~~~~(\mathbf{x},t)\in\Omega_{p}\times(0,T],\label{1.24} \\
  -k\nabla(k_{1}\xi+k_{2}\eta)&=&\mathbf{q},~~~~~~~~~~~~~(\mathbf{x},t)\in\Omega_{p}\times(0,T],\label{1.25} \\
  \eta_{t}+\nabla\cdot\mathbf{q}&=&s,~~~~~~~~~~~~~~(\mathbf{x},t)\in\Omega_{p}\times(0,T]. \label{1.26}
\end{eqnarray}

\begin{rem}\label{rem200527-2}
	The coupled problem (\ref{1.21})-(\ref{1.26}) reveals the underlying deformation and diffusion multiphysics process, which occurs in the Stokes-poroelasticity problem (\ref{1.1'})-(\ref{1.7'}). It should be noted that the original pressure is eliminated in the reformulation, which will be helpful to overcome the "locking" phenomenon.
\end{rem}


To define a weak solution of the problem (\ref{1.21})-(\ref{1.26}), we introduce the following function spaces
\begin{eqnarray*}
  \mathbf{V}^{f} &=& \{\varphi_{f}\in H^{1}(\Omega_{f})^{d}:\varphi_{f}=0,\mathbf{x}\in\Gamma_{f}^{out}\}, \ \
  Q^{f} = \{\psi_{f}\in L^{2}(\Omega_{f})\}, \\
  \mathbf{X}^{p} &=& \{\varphi_{p}\in H^{1}(\Omega_{p})^{d}:\varphi_{p}=0,\mathbf{x}\in\Gamma_{p}^{in}\cup\Gamma_{p}^{out};\varphi_{p}\cdot\tau_{p}=0,\mathbf{x}\in\Gamma_{p}^{ext}\}, \\
  \mathbf{V}^{p} &=& \{\mathbf{r}\in H(\mathrm{div};\Omega_{p}):\mathbf{r}\cdot\mathbf{n}_{p}=0,\mathbf{x}\in\Gamma_{p}^{in}\cup\Gamma_{p}^{out}\}, \\
  M^{p}&=&\{w\in L^{2}(\Omega_{p})\},\ \
  Q^{p}=\{z\in H^{1}(\Omega_{p})\},
\end{eqnarray*}
with the norm  
$
  \|\mathbf{v}\|^{2}_{H(\mathrm{div};\Omega_{p})}=\|\mathbf{v}\|^{2}_{L^{2}(\Omega_{p})}+\|\nabla\cdot\mathbf{v}\|^{2}_{L^{2}(\Omega_{p})}$.

For convenience, we assume that the functions $\mathbf{f},~g,~\mathbf{h}$~and~$s$~all are independent of $t$ in the remaining of the paper. We note that all the results of this paper can be easily extended to the case of time-dependent source functions.
\begin{defn}\label{defn2.1}
Let~$\mathbf{v}^{0}\in H^{1}(\Omega_{f}),~\mathbf{U}^{0}\in H^{1}(\Omega_{p}),~p_{p}^{0}\in L^{2}(\Omega_{p}),~\mathbf{f},g\in L^{2}(\Omega_{f}),~\mathbf{h},s\in L^{2}(\Omega_{p})$,~for any~$t\in(0, T ]$, we say that  $\mathbf{v}\in L^{\infty}(0,T;\mathbf{V}^{f}), p_{f}\in L^{\infty}(0,T;Q^{f}), \mathbf{U}\in L^{\infty}(0,T;\mathbf{X}^{p}), \mathbf{q}\in L^{\infty}(0,T;\mathbf{V}^{p}), p_{p}\in L^{\infty}(0,T;Q^{p})$ are a weak solution of the problem (\ref{1.1'})-(\ref{1.7'}) if there hold
\begin{eqnarray}
  -\langle\sigma_{f}\mathbf{n}_{f},\varphi_{f}\rangle_{\Gamma}+(2\mu_{f}\mathbf{D}(\mathbf{v}),
   \mathbf{D}(\varphi_{f}))_{\Omega_{f}}
   -(p_{f},\nabla\cdot\varphi_{f})_{\Omega_{f}}+\langle p_{in}(t)\mathbf{n}_{f},\varphi_{f}\rangle_{\Gamma_{f}^{in}}
    &=&(\mathbf{f},\varphi_{f})_{\Omega_{f}},\label{2.38} \\
  (\nabla\cdot\mathbf{v},\psi_{f})_{\Omega_{f}}
   &=&(g,\psi_{f})_{\Omega_{f}},\label{2.39}\\
   -\langle\sigma_{p}\mathbf{n}_{p},\varphi_{p}\rangle_{\Gamma}
  +2\mu_{p}(\mathbf{D}(\mathbf{U}),\mathbf{D}(\varphi_{p}))_{\Omega_{p}}
  +\lambda_{p}(\nabla\cdot\mathbf{U},\nabla\cdot\varphi_{p})_{\Omega_{p}}
   -\alpha(p_{p},\nabla\cdot\varphi_{p})_{\Omega_{p}}&=&(\mathbf{h},\varphi_{p})_{\Omega_{p}}, \label{2.40}\\
  k^{-1}(\mathbf{q},\mathbf{r})_{\Omega_{p}}-(p_{p},\nabla\cdot \mathbf{r})_{\Omega_{p}}+\langle p_{p}\cdot\mathbf{n}_{p},\mathbf{r}\rangle_{\Gamma}&=&0,\label{2.41}\\
  ((s_{0}p_{p}+\alpha\nabla\cdot\mathbf{U})_{t},z)_{\Omega_{p}}+(\nabla\cdot \mathbf{q},z)_{\Omega_{p}}
  &=&(s,z)_{\Omega_{p}},\label{2.42}
\end{eqnarray}
for any $(\varphi_{f}, \psi_{f}, \varphi_{p}, \mathbf{r}, z)\in\mathbf{V}^{f}\times Q^{f}\times\mathbf{X}^{p}\times\mathbf{V}^{p}\times
  Q^{p}$.
\end{defn}

Similarly, we can define the weak solution to the problem (\ref{1.21})-(\ref{1.26}).

\begin{defn}\label{defn2.2}
Let~$\mathbf{v}^{0}\in H^{1}(\Omega_{f}),~\mathbf{U}^{0}\in H^{1}(\Omega_{p}),~p_{p}^{0}\in L^{2}(\Omega_{p}),~\mathbf{f},g\in L^{2}(\Omega_{f}),~\mathbf{h},s\in L^{2}(\Omega_{p})$ for any $t\in(0, T)$, we say that $\mathbf{v}\in L^{\infty}(0,T;\mathbf{V}^{f}), p_{f}\in L^{\infty}(0,T;Q^{f}), \mathbf{U}\in L^{\infty}(0,T;\mathbf{X}^{p}), \mathbf{q}\in L^{\infty}(0,T;\mathbf{V}^{p}), \xi\in L^{\infty}(0,T;M^{p}), \eta\in L^{\infty}(0,T;Q^{p})$ are a weak solution of the problem (\ref{1.21})-(\ref{1.26}) if there hold
\begin{eqnarray}
  -\langle\sigma_{f}\mathbf{n}_{f},\varphi_{f}\rangle_{\Gamma}+(2\mu_{f}\mathbf{D}(\mathbf{v}),
   \mathbf{D}(\varphi_{f}))_{\Omega_{f}}~~~~~~~~~~~~~~~&&\nonumber\\
   -(p_{f},\nabla\cdot\varphi_{f})_{\Omega_{f}}+\langle p_{in}(t)\mathbf{n}_{f},\varphi_{f}\rangle_{\Gamma_{f}^{in}}
    &=&(\mathbf{f},\varphi_{f})_{\Omega_{f}},\label{2.15} \\
  (\nabla\cdot\mathbf{v},\psi_{f})_{\Omega_{f}}
   &=&(g,\psi_{f})_{\Omega_{f}},\label{2.16}\\
   -\langle\sigma_{p}\mathbf{n}_{p},\varphi_{p}\rangle_{\Gamma}
  +2\mu_{p}(\mathbf{D}(\mathbf{U}),\mathbf{D}(\varphi_{p}))_{\Omega_{p}}-(\xi,\nabla\cdot\varphi_{p})_{\Omega_{p}}
   &=&(\mathbf{h},\varphi_{p})_{\Omega_{p}}, \label{2.17}\\
  k_{3}(\xi,w)_{\Omega_{p}}+(\nabla\cdot\mathbf{U},w) _{\Omega_{p}}&=& k_{1}(\eta,w)_{\Omega_{p}},\label{2.18}\\
  k^{-1}(\mathbf{q},\mathbf{r})_{\Omega_{p}}-(k_{1}\xi+k_{2}\eta,\nabla\cdot \mathbf{r})_{\Omega_{p}}+\langle(k_{1}\xi+k_{2}\eta)\cdot\mathbf{n}_{p},\mathbf{r}\rangle_{\Gamma}&=&0,\label{2.19}\\
  (\partial_{t}\eta,z)_{\Omega_{p}}+(\nabla\cdot \mathbf{q},z)_{\Omega_{p}}
  &=&(s,z)_{\Omega_{p}},\label{2.20}
\end{eqnarray}
for any $(\varphi_{f}, \psi_{f}, \varphi_{p}, \mathbf{r}, w, z)\in\mathbf{V}^{f}\times Q^{f}\times\mathbf{X}^{p}\times\mathbf{V}^{p}\times
 M^{p}\times  Q^{p}$.
\end{defn}

\begin{lemm}\label{lemn2}
Every weak solution~$(\mathbf{v},p_{f},\mathbf{U},\mathbf{q},\xi,\eta)$ of the problem (\ref{2.15})-(\ref{2.20}) satisfies the following energy law:
	\begin{equation}\label{2.23}
		\mathbf{J}(t)+\int_{0}^{t}[2\mu_{f}\|\mathbf{D}(\mathbf{v})\|^{2}_{L^{2}(\Omega_{f})}+k^{-1}\|\mathbf{q}\|^{2}_{L^{2}(\Omega_{p})}
		+\beta\|(\mathbf{v}-\mathbf{U}_{t})\cdot\tau\|^{2}_{L^{2}(\Gamma)}]dt-\int_{0}^{t}\mathbf{F}(t)dt=\mathbf{J}(0),
	\end{equation}
	for all~$t\in(0,T]$, where
	\begin{eqnarray}
		&\mathbf{J}(t)=\frac{1}{2}[2\mu_{p}\|\mathbf{D}(\mathbf{U}(t))\|^{2}_{L^{2}(\Omega_{p})}+k_{2}\|\eta(t)\|^{2}_{L^{2}(\Omega_{p})}+k_{3}\|\xi(t)\|^{2}_{L^{2}(\Omega_{p})}-2(\mathbf{h},\mathbf{U}(t))_{\Omega_{p}}],\label{2.24} \\
		&\mathbf{F}(t)=(\mathbf{f},\mathbf{v})_{\Omega_{f}}+\langle p_{in}(t)\mathbf{n}_{f},\mathbf{v}\rangle_{\Gamma_{f}^{in}}+(g,p_{f})_{\Omega_{f}}+(s,p_{p})_{\Omega_{p}}.\label{2.25}
	\end{eqnarray}
\end{lemm}
\begin{proof}
Differentiating (\ref{2.18}) with respect to~$t$, taking $
		\varphi_{f}=\mathbf{v}, \mathbf{r}=\mathbf{q}, \psi_{f}=p_{f}, \varphi_{p}=\mathbf{U}_{t}, w=\xi, z=p_{p}=k_{1}\xi+k_{2}\eta$,  we have
	\begin{eqnarray}
			&-\langle\sigma_{f}\mathbf{n}_{f},\mathbf{v}\rangle_{\Gamma}+2\mu_{f}\|\mathbf{D}(\mathbf{v})\|^{2}_{L^{2}(\Omega_{f})}
		=(\mathbf{f},\mathbf{v})_{\Omega_{f}}+\langle-p_{in}(t)\mathbf{n}_{f},\mathbf{v}\rangle_{\Gamma_{f}^{in}}+(g,p_{f})_{\Omega_{f}},\label{2.51}\\		
			&\mu_{p}\frac{d}{dt}\|\mathbf{D}(\mathbf{U})\|^{2}_{L^{2}(\Omega_{p})}
			+\frac{k_{2}}{2}\frac{d}{dt}\|\eta\|^{2}_{L^{2}(\Omega_{p})}+\frac{k_{3}}{2}\frac{d}{dt}\|\xi\|^{2}_{L^{2}(\Omega_{p})}-\langle\sigma_{p}\mathbf{n}_{p},\mathbf{U}_{t}\rangle_{\Gamma}\nonumber\\
			&+\langle p_{p}\mathbf{n}_{p},\mathbf{q}\rangle_{\Gamma}
			+k^{-1}\|\mathbf{q}\|^{2}_{L^{2}(\Omega_{p})}
			=(s,p_{p})_{\Omega_{p}}+\frac{d}{dt}(\mathbf{h},\mathbf{U})_{\Omega_{p}}.\label{2.52}
			\end{eqnarray}
	Adding (\ref{2.51}) and (\ref{2.52}), we get
	\begin{eqnarray}\label{2.54}
		\begin{aligned}
			&\beta\|(\mathbf{v}-\mathbf{U}_{t})\cdot\tau\|_{L^{2}(\Gamma)}^{2}+2\mu_{f}\|\mathbf{D}(\mathbf{v})\|^{2}_{L^{2}(\Omega_{f})}
			+k^{-1}\|\mathbf{q}\|^{2}_{L^{2}(\Omega_{p})}\\
			&+\mu_{p}\frac{d}{dt}\|\mathbf{D}(\mathbf{U})\|^{2}_{L^{2}(\Omega_{p})}
			+\frac{k_{2}}{2}\frac{d}{dt}\|\eta\|^{2}_{L^{2}(\Omega_{p})}+\frac{k_{3}}{2}\frac{d}{dt}\|\xi\|^{2}_{L^{2}(\Omega_{p})}\\
			=&(\mathbf{f},\mathbf{v})_{\Omega_{f}}+\langle-p_{in}(t)\mathbf{n}_{f},\mathbf{v}\rangle_{\Gamma_{f}^{in}}
			+(g,p_{f})_{\Omega_{f}}+(s,p_{p})_{\Omega_{p}}+\frac{d}{dt}(\mathbf{h},\mathbf{U})_{\Omega_{p}}.
		\end{aligned}
	\end{eqnarray}
Integrating (\ref{2.54}) from $[0, t]$, we obtain
	\begin{eqnarray*}
		&&\int^{t}_{0}[\beta\|(\mathbf{v}-\mathbf{U}_{t})\cdot\tau\|_{L^{2}(\Gamma)}^{2}+2\mu_{f}\|\mathbf{D}(\mathbf{v})\|^{2}_{L^{2}(\Omega_{f})}
		+k^{-1}\|\mathbf{q}\|^{2}_{L^{2}(\Omega_{p})}]dt\\
		&&+\mu_{p}\|\mathbf{D}(\mathbf{U}(t))\|^{2}_{L^{2}(\Omega_{p})}
		+\frac{k_{2}}{2}\|\eta(t)\|^{2}_{L^{2}(\Omega_{p})}+\frac{k_{3}}{2}\|\xi(t)\|^{2}_{L^{2}(\Omega_{p})}-(\mathbf{h},\mathbf{U}(t))_{\Omega_{p}}\\
		&=&\mu_{p}\|\mathbf{D}(\mathbf{U}(0))\|^{2}_{L^{2}(\Omega_{p})}
		+\frac{k_{2}}{2}\|\eta(0)\|^{2}_{L^{2}(\Omega_{p})}+\frac{k_{3}}{2}\|\xi(0)\|^{2}_{L^{2}(\Omega_{p})}-(\mathbf{h},\mathbf{U}(0))_{\Omega_{p}}\\
		&&+\int^{t}_{0}[(\mathbf{f},\mathbf{v})_{\Omega_{f}}+\langle-p_{in}(t)\mathbf{n}_{f},\mathbf{v}\rangle_{\Gamma_{f}^{in}}
		+(g,p_{f})_{\Omega_{f}}+(s,p_{p})_{\Omega_{p}}]dt,
	\end{eqnarray*}
	which implies that (\ref{2.23}) holds. The proof is complete.
\end{proof}

Taking the similar argument of Lemma \ref{lemn2}, we can get the following result, and here we omit the detail of the proof. 

\begin{lemm}\label{lemn1}
	Every weak solution~$(\mathbf{v},p_{f},\mathbf{U},\mathbf{q},p_{p})$~of the problem (\ref{2.38})-(\ref{2.42}) satisfies the following energy law
	\begin{equation}\label{2.44}
		\mathbf{E}(t)+\int_{0}^{t}[2\mu_{f}\|\mathbf{D}(\mathbf{v})\|^{2}_{L^{2}(\Omega_{f})}+k^{-1}\|\mathbf{q}\|^{2}_{L^{2}(\Omega_{p})}
		+\beta\|(\mathbf{v}-\mathbf{U}_{t})\cdot\tau\|^{2}_{L^{2}(\Gamma)}]dt-\int_{0}^{t}\mathbf{F}(t)dt=\mathbf{E}(0),
	\end{equation}
	for all~$t\in(0,T]$, where
	\begin{eqnarray}
		\mathbf{E}(t)&=&\frac{1}{2}[2\mu_{p}\|\mathbf{D}(\mathbf{U}(t))\|^{2}_{L^{2}(\Omega_{p})}+\lambda_{p}\|\nabla\cdot\mathbf{U}(t)\|^{2}_{L^{2}(\Omega_{p})}+s_{0}\|p_{p}(t)\|^{2}_{L^{2}(\Omega_{p})}-2(\mathbf{h},\mathbf{U}(t))_{\Omega_{p}}],\label{2.45} \\
		\mathbf{F}(t)&=&(\mathbf{f},\mathbf{v})_{\Omega_{f}}+\langle p_{in}(t)\mathbf{n}_{f},\mathbf{v}\rangle_{\Gamma_{f}^{in}}+(g,p_{f})_{\Omega_{f}}+(s,p_{p})_{\Omega_{p}}.\label{2.46}
	\end{eqnarray}
\end{lemm}

From Lemma \ref{lemn2}, we obtain the following estimates:
\begin{lemm}
There exists a positive constant $C_{1}$ such that 
\begin{eqnarray}
\begin{aligned}
  &\sqrt{2\mu_{f}}\|\mathbf{D}(\mathbf{v})\|_{L^{2}(0,T,L^{2}(\Omega_{f}))}+\sqrt{k^{-1}}\|\mathbf{q}\|_{L^{2}(0,T,L^{2}(\Omega_{p}))}\\
  &+\sqrt{\mu_{p}}\|\mathbf{D}(\mathbf{U})\|_{L^{\infty}(0,T,L^{2}(\Omega_{p}))}
  +\sqrt{\beta}\|(\mathbf{v}-\mathbf{U}_{t})\cdot\tau\|_{L^{2}(0,T,L^{2}(\Gamma))}\\
  &+\sqrt{\frac{k_{2}}{2}}\|\eta\|_{L^{\infty}(0,T,L^{2}(\Omega_{p}))}+\sqrt{\frac{k_{3}}{2}}\|\xi\|_{L^{\infty}(0,T,L^{2}(\Omega_{p}))}\leq C_{1},
\end{aligned}
\end{eqnarray}
where $C_1$ is a constant dependent on the initial data and source functions of $\|\mathbf{v}^{0}\|_{H^{1}(\Omega_{f})}$, $\|p_{f}^{0}\|_{L^{2}(\Omega_{f})}$, $\|p_{p}^{0}\|_{L^{2}(\Omega_{p})}$,
$\|\mathbf{U}^{0}\|_{H^{1}(\Omega_{p})}$, $ 
\|\mathbf{f}\|_{L^{2}(\Omega_{f})}$, $\|g\|_{L^{2}(\Omega_{f})}$, $\|\mathbf{h}\|_{L^{2}(\Omega_{p})}$, $\|s\|_{L^{2}(\Omega_{p})}$.
\end{lemm}

Furthermore, we have the following priori estimates:
\begin{lemm}\label{lma210801-1}
Suppose that $\mathbf{v^{0}},~\mathbf{U^{0}}$ and~$p_{p}^{0}$~are sufficiently smooth, then there exist positive constants ~$C_{2}=C_{2}(C_{1},\|\mathbf{q}^{0}\|_{L^{2}(\Omega_{p})})$~and~$C_{3}=C_{3}(C_{1},
C_{2},\|\mathbf{U}^{0}\|_{H^{2}(\Omega_{p})},\|p_{p}^{0}\|_{H^{2}(\Omega_{p})})$~such that 
\begin{eqnarray}
  &\sqrt{\mu_{f}}\|\mathbf{D}(\mathbf{v})\|_{L^{\infty}(0,T,L^{2}(\Omega_{f}))}+\sqrt{(2k)^{-1}}\|\mathbf{q}\|_{L^{\infty}(0,T,L^{2}(\Omega_{p}))}+\sqrt{2\mu_{p}}\|\mathbf{D}(\mathbf{U}_{t})\|_{L^{2}(0,T,L^{2}(\Omega_{p}))}\nonumber\\
  &
  +\sqrt{\frac{\beta}{2}}\|(\mathbf{v}-\mathbf{U}_{t})\cdot\tau\|_{L^{\infty}(0,T,L^{2}(\Gamma))}+\sqrt{k_{2}}\|\eta_{t}\|_{L^{2}(0,T,L^{2}(\Omega_{p}))}
  +\sqrt{k_{3}}\|\xi_{t}\|_{L^{2}(0,T,L^{2}(\Omega_{p}))}\leq C_{2},\label{2.32}\\
  &\sqrt{2\mu_{f}}\|\mathbf{D}(\mathbf{v}_{t})\|_{L^{2}(0,T,L^{2}(\Omega_{f}))}+\sqrt{k^{-1}}\|\mathbf{q}_{t}\|_{L^{2}(0,T,L^{2}(\Omega_{p}))}+\sqrt{\mu_{p}}\|\mathbf{D}(\mathbf{U}_{t})\|_{L^{\infty}(0,T,L^{2}(\Omega_{p}))}\nonumber\\
  &
  +\sqrt{\beta}\|(\mathbf{v}_{t}-\mathbf{U}_{tt})\cdot\tau\|_{L^{2}(0,T,L^{2}(\Gamma))}+\sqrt{\frac{k_{2}}{2}}\|\eta_{t}\|_{L^{\infty}(0,T,L^{2}(\Omega_{p}))}
  +\sqrt{\frac{k_{3}}{2}}\|\xi_{t}\|_{L^{\infty}(0,T,L^{2}(\Omega_{p}))}\leq C_{3}.\label{2.33}
\end{eqnarray}
\end{lemm}
\begin{proof}
Differentiating (\ref{2.15}), (\ref{2.17}), (\ref{2.18}) and (\ref{2.19}) with respect to $t$, taking $
  \varphi_{f}=\mathbf{v}, \mathbf{r}=\mathbf{q}, \psi_{f}=p_{f,t},  \varphi_{p}=\mathbf{U}_{t}, w=\xi_{t}, z=k_{1}\xi_{t}+k_{2}\eta_{t}$, we have
\begin{eqnarray}
& -\langle\sigma_{f,t}\mathbf{n}_{f},\mathbf{v}\rangle_{\Gamma}+\mu_{f}\frac{d}{dt}\|\mathbf{D}(\mathbf{v})\|^{2}_{L^{2}(\Omega_{f})}
=(g,p_{f,t})_{\Omega_{f}},\label{2.28}\\
 & -\langle\sigma_{p,t}\mathbf{n}_{p},\mathbf{U}_{t}\rangle_{\Gamma}+\langle p_{p,t}\mathbf{n}_{p},\mathbf{q}\rangle_{\Gamma}+2\mu_{p}\|\mathbf{D}(\mathbf{U}_{t})\|^{2}_{L^{2}(\Omega_{p})}
  +(2k)^{-1}\frac{d}{dt}\|\mathbf{q}\|^{2}_{L^{2}(\Omega_{p})}\nonumber\\
  &+k_{2}\|\eta_{t}\|^{2}_{L^{2}(\Omega_{p})}+k_{3}\|\xi_{t}\|^{2}_{L^{2}(\Omega_{p})}=(s,p_{p,t})_{\Omega_{p}}.\label{2.29}
\end{eqnarray}
Adding (\ref{2.28}) and (\ref{2.29}), we have
\begin{eqnarray}\label{2.31}
\begin{aligned}
  &\frac{\beta}{2}\frac{d}{dt}\|(\mathbf{v}-\mathbf{U}_{t})\cdot\tau\|_{L^{2}(\Gamma)}^{2}+\mu_{f}\frac{d}{dt}\|\mathbf{D}(\mathbf{v})\|^{2}_{L^{2}(\Omega_{f})}
  +(2k)^{-1}\frac{d}{dt}\|\mathbf{q}\|^{2}_{L^{2}(\Omega_{p})}\\
  +2\mu_{p}&\|\mathbf{D}(\mathbf{U}_{t})\|^{2}_{L^{2}(\Omega_{p})}+k_{2}\|\eta_{t}\|^{2}_{L^{2}(\Omega_{p})}+k_{3}\|\xi_{t}\|^{2}_{L^{2}(\Omega_{p})}
  =\frac{d}{dt}[(g,p_{f})_{\Omega_{f}}+(s,p_{p})_{\Omega_{p}}].
\end{aligned}
\end{eqnarray}
Integrating in $t$, we get 
\begin{eqnarray*}
  &&\frac{\beta}{2}\|(\mathbf{v}(t)-\mathbf{U}_{t}(t))\cdot\tau\|_{L^{2}(\Gamma)}^{2}+\mu_{f}\|\mathbf{D}(\mathbf{v}(t))\|^{2}_{L^{2}(\Omega_{f})}
  +(2k)^{-1}\|\mathbf{q}(t)\|^{2}_{L^{2}(\Omega_{p})}\\
  &&+\int_{0}^{t}[2\mu_{p}\|\mathbf{D}(\mathbf{U}_{t})\|^{2}_{L^{2}(\Omega_{p})}+k_{2}\|\eta_{t}\|^{2}_{L^{2}(\Omega_{p})}+k_{3}\|\xi_{t}\|^{2}_{L^{2}(\Omega_{p})}]dt\\
  &=&\beta\|(\mathbf{v}(0)-\mathbf{U}_{t}(0))\cdot\tau\|_{L^{2}(\Gamma)}+\mu_{f}\|\mathbf{D}(\mathbf{v}(0))\|^{2}_{L^{2}(\Omega_{f})}
  +(2k)^{-1}\|\mathbf{q}(0)\|^{2}_{L^{2}(\Omega_{p})}\\
  &&[(g,p_{f}(t)-p_{f}^{0})_{\Omega_{f}}+(s,p_{p}(t)-p_{p}^{0})_{\Omega_{p}}],
\end{eqnarray*}
which implies that (\ref{2.32}) holds.

Differentiating (\ref{2.15}), (\ref{2.16}), (\ref{2.17}),  (\ref{2.19}) and (\ref{2.20}) with respect to $t$, differentiating (\ref{2.18}) with respect $t$ two times, taking $
  \varphi_{f}=\mathbf{v}_{t}, \mathbf{r}=\mathbf{q}_{t}, \psi_{f}=p_{f,t}, \varphi_{p}=\mathbf{U}_{tt}, w=\xi_{t}, z=p_{p,t}=k_{1}\xi_{t}+k_{2}\eta_{t}$, we have
\begin{eqnarray}\label{2.37}
\begin{aligned}
  &\beta\|(\mathbf{v}_{t}-\mathbf{U}_{tt})\cdot\tau\|_{L^{2}(\Gamma)}^{2}+2\mu_{f}\|\mathbf{D}(\mathbf{v}_{t})\|^{2}_{L^{2}(\Omega_{f})}
  +k^{-1}\|\mathbf{q}_{t}\|^{2}_{L^{2}(\Omega_{p})}\\
  +&\mu_{p}\frac{d}{dt}\|\mathbf{D}(\mathbf{U}_{t})\|^{2}_{L^{2}(\Omega_{p})}+\frac{k_{2}}{2}\frac{d}{dt}\|\eta_{t}\|^{2}_{L^{2}(\Omega_{p})}
  +\frac{k_{3}}{2}\frac{d}{dt}\|\xi_{t}\|^{2}_{L^{2}(\Omega_{p})}
  =0.
\end{aligned}
\end{eqnarray}
Integrating (\ref{2.37}) in $t$, we get 
\begin{eqnarray*}
  && \int_{0}^{t}[\beta\|(\mathbf{v}_{t}-\mathbf{U}_{tt})\cdot\tau\|_{L^{2}(\Gamma)}^{2}+2\mu_{f}\|\mathbf{D}(\mathbf{v}_{t})\|^{2}_{L^{2}(\Omega_{f})}
  +k^{-1}\|\mathbf{q}_{t}\|^{2}_{L^{2}(\Omega_{p})}]dt\\
  &&+\mu_{p}\|\mathbf{D}(\mathbf{U}_{t}(t))\|^{2}_{L^{2}(\Omega_{p})}+\frac{k_{2}}{2}\|\eta_{t}(t)\|^{2}_{L^{2}(\Omega_{p})}
  +\frac{k_{3}}{2}\|\xi_{t}(t)\|^{2}_{L^{2}(\Omega_{p})}\\
  &=&\mu_{p}\|\mathbf{D}(\mathbf{U}_{t}(0))\|^{2}_{L^{2}(\Omega_{p})}+\frac{k_{2}}{2}\|\eta_{t}(0)\|^{2}_{L^{2}(\Omega_{p})}
  +\frac{k_{3}}{2}\|\xi_{t}(0)\|^{2}_{L^{2}(\Omega_{p})},
\end{eqnarray*}
which implies that (\ref{2.33}) holds. The proof is complete.
\end{proof}

%
%

\begin{lemm}\label{lemm2.5}
Every weak solution ($\mathbf{v},p_{f},\mathbf{U},\xi,\eta,\mathbf{q},p_{p}$) of the problem (\ref{2.15})-(\ref{2.20}) satisfies the following relations:
\begin{eqnarray}
  \mathbf{C}_{\mathbf{v}}(t)&:=&\langle\mathbf{v}(\cdot,t)\cdot\mathbf{n}_{f},1\rangle_{\partial\Omega_{f}}=(g,1)_{\Omega_{f}},\label{2.66}\\
  \mathbf{C}_{p_{f}}(t)&:=&(p_{f}(\cdot,t),1)_{\Omega_{f}}
  =-\langle\sigma_{f}\mathbf{n}_{f},x\rangle_{\Gamma}+2\mu_{f}\mathbf{C}_{\mathbf{v}}(t)+\langle p_{in}(t)\mathbf{n}_{f},x\rangle_{\Gamma_{f}^{in}}-(\mathbf{f},x)_{\Omega_{f}}, \\
  \mathbf{C}_{\eta}(t)&:=&(\eta(\cdot,t),1)_{\Omega_{p}}=t[(s,1)_{\Omega_{p}}-\mathbf{C}_{\mathbf{q}}]+(\eta(0),1)_{\Omega_{p}}~~~~t\geq0, \\
  \mathbf{C}_{\xi}(t)&:=& (\xi(\cdot,t),1)_{\Omega_{p}}=\frac{1}{2\mu_{p}k_{3}+1}[2\mu_{p}k_{1}\mathbf{C}_{\eta}-(\mathbf{h},x)_{\Omega_{p}}-\langle\sigma_{p}\mathbf{n}_{p},x\rangle_{\Gamma}], \\
  \mathbf{C}_{\mathbf{U}}(t)&:=&\langle\mathbf{U}(\cdot,t)\cdot\mathbf{n}_{p},1\rangle_{\partial\Omega_{p}}=k_{1}\mathbf{C}_{\eta}(t)-k_{3}\mathbf{C}_{\xi}(t), \\
  \mathbf{C}_{p_{p}}(t)&:=&(p_{p}(\cdot,t),1)_{\Omega_{p}}=k_{1}\mathbf{C}_{\eta}(t)+k_{2}\mathbf{C}_{\xi}(t), \\
  \mathbf{C}_{\mathbf{q}}(t)&:=&\langle\mathbf{q}(\cdot,t)\cdot\mathbf{n}_{p},1\rangle_{\partial\Omega_{p}}=\langle k\nabla\mathbf{C}_{p_{p}}\cdot\mathbf{n}_{p},1\rangle_{\partial\Omega_{p}}.\label{2.67}
\end{eqnarray}
\end{lemm}
\begin{proof}
Letting
\begin{gather*}
	\mathbf{C}_{\mathbf{v}}(t) := \langle\mathbf{v}(\cdot,t)\cdot\mathbf{n}_{f},1\rangle_{\partial\Omega_{f}},~\mathbf{C}_{\mathbf{U}}(t) := \langle\mathbf{U}(\cdot,t)\cdot\mathbf{n}_{p},1\rangle_{\partial\Omega_{p}},~\mathbf{C}_{\mathbf{q}}(t) := \langle\mathbf{q}(\cdot,t)\cdot\mathbf{n}_{p},1\rangle_{\partial\Omega_{p}}, \\
	\mathbf{C}_{p_{f}}(t) := (p_{f}(\cdot,t),1)_{\Omega_{f}},\mathbf{C}_{\eta}(t) := (\eta(\cdot,t),1)_{\Omega_{p}},\mathbf{C}_{\xi}(t) := (\xi(\cdot,t),1)_{\Omega_{p}},
	\mathbf{C}_{p_{p}}(t) := (p_{p}(\cdot,t),1)_{\Omega_{p}},
\end{gather*}
and taking $\varphi_{f}=x, \psi_{f}=1, \varphi_{p}=x, z=1, w=1$ in (\ref{2.15})-(\ref{2.20}), 
we have
\begin{eqnarray}
  -\langle\sigma_{f}\mathbf{n}_{f},x\rangle_{\Gamma}+2\mu_{f}\mathbf{C}_{\mathbf{v}}(t)-\mathbf{C}_{p_{f}}(t)+\langle p_{in}(t)\mathbf{n}_{f},x\rangle_{\Gamma_{f}^{in}}
    &=&(\mathbf{f},x)_{\Omega_{f}}, \label{2.60}\\
  \mathbf{C}_{\mathbf{v}}(t)
   &=&(g,1)_{\Omega_{f}},\label{2.61}\\
   -\langle\sigma_{p}\mathbf{n}_{p},x\rangle_{\Gamma}
  +2\mu_{p}\mathbf{C}_{\mathbf{U}}(t)-\mathbf{C}_{\xi}(t)
   &=&(\mathbf{h},x)_{\Omega_{p}},\label{2.62}\\
  k_{3}\mathbf{C}_{\xi}(t)+\mathbf{C}_{\mathbf{U}}(t)&=& k_{1}\mathbf{C}_{\eta}(t),\label{2.63}\\
  \frac{1}{t}[\mathbf{C}_{\eta}(t)-(\eta(0),1)]+\mathbf{C}_{\mathbf{q}}(t)
  &=&(s,1)_{\Omega_{p}},\label{2.64}
\end{eqnarray}
which imply that (\ref{2.66})-(\ref{2.67}) hold. The proof is complete.
\end{proof}

With the help of Lemma \ref{lemn2}, Lemma \ref{lemn1} and Lemma \ref{lemm2.5}, we obtain the following main result.

\begin{thrm}\label{thm200527-1}
Suppose that $\mathbf{v}^{0}\in H^{1}(\Omega_{f}),~\mathbf{U}^{0}\in H^{1}(\Omega_{p}),~p_{p}^{0}\in L^{2}(\Omega_{p}),~\mathbf{f},g\in L^{2}(\Omega_{f}),~\mathbf{h},s\in L^{2}(\Omega_{p})$, for any $t\in(0, T ]$, then there exists a unique weak solution to the problem (\ref{1.1'})-(\ref{1.7'}), likewise, there exists a unique weak solution to the problem (\ref{1.21})-(\ref{1.26}).
\end{thrm}
\begin{proof}
Note that the energy laws established in Lemma \ref{lemn2}  and Lemma \ref{lemn1} guarantee the required uniform estimates, so it is standard to prove the existence of weak solution by the standard Galerkin method and compactness argument (cf. \cite{R.Temam}), here we omit the details. 

Suppose that there are two group of  weak solutions: $\mathbf{v}_{1},~p_{f,1},~\mathbf{U}_{1},~p_{p,1},~\mathbf{q}_{1}$ and $\mathbf{v}_{2},~p_{f,2},~\mathbf{U}_{2},~p_{p,2},~\mathbf{q}_{2}$. Let $\mathbf{v}=\mathbf{v}_{1}-\mathbf{v}_{2},~p_{f}=p_{f,1}-p_{f,2},~\mathbf{U}=\mathbf{U}_{1}-\mathbf{U}_{2},~p_{p}=p_{p,1}-p_{p,2},~\mathbf{q}=\mathbf{q}_{1}-\mathbf{q}_{2}$.
It suffices to show that $\mathbf{v}\equiv0,~p_{f}\equiv0,~\mathbf{U}\equiv0,~p_{p}\equiv0,~\mathbf{q}\equiv0$. By the linearity of equations, we immediately imply that $\mathbf{v},~p_{f},~\mathbf{U},~p_{p},~\mathbf{q}$~ satisfy (\ref{1.1'})-(\ref{1.7'}) with ~$\mathbf{f}=\mathbf{h}=g=s=p_{in}(t)=0$~and~$\mathbf{v}^{0}=0,~\mathbf{U}^{0}=0,~p_{p}^{0}=0$,~thus~$\mathbf{F}(t)=0,~\mathbf{E}(0)=0$.

Using (\ref{2.44}), we have
\begin{align*}
  \mu_{p}\|\mathbf{D}(\mathbf{U}(t))\|^{2}_{L^{2}(\Omega_{p})}&+\frac{\lambda_{p}}{2}\|\nabla\cdot\mathbf{U}(t)\|^{2}_{L^{2}(\Omega_{p})}
  +\frac{s_{0}}{2}\|p_{p}(t)\|^{2}_{L^{2}(\Omega_{p})}\\+\int_{0}^{t}[2\mu_{f}\|\mathbf{D}(\mathbf{v})\|^{2}_{L^{2}(\Omega_{f})}&+k^{-1}\|\mathbf{q}\|^{2}_{L^{2}(\Omega_{p})}
  +\beta\|(\mathbf{v}-\mathbf{U}_{t})\cdot\tau\|^{2}_{L^{2}(\Gamma)}]dt=0,
\end{align*}
which implies that $\mathbf{D}(\mathbf{U})=0,~\mathbf{q}=0,~p_{p}=0,~\mathbf{D}(\mathbf{v})=0$.

Therefore, we see that $\mathbf{U}=0,~\mathbf{v}=0,~p_{f}=0$ by using Lemma \ref{lemm2.5} and ~$\mathbf{U}\in H^{1}(\Omega_{p}),~\mathbf{v}\in H^{1}(\Omega_{f})$. The proof is complete.
\end{proof}

\section{Fully discrete multiphysics mixed finite element method with Nitsche's technique}\label{sec3}
\subsection{Multiphysics mixed finite element method}
Let $\mathcal{T}_{h}^{f}$ and $\mathcal{T}_{h}^{p}$ be fixed, quasi-uniform meshes defined on the domains $\Omega_{f}$ and $\Omega_{p}$ with the maximum mesh size $h$. We require that $\Omega_{f}$ and $\Omega_{p}$ are polygonal or polyhedral domains and that they conform at the interface~$\Gamma$.
We denote $\mathbf{V}_{h}^{f}\subset\mathbf{V}^{f},~Q_{h}^{f}\subset Q^{f}$ by the finite element spaces
for the velocity and pressure approximation on the fluid domain $\Omega_{f}$, $\mathbf{V}_{h}^{p}\subset\mathbf{V}^{p}, M_{h}^{p}\subset M^{p}, Q_{h}^{p}\subset Q^{p}$ by the spaces for velocity and pressure approximation on the porous matrix $\Omega_{p}$ and $\mathbf{X}_{h}^{p}\subset\mathbf{X}^{p}$ by the approximation spaces for the structure displacement, respectively.

Using the Nitsche's technique to deal with the interface conditions ( cf. \cite{P.Hansbo}), we get
\begin{eqnarray*}
  &-\mathbf{I}_{\Gamma}^{*}
  =-\int_{\Gamma}(\mathbf{n}\cdot\sigma_{f,h}(\mathbf{v}_{h},p_{f,h})\mathbf{n}(\varphi_{f,h}-\varphi_{p,h}-\mathbf{r}_{h})\cdot\mathbf{n}+
  \tau\cdot\sigma_{f,h}(\mathbf{v}_{h},p_{f,h})\mathbf{n}(\varphi_{f,h}-\varphi_{p,h})\cdot\tau)~dx\\
  &+\int_{\Gamma}\gamma_{f}\mu_{f}h^{-1}[(\mathbf{v}_{h}-\partial_{t}\mathbf{U}_{h}-\mathbf{q}_{h})\cdot\mathbf{n}(\varphi_{f,h}-\varphi_{p,h}- \mathbf{r}_{h})\cdot\mathbf{n}+(\mathbf{v}_{h}-\partial_{t}\mathbf{U}_{h})\cdot\tau(\varphi_{f,h}-\varphi_{p,h})\cdot\tau]~dx,
\end{eqnarray*}
where $\gamma_{f}>0$ is a penalty parameter. 

Furthermore, we introduce
\begin{eqnarray*}
  -\mathbf{S}_{\Gamma}^{*,\varsigma}
  =-\int_{\Gamma}(\mathbf{n}\cdot\sigma_{f,h}(\varsigma\varphi_{f,h},-\psi_{f,h})\mathbf{n}(\mathbf{v}_{h}-\partial_{t}\mathbf{U}_{h}-\mathbf{q}_{h})\cdot\mathbf{n}+
  \tau\cdot\sigma_{f,h}(\varsigma\varphi_{f,h},-\psi_{f,h})\mathbf{n}(\mathbf{v}_{h}-\partial_{t}\mathbf{U}_{h})\cdot\tau)~dx,
\end{eqnarray*}
where $\varsigma$ can be chosen $1$, $0$ or $-1$. 

Also, we can accommodate the weak enforcement of the Beavers¨CJoseph¨CSaffman condition (\ref{1.15}) by
\begin{eqnarray*}
  -\mathbf{I}_{\Gamma}^{+}-\mathbf{S}_{\Gamma}^{+,\varsigma}
  &=&-\int_{\Gamma}\mathbf{n}\cdot\sigma_{f,h}(\mathbf{v}_{h},p_{f,h})\mathbf{n}(\varphi_{f,h}-\varphi_{p,h}-\mathbf{r}_{h})\cdot\mathbf{n}~dx\\
  &&-\int_{\Gamma}\mathbf{n}\cdot\sigma_{f,h}(\varsigma\varphi_{f,h},-\psi_{f,h})\mathbf{n}(\mathbf{v}_{h}-\partial_{t}\mathbf{U}_{h}-\mathbf{q}_{h})\cdot\mathbf{n}~dx\\
  &&+\int_{\Gamma}\gamma_{f}\mu_{f}h^{-1}(\mathbf{v}_{h}-\partial_{t}\mathbf{U}_{h}-\mathbf{q}_{h})\cdot\mathbf{n}(\varphi_{f,h}-\varphi_{p,h}- \mathbf{r}_{h})\cdot\mathbf{n}~dx\\
  &&+\int_{\Gamma}\beta(\mathbf{v}_{h}-\partial_{t}\mathbf{U}_{h})\cdot\tau(\varphi_{f,h}-\varphi_{p,h})\cdot\tau~dx.
\end{eqnarray*}

\begin{rem}\label{rem210802-1}
Comparing $\mathbf{I}_{\Gamma}^{+}+\mathbf{S}_{\Gamma}^{+,\varsigma}$ with $\mathbf{I}_{\Gamma}^{*}+\mathbf{S}_{\Gamma}^{*,\varsigma}$, we find that the operators corresponding to (\ref{1.15}) can be seen as a particular form of the more general case that is obtained when no-slip conditions (\ref{1.14}) are enforced weakly, namely $\mathbf{I}_{\Gamma}^{*}+\mathbf{S}_{\Gamma}^{*,\varsigma}$. For this reason, we perform the analysis of the numerical scheme in the latter form.
\end{rem}

The semi-discrete scheme of the problem (\ref{1.21})-(\ref{1.26}) is: find $\mathbf{v}_{h}\in L^{\infty}(0,T;\mathbf{V}^{f}_{h})$, $p_{f,h}\in L^{\infty}(0,T;Q^{f}_{h})$, $\mathbf{U}_{h}\in L^{\infty}(0,T;\mathbf{X}^{p}_{h}), \mathbf{q}_{h}\in L^{\infty}(0,T;\mathbf{V}^{p}_{h}), \xi_{h}\in L^{\infty}(0,T;M^{p}_{h}), \eta_{h}\in L^{\infty}(0,T;Q^{p}_{h})$ such that 
\begin{eqnarray}\label{2.7}
\begin{aligned}
  &(2\mu_{f}\mathbf{D}(\mathbf{v}_{h}),\mathbf{D}(\varphi_{f,h}))_{\Omega_{f}}-(p_{f,h},\nabla\cdot\varphi_{f,h})_{\Omega_{f}}
  +(\nabla\cdot\mathbf{v}_{h},\psi_{f,h})_{\Omega_{f}}\\
  &+2\mu_{p}(\mathbf{D}(\mathbf{U}_{h}),\mathbf{D}(\varphi_{p,h}))_{\Omega_{p}}
  -(\xi_{h},\nabla\cdot\varphi_{p,h})_{\Omega_{p}}+(\nabla\cdot\mathbf{U}_{h},w_{h})_{\Omega_{p}}\\
  &+k_{3}(\xi_{h},w_{h})_{\Omega_{p}}-k_{1}(\eta_{h},w_{h})_{\Omega_{p}}
  +k^{-1}(\mathbf{q}_{h},\mathbf{r}_{h})_{\Omega_{p}}
   -((k_{1}\xi_{h}+k_{2}\eta_{h}),\nabla\cdot \mathbf{r}_{h})_{\Omega_{p}}\\
   &+(\partial_{t}\eta_{h},z_{h})_{\Omega_{p}}+(\nabla\cdot\mathbf{q}_{h},z_{h})_{\Omega_{p}}-(\mathbf{I}_{\Gamma}^{*}+\mathbf{S}_{\Gamma}^{*,\varsigma})
  =\mathcal{F}(t;\varphi_{f,h},\psi_{f,h},\varphi_{p,h},z_{h}),
\end{aligned}
\end{eqnarray}
for any $(\varphi_{f,h}, \psi_{f,h}, \mathbf{r}_{h},w_{h},  z_{h}, \varphi_{p,h})\in\mathbf{V}_{h}^{f}\times Q_{h}^{f}\times\mathbf{V}_{h}^{p}\times M_{h}^{p}\times
Q_{h}^{p}\times\mathbf{X}_{h}^{p}$, where~$\mathcal{F}(t;\varphi_{f,h},\psi_{f,h},\varphi_{p,h},z_{h})=
(\mathbf{f},\varphi_{f,h})_{\Omega_{f}}+\langle-p_{in}(t)\mathbf{n}_{f},\varphi_{f,h}\rangle_{\Gamma_{f}^{in}}
+(g,\psi_{f,h})_{\Omega_{f}}+(\mathbf{h},\varphi_{p,h})_{\Omega_{p}}+(s,z_{h})_{\Omega_{p}}$.

%
%

Let $\Delta t$ denote the time step,~$t_{n}=n\Delta t,~0\leq n\leq N$, and define the
(backward) discrete time derivative by
$d_{t}\mathbf{u}^{n}:=\frac{\mathbf{u}^{n}-\mathbf{u}^{n-1}}{\Delta t}$. Then, the fully multiphysics mixed finite element method for the problem (\ref{1.21})-(\ref{1.26}) is: find $(\mathbf{v}_{h}^{n},p_{f,h}^{n},\mathbf{q}_{h}^{n},\xi_{h}^{n},\eta_{h}^{n},\mathbf{U}_{h}^{n})\in\mathbf{V}_{h}^{f}\times Q_{h}^{f}\times\mathbf{V}_{h}^{p}\times M_{h}^{p}\times
Q_{h}^{p}\times\mathbf{X}_{h}^{p}$ such that 
\begin{eqnarray}\label{2.9}
\begin{aligned}
  &(2\mu_{f}\mathbf{D}(\mathbf{v}_{h}^{n}),\mathbf{D}(\varphi_{f,h}))_{\Omega_{f}}-(p_{f,h}^{n},\nabla\cdot\varphi_{f,h})_{\Omega_{f}}
  +(\nabla\cdot\mathbf{v}_{h}^{n},\psi_{f,h})_{\Omega_{f}}\\
  &+2\mu_{p}(\mathbf{D}(\mathbf{U}_{h}^{n}),\mathbf{D}(\varphi_{p,h}))_{\Omega_{p}}
  -(\xi_{h}^{n},\nabla\cdot\varphi_{p,h})_{\Omega_{p}}+(\nabla\cdot\mathbf{U}_{h}^{n},w_{h})_{\Omega_{p}}\\
  &+k_{3}(\xi_{h}^{n},w_{h})_{\Omega_{p}}-k_{1}(\eta_{h}^{n},w_{h})_{\Omega_{p}}
  +k^{-1}(\mathbf{q}_{h}^{n},\mathbf{r}_{h})_{\Omega_{p}}\\
  & -((k_{1}\xi_{h}^{n}+k_{2}\eta_{h}^{n}),\nabla\cdot \mathbf{r}_{h})_{\Omega_{p}}+(d_{t}\eta_{h}^{n},z_{h})_{\Omega_{p}}+(\nabla\cdot\mathbf{q}_{h}^{n},z_{h})_{\Omega_{p}}-(\mathbf{I}_{\Gamma}^{*}+\mathbf{S}_{\Gamma}^{*,\varsigma})\\
  =& \mathcal{F}(t_{n};\varphi_{f,h},\psi_{f,h},\varphi_{p,h},z_{h}),
\end{aligned}
\end{eqnarray}
for any $(\varphi_{f,h},\psi_{f,h},\mathbf{r}_{h},w_{h},z_{h},\varphi_{p,h})\in\mathbf{V}_{h}^{f}\times Q_{h}^{f}\times\mathbf{V}_{h}^{p}\times M_{h}^{p}\times
Q_{h}^{p}\times\mathbf{X}_{h}^{p}$.

We  also recall the following inverse inequality
\begin{equation}\label{2.11}
  h\|\mathbf{D}(\mathbf{u}_{h})\mathbf{n}\|^{2}_{L^{2}(\Gamma)}\leq C_{TI}\|\mathbf{D}(\mathbf{u}_{h})\|^{2}_{L^{2}(\Omega)},
\end{equation}
where $\mathcal{C}_{TI}$ is a positive constant uniformly upper bounded with respect to the mesh size $h$.

Define
\begin{equation*}
	\mathbf{E}^{n}_{p,h}:=\frac{1}{2}(2\mu_{p}\|\mathbf{D}(\mathbf{U}_{h}^{n})\|^{2}_{L^{2}(\Omega_{p})}+k_{3}\| \xi_{h}^{n}\parallel^{2}_{L^{2}(\Omega_{p})}+k_{2}\| \eta_{h}^{n}\|^{2}_{L^{2}(\Omega_{p})}).
\end{equation*}

\begin{thrm}\label{thm2.1}
For any~$\overline{\epsilon}^{1}_{f},\underline{\epsilon}^{1}_{f}$~ satisfying
\begin{equation*}
  1-(\varsigma+1)\overline{\epsilon}^{1}_{f}C_{TI}-\frac{\underline{\epsilon}^{1}_{f}}{4}>0
\end{equation*}
 provided that~$\gamma_{f}>(\varsigma+1)(\overline{\epsilon}^{1}_{f})^{-1}$,  then there exist constants $0 < c < 1$~ and ~$C > 1$ (uniformly independent of the mesh characteristic size $h$) such that
\begin{eqnarray}\label{2.12}
\begin{aligned}
   & \mathbf{E}^{N}_{p,h}+c\Delta t\sum^{N}_{n=1}[2\mu_{f}\|\mathbf{D}(\mathbf{v}_{h}^{n})\|^{2}_{L^{2}(\Omega_{f})}+k^{-1}\|\mathbf{q}_{h}^{n}\|^{2}_{L^{2}(\Omega_{p})} \\
   & +\frac{\Delta t}{2}(2\mu_{p}\|d_{t}\mathbf{D}(\mathbf{U}_{h}^{n})\|^{2}_{L^{2}(\Omega_{p})}+k_{3}\| d_{t}\xi_{h}^{n}\|^{2}_{L^{2}(\Omega_{p})}+k_{2}\| d_{t}\eta_{h}^{n}\|^{2}_{L^{2}(\Omega_{p})}) \\
   & +\mu_{f}h^{-1}(\|(\mathbf{v}_{h}^{n}-\mathbf{q}_{h}^{n}-d_{t}\mathbf{U}_{h}^{n})\cdot\mathbf{n}\|^{2}_{L^{2}(\Gamma)}
   +\|(\mathbf{v}_{h}^{n}-d_{t}\mathbf{U}_{h}^{n})\cdot\tau\|^{2}_{L^{2}(\Gamma)})] \\
   &\leq  \mathbf{E}^{0}_{p,h}+\Delta t\sum^{N}_{n=1}\frac{C}{\mu_{f}}\|\mathcal{F}(t_{n})\|^{2}
\end{aligned}
\end{eqnarray}
with $c < \min\{(1-(\varsigma+1)\overline{\epsilon}^{1}_{f}C_{TI}-\frac{\underline{\epsilon}^{1}_{f}}{4}),(\gamma_{f}-(\varsigma+1)(\overline{\epsilon}^{1}_{f})^{-1})\}$ and $C > (2\underline{\epsilon}^{1}_{f})^{-1}$.
\end{thrm}
\begin{proof}
Taking $\varphi_{f,h}=\mathbf{v}_{h}^{n}, \mathbf{r}_{h}=\mathbf{q}_{h}^{n}, \psi_{f,h}=p_{f,h}^{n}, \varphi_{p,h}=d_{t}\mathbf{U}_{h}^{n}, w_{h}=d_{t}\xi_{h}^{n}, z_{h}=k_{1}\xi_{h}^{n}+k_{2}\eta_{h}^{n}$ in (\ref{2.9}), we have
\begin{eqnarray*}
   &&   (2\mu_{f}\mathbf{D}(\mathbf{v}_{h}^{n}),\mathbf{D}(\mathbf{v}_{h}^{n}))_{\Omega_{f}}-(p_{f,h}^{n},\nabla\cdot\mathbf{v}_{h}^{n}))_{\Omega_{f}}
  +(\nabla\cdot\mathbf{v}_{h}^{n},p_{f,h}^{n})_{\Omega_{f}}\\
  &&+k^{-1}(\mathbf{q}_{h}^{n},\mathbf{q}_{h}^{n})_{\Omega_{p}}-((k_{1}\xi_{h}^{n}+k_{2}\eta_{h}^{n}),\nabla\cdot \mathbf{q}_{h}^{n})_{\Omega_{p}}+(\nabla\cdot\mathbf{q}_{h}^{n},(k_{1}\xi_{h}^{n}+k_{2}\eta_{h}^{n}))_{\Omega_{p}}\\
  &&\quad= 2\mu_{f}\|\mathbf{D}(\mathbf{v}_{h}^{n})\|^{2}_{L^{2}(\Omega_{f})}+k^{-1}\|\mathbf{q}_{h}^{n}\|^{2}_{L^{2}(\Omega_{p})}.
\end{eqnarray*}
The interface terms for the coupling between the fluid and the structure can be bounded as follows
\begin{eqnarray}\label{eq210802-2}
\begin{aligned}
   &-\int_{\Gamma}\mathbf{n}\cdot\sigma_{f,h}(\mathbf{v}_{h}^{n},p_{f,h}^{n})\mathbf{n}(\mathbf{v}_{h}^{n}-\mathbf{q}_{h}^{n}-d_{t}\mathbf{U}_{h}^{n})\cdot\mathbf{n}~dx -\int_{\Gamma}\mathbf{n}\cdot\sigma_{f,h}(\varsigma\mathbf{v}_{h}^{n},-p_{f,h}^{n})\mathbf{n}(\mathbf{v}_{h}^{n}-d_{t}\mathbf{U}_{h}^{n}-\mathbf{q}_{h}^{n})\cdot\mathbf{n}~dx\\
   &=-(1+\varsigma)\int_{\Gamma}\mathbf{n}\cdot(2\mu_{f}\mathbf{D}(\mathbf{v}_{h}^{n}))\mathbf{n}(\mathbf{v}_{h}^{n}-d_{t}\mathbf{U}_{h}^{n}-\mathbf{q}_{h}^{n})\cdot\mathbf{n}~dx\\
   &\leq 2\mu_{f}(1+\varsigma)\|\mathbf{D}(\mathbf{v}_{h}^{n})\mathbf{n}\|_{L^{2}(\Gamma)}
   \|(\mathbf{v}_{h}^{n}-d_{t}\mathbf{U}_{h}^{n}-\mathbf{q}_{h}^{n})\cdot\mathbf{n}\|_{L^{2}(\Gamma)} \\
   &\leq \mu_{f}(1+\varsigma)\overline{\epsilon}^{1}_{f}C_{TI}\|\mathbf{D}(\mathbf{v}_{h}^{n})\|_{L^{2}(\Omega_{f})}^{2}+\mu_{f}(1+\varsigma)(\overline{\epsilon}^{1}_{f}h)^{-1}
   \|(\mathbf{v}_{h}^{n}-d_{t}\mathbf{U}_{h}^{n}-\mathbf{q}_{h}^{n})\cdot\mathbf{n}\|_{L^{2}(\Gamma)}^{2}.
\end{aligned}
\end{eqnarray}

Also, for the structure problem, we have
\begin{eqnarray*}
  && 2\mu_{p}(\mathbf{D}(\mathbf{U}_{h}^{n}),d_{t}\mathbf{D}(\mathbf{U}_{h}^{n}))_{\Omega_{p}}-(\xi_{h}^{n},d_{t}\nabla\cdot\mathbf{U}_{h}^{n})_{\Omega_{p}}
  +(d_{t}\eta_{h}^{n},k_{1}\xi_{h}^{n}+k_{2}\eta_{h}^{n})_{\Omega_{p}} \\
  &&+(\nabla\cdot\mathbf{U}_{h}^{n},d_{t}\xi_{h}^{n})_{\Omega_{p}}+k_{3}(\xi_{h}^{n},d_{t}\xi_{h}^{n})_{\Omega_{p}}-k_{1}(\eta_{h}^{n},d_{t}\xi_{h}^{n})_{\Omega_{p}}\\
  &&= 2\mu_{p}(\mathbf{D}(\mathbf{U}_{h}^{n}),d_{t}\mathbf{D}(\mathbf{U}_{h}^{n}))_{\Omega_{p}}-(\xi_{h}^{n},d_{t}(k_{1}\eta_{h}^{n}-k_{3}\xi_{h}^{n}))_{\Omega_{p}}
  +(d_{t}\eta_{h}^{n},k_{1}\xi_{h}^{n}+k_{2}\eta_{h}^{n})_{\Omega_{p}} \\
  &&= d_{t}\mathbf{E}_{p,h}^{n}+\frac{\Delta t}{2}(2\mu_{p}\|d_{t}\mathbf{D}(\mathbf{U}_{h}^{n})\|^{2}_{L^{2}(\Omega_{p})}+k_{3}\| d_{t}\xi_{h}^{n}\|^{2}_{L^{2}(\Omega_{p})}+k_{2}\|d_{t}\eta_{h}^{n}\|^{2}_{L^{2}(\Omega_{p})}).
\end{eqnarray*}

Summing with respect to the time index~$n = 1,\cdots,N$~and multiplying by~$\Delta t$, we get
\begin{eqnarray}\label{2.13}
\begin{aligned}
   & \mathbf{E}^{N}_{p,h}+\Delta t\sum^{N}_{n=1}[2\mu_{f}(1-(\varsigma+1)\overline{\epsilon}^{1}_{f}\mathcal{C}_{TI}-\frac{\underline{\epsilon}^{1}_{f}}{4})\|\mathbf{D}(\mathbf{v}_{h}^{n})\|^{2}_{L^{2}(\Omega_{f})}
   +k^{-1}\|\mathbf{q}_{h}^{n}\|^{2}_{L^{2}(\Omega_{p})} \\
   & +\frac{\Delta t}{2}(2\mu_{p}\|d_{t}\mathbf{D}(\mathbf{U}_{h}^{n})\|^{2}_{L^{2}(\Omega_{p})}+k_{3}\| d_{t}\xi_{h}^{n}\|^{2}_{L^{2}(\Omega_{p})}+k_{2}\| d_{t}\eta_{h}^{n}\|^{2}_{L^{2}(\Omega_{p})}) \\
   & +(\gamma_{f}-(\varsigma+1)(\overline{\epsilon}^{1}_{f})^{-1}) \mu_{f}h^{-1}(\|(\mathbf{v}_{h}^{n}-\mathbf{q}_{h}^{n}-d_{t}\mathbf{U}_{h}^{n})\cdot\mathbf{n}\|^{2}_{L^{2}(\Gamma)}\\
   &+\|(\mathbf{v}_{h}^{n}-d_{t}\mathbf{U}_{h}^{n})\cdot\tau\|^{2}_{L^{2}(\Gamma)})] \\
   &\leq  \mathbf{E}^{0}_{p,h}+\Delta t\sum^{N}_{n=1}(2\underline{\epsilon}^{1}_{f}\mu_{f})^{-1}\|\mathcal{F}(t_{n})\|^{2},
\end{aligned}
\end{eqnarray}
which implies that (\ref{2.12}) holds. The proof is complete.
\end{proof}


Next, we design a loosely-coupled time-stepping method
to decouple the problem (\ref{2.9}) into three sub-problems at each time step. The loosely-coupled time-stepping algorithm is:

Step 1: Given $\mathbf{v}_{h}^{n-1},p_{f,h}^{n-1},\eta_{h}^{n-1},\mathbf{q}_{h}^{n-1}\in\mathbf{V}_{h}^{f}\times Q_{h}^{f}\times Q_{h}^{p}\times\mathbf{V}_{h}^{p}$, solve $\mathbf{U}_{h}^{n},\xi_{h}^{n}\in\mathbf{X}_{h}^{p}\times M_{h}^{p}$ such that
\begin{eqnarray}\label{3.5}
  &&2\mu_{p}(\mathbf{D}(\mathbf{U}_{h}^{n}),\mathbf{D}(\varphi_{p,h}))_{\Omega_{p}}-(\xi_{h}^{n},\nabla\cdot\varphi_{p,h})_{\Omega_{p}}+(\nabla\cdot\mathbf{U}_{h}^{n},
  w_{h}) _{\Omega_{p}}+k_{3}(\xi_{h}^{n},w_{h})_{\Omega_{p}}\nonumber\\
  &&+\int_{\Gamma}\gamma_{f}\mu_{f}h^{-1}d_{t}\mathbf{U}_{h}^{n}\cdot\tau_{p}\varphi_{p,h}\cdot\tau_{p}~dx
  +\int_{\Gamma}\gamma_{f}\mu_{f}h^{-1}d_{t}\mathbf{U}_{h}^{n}\cdot\mathbf{n}_{p}\varphi_{p,h}\cdot\mathbf{n}_{p}~dx  \nonumber\\
  &=&-\int_{\Gamma}\mathbf{n}_{p}\cdot\sigma_{f,h}^{n-1}\mathbf{n}_{p}(-\varphi_{p,h})\cdot\mathbf{n}_{p}~dx
  -\int_{\Gamma}\tau_{p}\cdot\sigma_{f,h}^{n-1}\mathbf{n}_{p}(-\varphi_{p,h})\cdot\tau_{p}~dx\\
  &&+\int_{\Gamma}\gamma_{f}\mu_{f}h^{-1}\mathbf{v}_{h}^{n-1}\cdot\tau_{p}\varphi_{p,h}\cdot\tau_{p}~dx+\int_{\Gamma}\gamma_{f}\mu_{f}h^{-1}(\mathbf{v}_{h}^{n-1}-\mathbf{q}_{h}^{n-1})
  \cdot\mathbf{n}_{p}\varphi_{p,h}\cdot\mathbf{n}_{p}~dx\nonumber\\
  &&+(\mathbf{h},\varphi_{p,h})_{\Omega_{p}}+k_{1}(\eta_{h}^{n-1},w_{h})_{\Omega_{p}}.\nonumber
\end{eqnarray}

Step 2: Given~$\mathbf{v}_{h}^{n-1},p_{f,h}^{n-1},\xi_{h}^{n},\mathbf{U}_{h}^{n-1}\in\mathbf{V}_{h}^{f}\times Q_{h}^{f}\times M_{h}^{p}\times\mathbf{X}_{h}^{p}$, solve $\mathbf{q}_{h}^{n},\eta_{h}^{n}
\in\mathbf{V}_{h}^{p}\times Q_{h}^{p}$~such that
\begin{eqnarray}\label{3.6}
  &&-(k_{2}\eta_{h}^{n},\nabla\cdot \mathbf{r}_{h})_{\Omega_{p}}+k^{-1}(\mathbf{q}_{h}^{n},\mathbf{r}_{h})_{\Omega_{p}}+(d_{t}\eta_{h}^{n},z_{h})_{\Omega_{p}}+(\nabla\cdot\mathbf{q}_{h}^{n},z_{h})_{\Omega_{p}}\nonumber\\
  &&+\int_{\Gamma}\gamma_{f}\mu_{f}h^{-1}\mathbf{q}_{h}^{n}\cdot\mathbf{n}_{p}\mathbf{r}_{h}\cdot\mathbf{n}_{p}~dx
  +s_{f,q}(d_{t}\mathbf{q}_{h}\cdot\mathbf{n}_{p},\mathbf{r}_{h}\cdot\mathbf{n}_{p})  \nonumber\\
  &=&(k_{1}\xi_{h}^{n},\nabla\cdot \mathbf{r}_{h})_{\Omega_{p}}+\int_{\Gamma}\gamma_{f}\mu_{f}h^{-1}(\mathbf{v}_{h}^{n-1}-d_{t}\mathbf{U}_{h}^{n-1})\cdot\mathbf{n}_{p}\mathbf{r}_{h}\cdot\mathbf{n}_{p}~dx\\
  &&+\int_{\Gamma}\mathbf{n}_{p}\cdot\sigma_{f,h}^{n-1}\mathbf{n}_{p}\mathbf{r}_{h}\cdot\mathbf{n}_{p}~dx+(s,z_{h})_{\Omega_{p}}.\nonumber
\end{eqnarray}

Step 3: Solve $\mathbf{v}_{h}^{n},p_{f,h}^{n}\in\mathbf{V}_{h}^{f}\times Q_{h}^{f}$~such that
\begin{eqnarray}
  &&(2\mu_{f}\mathbf{D}(\mathbf{v}_{h}^{n}),\mathbf{D}(\varphi_{f,h}))_{\Omega_{f}}-(p_{f,h}^{n},\nabla\cdot\varphi_{f,h})_{\Omega_{f}}
  +(\nabla\cdot\mathbf{v}_{h}^{n},\psi_{f,h})_{\Omega_{f}}\nonumber\\
  &&+s_{f,p}(d_{t}p_{f,h},\psi_{f,h})
  +s_{f,v}(d_{t}\mathbf{v}_{h}^{n}\cdot\mathbf{n}_{f},\varphi_{f,h}\cdot\mathbf{n}_{f})\nonumber\\
  &&-\int_{\Gamma}\sigma_{f,h}(\varsigma\varphi_{f,h},-\psi_{f,h})\mathbf{n}_{f}\cdot\mathbf{v}_{h}^{n}~dx
  +\int_{\Gamma}\gamma_{f}\mu_{f}h^{-1}\mathbf{v}_{h}^{n}\cdot\varphi_{f,h}~dx\nonumber\\
  &=& \int_{\Gamma}\sigma_{f,h}^{n-1}\mathbf{n}_{f}\cdot\varphi_{f,h}~dx
  -\int_{\Gamma}\sigma_{f,h}(\varsigma\varphi_{f,h},-\psi_{p,h})\mathbf{n}_{f}d_{t}\mathbf{U}_{h}^{n}~dx\nonumber\\
  &&-\int_{\Gamma}\mathbf{n}_{f}\cdot\sigma_{f,h}(\varsigma\varphi_{f,h},-\psi_{f,h})\mathbf{n}_{f}\mathbf{q}_{h}^{n}\cdot\mathbf{n}_{f}~dx\nonumber\\
  &&+\int_{\Gamma}\gamma_{f}\mu_{f}h^{-1}d_{t}\mathbf{U}_{h}^{n}\cdot\varphi_{f,h}~dx
  +\int_{\Gamma}\gamma_{f}\mu_{f}h^{-1}\mathbf{q}_{h}^{n}\cdot\mathbf{n}_{f}\varphi_{f,h}\cdot\mathbf{n}_{f}~dx
   \nonumber\\
  &&+(\mathbf{f},\varphi_{f,h})_{\Omega_{f}}+\langle-p_{in}(t)\mathbf{n}_{f},\varphi_{f,h}\rangle_{\Gamma_{f}^{in}}+(g,\psi_{f,h})_{\Omega_{f}}.\label{3.7}
\end{eqnarray}

\begin{rem}\label{rem200527-1}
%
%
%

The stabilization terms $s_{f,v}, s_{f,q}$ (cf. \cite{R.Zakerzadeh}) can be chosen by
\begin{gather*}
  s_{f,q}(d_{t}\mathbf{q}_{h}\cdot\mathbf{n},\mathbf{r}_{h}\cdot\mathbf{n}):=\gamma'_{stab}\gamma_{f}\mu_{f}\frac{\Delta t}{h}\int_{\Gamma}d_{t}\mathbf{q}_{h}^{n}\cdot\mathbf{n}\mathbf{r}_{h}\cdot\mathbf{n}~dx,\\
  s_{f,v}(d_{t}\mathbf{v}_{h}^{n}\cdot\mathbf{n},\varphi_{f,h}\cdot\mathbf{n}):=\gamma'_{stab}\gamma_{f}\mu_{f}\frac{\Delta t}{h}\int_{\Gamma}d_{t}\mathbf{v}_{h}^{n}\cdot\mathbf{n}\varphi_{f,h}\cdot\mathbf{n}~dx,
\end{gather*}
and $s_{f,p}(d_{t}p_{f,h},\psi_{f,h})$ is a stabilization term (cf \cite{E.Burman}) with
\begin{equation*}
  s_{f,p}(d_{t}p_{f,h},\psi_{f,h}):=\gamma_{stab}\frac{h\Delta t}{\gamma_{f}\mu_{f}}\int_{\Gamma}d_{t}p_{f,h}^{n}\psi_{f,h}~dx.
\end{equation*}
\end{rem}

\subsection{Stability analysis}

Following \cite{E.Burman} and using Theorem \ref{thm2.1}, we can easily give the stability analysis of the explicit scheme.

\begin{lemm}\label{lemm4.2}
For any~$\underline{\epsilon}^{1}_{f},\overline{\epsilon}^{1}_{f}>0$, we have
\begin{eqnarray}\label{4.2}
\begin{aligned}
   & \mathbf{E}^{N}_{p,h}+\Delta t\frac{\gamma_{stab}}{2}\frac{h}{\gamma_{f}\mu_{f}}\| p_{f,h}^{N}\|_{L^{2}(\Gamma)}^{2}+\Delta t\frac{\gamma'_{stab}}{2}\frac{\gamma_{f}\mu_{f}}{h}(\|\mathbf{v}_{h}^{N}\cdot\mathbf{n}_{f}\|_{L^{2}(\Gamma)}^{2}
   +\|\mathbf{q}_{h}^{N}\cdot\mathbf{n}_{p}\|_{L^{2}(\Gamma)}^{2})\\
   &+\Delta t\sum^{N}_{n=1}[2\mu_{f}(1-(\varsigma+1)\overline{\epsilon}^{1}_{f}C_{TI}-\frac{\underline{\epsilon}^{1}_{f}}{4})\|\mathbf{D}(\mathbf{v}_{h}^{n})\|^{2}_{L^{2}(\Omega_{f})}
   +k^{-1}\|\mathbf{q}_{h}^{n}\|^{2}_{L^{2}(\Omega_{p})} \\
   & +\frac{\Delta t}{2}(2\mu_{p}\| d_{t}\mathbf{D}(\mathbf{U}_{h}^{n})\|^{2}_{L^{2}(\Omega_{p})}+k_{3}\| d_{t}\xi_{h}^{n}\|^{2}_{L^{2}(\Omega_{p})}+k_{2}\| d_{t}\eta_{h}^{n}\|^{2}_{L^{2}(\Omega_{p})}) \\
   & +(\gamma_{f}-(\varsigma+1)(\overline{\epsilon}^{1}_{f})^{-1}) \mu_{f}h^{-1}(\|(\mathbf{v}_{h}^{n}-\mathbf{q}_{h}^{n}-d_{t}\mathbf{U}_{h}^{n})\cdot\mathbf{n}\|^{2}_{L^{2}(\Gamma)}
   +\|(\mathbf{v}_{h}^{n}-d_{t}\mathbf{U}_{h}^{n})\cdot\tau\|^{2}_{L^{2}(\Gamma)}) \\
   &+\frac{\gamma_{stab}}{2}\frac{h}{\gamma_{f}\mu_{f}}\| p_{f,h}^{n}-p_{f,h}^{n-1}\|_{L^{2}(\Gamma)}^{2}\\
   &+\frac{\gamma'_{stab}}{2}\frac{\gamma_{f}\mu_{f}}{h}(\|(\mathbf{v}_{h}^{n}-\mathbf{v}_{h}^{n-1})\cdot\mathbf{n}_{f}\|_{L^{2}(\Gamma)}^{2}
   +\|(\mathbf{q}_{h}^{n}-\mathbf{q}_{h}^{n-1})\cdot\mathbf{n}_{p}\|_{L^{2}(\Gamma)}^{2})]\\
   \leq&  \mathbf{E}^{0}_{p,h}+\Delta t\sum^{N}_{n=1}[\mathcal{T}_{1,a}^{n}+\mathcal{T}_{1,b}^{n}+\mathcal{T}_{1,c}^{n}+\mathcal{T}_{2}^{n}+\mathcal{T}_{3}^{n}
   +\mathcal{T}_{4}^{n}](\mathbf{v}_{h}^{n},\mathbf{q}_{h}^{n},d_{t}\mathbf{U}_{h}^{n},\xi_{h}^{n})\\
   &+\Delta t\frac{\gamma_{stab}}{2}\frac{h}{\gamma_{f}\mu_{f}}\|p_{f,h}^{0}\|_{L^{2}(\Gamma)}^{2}+\Delta t\frac{\gamma'_{stab}}{2}\frac{\gamma_{f}\mu_{f}}{h}(\|\mathbf{v}_{h}^{0}\cdot\mathbf{n}_{f}\|_{L^{2}(\Gamma)}^{2}
   +\|\mathbf{q}_{h}^{0}\cdot\mathbf{n}_{p}\|_{L^{2}(\Gamma)}^{2})\\
   &+\Delta t\sum^{N}_{n=1}(2\underline{\epsilon}^{1}_{f}\mu_{f})^{-1}\|\mathcal{F}(t_{n})\|^{2}.
\end{aligned}
\end{eqnarray}
\end{lemm}

\begin{proof}
Denoting $\widetilde{\mathcal{A}}_{h}(\cdot,\cdot)$ by the collection of terms on the left hand side of (\ref{2.9}), and $\mathbf{y}_{h}=\{\mathbf{U}_{h}, \mathbf{v}_{h}, p_{f,h}, \mathbf{q}_{h}, \xi_{h}, \eta_{h}\}$ by the vector of all the solution components and with $\mathbf{z}_{h}$.

Thus, the loosely-coupled time-stepping method is equivalent to the following formulation
\begin{eqnarray}\label{4.1}
  &&\widetilde{\mathcal{A}}_{h}(\mathbf{y}_{h}^{n},\mathbf{z}_{h})+s_{f,p}(d_{t}p_{f,h},\psi_{f,h})+s_{f,q}(d_{t}\mathbf{q}_{h}\cdot\mathbf{n},\mathbf{r}_{h}\cdot\mathbf{n})+
  s_{f,v}(d_{t}\mathbf{v}_{h}^{n}\cdot\mathbf{n},\varphi_{f,h}\cdot\mathbf{n})  \nonumber\\
  &&=\mathcal{F}^{n}(t_{n},\mathbf{z}_{h})+\underbrace{\int_{\Gamma}\gamma_{f}\mu_{f}h^{-1}(\mathbf{v}_{h}^{n}-\mathbf{v}_{h}^{n-1})\cdot\tau(-\varphi_{p,h})\cdot\tau~dx}
  _{\mathcal{T}_{1,a}^{n}(\varphi_{f,h},\mathbf{r}_{h},\varphi_{p,h},w_{h})} \nonumber\\
  &&+\underbrace{\int_{\Gamma}\gamma_{f}\mu_{f}h^{-1}((\mathbf{v}_{h}^{n}-\mathbf{v}_{h}^{n-1})-d_{t}(\mathbf{U}_{h}^{n}-\mathbf{U}_{h}^{n-1}))\cdot\mathbf{n}(-\mathbf{r}_{h})\cdot\mathbf{n} ~dx}_{\mathcal{T}_{1,b}^{n}(\varphi_{f,h},\mathbf{r}_{h},\varphi_{p,h},w_{h})} \nonumber\\
  &&+\underbrace{\int_{\Gamma}\gamma_{f}\mu_{f}h^{-1}((\mathbf{v}_{h}^{n}-\mathbf{v}_{h}^{n-1})-(\mathbf{q}_{h}^{n}-\mathbf{q}_{h}^{n-1}))
  \cdot\mathbf{n}(-\varphi_{p,h})\cdot\mathbf{n}~dx}_{\mathcal{T}_{1,c}^{n}(\varphi_{f,h},\mathbf{r}_{h},\varphi_{p,h},w_{h})} \\
  &&-\underbrace{\int_{\Gamma}2\mu_{f}(\mathbf{D}(\mathbf{v}_{h}^{n})-\mathbf{D}(\mathbf{v}_{h}^{n-1}))\mathbf{n}\cdot [\mathbf{n}(\varphi_{f,h}-\mathbf{r}_{h}-\varphi_{p,h})\cdot\mathbf{n}+\tau(\varphi_{f,h}-\varphi_{p,h})\cdot\tau]~dx}_{\mathcal{T}_{2}^{n}(\varphi_{f,h},\mathbf{r}_{h},\varphi_{p,h},w_{h})} \nonumber\\
  &&+\underbrace{\int_{\Gamma}\mathbf{n}\cdot(p_{f,h}^{n}-p_{f,h}^{n-1})\mathbf{n}(\varphi_{f,h}-\mathbf{r}_{h}-\varphi_{p,h})\cdot\mathbf{n}~dx}_{\mathcal{T}_{3}^{n}(\varphi_{f,h},\mathbf{r}_{h},\varphi_{p,h},w_{h})}
  -\underbrace{\int_{\Omega_{p}} k_{1}(\eta_{h}^{n}-\eta_{h}^{n-1})w_{h}~dx}_{\mathcal{T}_{4}^{n}(\varphi_{f,h},\mathbf{r}_{h},\varphi_{p,h},w_{h})}.\nonumber
\end{eqnarray}
Using (\ref{4.1}), following the proof of Theorem \ref{thm2.1}, summing up with
respect to the index $n$, multiplying by $\Delta t$,
%
we see that  (\ref{4.2}) holds. The proof is finished.
\end{proof}

The stability of the loosely-coupled time-stepping method follows from (\ref{4.2}) provided that the terms $\mathcal{T}_{1,a}^{n},~\mathcal{T}_{1,b}^{n},~\mathcal{T}_{1,c}^{n},~\mathcal{T}_{2}^{n},~\mathcal{T}_{3}^{n},~\mathcal{T}_{4}^{n}$ defined in (\ref{4.1}) can be bounded. To do that, we give the following lemma.

\begin{lemm}\label{lemm4.1}
For any~$\epsilon_{f}^{2},\underline{\epsilon}^{3}_{f},\overline{\epsilon}^{3}_{f}>0$, there holds 
\begin{eqnarray}\label{4.3}
  &&\Delta t\sum^{N}_{n=1}[\mathcal{T}_{1,a}^{n}+\mathcal{T}_{1,b}^{n}+\mathcal{T}_{1,c}^{n}+\mathcal{T}_{2}^{n}+\mathcal{T}_{3}^{n}+\mathcal{T}_{4}^{n}]
  (\mathbf{v}_{h}^{n},\mathbf{q}_{h}^{n},d_{t}\mathbf{U}_{h}^{n},d_{t}\xi_{h}^{n}) \nonumber\\
  &&\leq \frac{\gamma_{f}\mu_{f}}{2}\frac{\Delta t}{h}(2\|\mathbf{v}_{h}^{0}\cdot\mathbf{n}\|_{L^{2}(\Gamma)}^{2}+
  \|\mathbf{v}_{h}^{0}\cdot\tau\|_{L^{2}(\Gamma)}^{2}+\|\mathbf{q}_{h}^{0}\cdot\mathbf{n}\|_{L^{2}(\Gamma)}^{2}+\| d_{t}\mathbf{U}_{h}^{0}\cdot\mathbf{n}\|_{L^{2}(\Gamma)}^{2}) \nonumber\\
  && +\frac{\epsilon_{f}^{2}C_{TI}}{2}\mu_{f}\|\mathbf{D}(\mathbf{v}_{h}^{0})\|_{L^{2}(\Omega_{f})}^{2}+\Delta t\sum^{N}_{n=1}[\frac{\mu_{f}}{2h}(\gamma_{f}(\underline{\epsilon}_{f}^{3}+\overline{\epsilon}_{f}^{3}+1)+2(\epsilon_{f}^{2})^{-1}) \nonumber\\
   &&\times\|(\mathbf{v}_{h}^{n}-\mathbf{q}_{h}^{n}-d_{t}\mathbf{U}_{h}^{n})\cdot\mathbf{n}\|^{2}_{L^{2}(\Gamma)}+\frac{\mu_{f}}{h}(\frac{\gamma_{f}}{2}+(\epsilon_{f}^{2})^{-1})
   \|(\mathbf{v}_{h}^{n}-d_{t}\mathbf{U}_{h}^{n})\cdot\tau\|^{2}_{L^{2}(\Gamma)} \\
   && +2\epsilon_{f}^{2}C_{TI}\mu_{f}\|\mathbf{D}(\mathbf{v}_{h}^{n})\|_{L^{2}(\Omega_{f})}^{2}+\frac{k_{1}\Delta t}{2}(\| d_{t}\eta_{h}^{n}\|_{L^{2}(\Omega_{p})}^{2}+\| d_{t}\xi_{h}^{n}\|_{L^{2}(\Omega_{p})}^{2}) \nonumber\\
   && +\frac{(\overline{\epsilon}_{f}^{3})^{-1}}{2}\frac{h}{\gamma_{f}\mu_{f}}\| p_{f,h}^{n}-p_{f,h}^{n-1}\|_{L^{2}(\Gamma)}^{2}+\frac{((\underline{\epsilon}_{f}^{3})^{-1}-1)}{2}\frac{\gamma_{f}\mu_{f}}{h} \nonumber\\
   && \times(\|(\mathbf{v}_{h}^{n}-\mathbf{v}_{h}^{n-1})\cdot\mathbf{n}\|_{L^{2}(\Gamma)}^{2}
   +\|(\mathbf{q}_{h}^{n}-\mathbf{q}_{h}^{n-1})\cdot\mathbf{n}\|_{L^{2}(\Gamma)}^{2})].\nonumber
\end{eqnarray}
\end{lemm}

\begin{proof}
As for $\mathcal{T}_{1,a}^{n}$, it is easy to check that
\begin{eqnarray}\label{5.1}
\begin{aligned}
  &\mathcal{T}_{1,a}^{n}(\mathbf{v}_{h}^{n},\mathbf{q}_{h}^{n},d_{t}\mathbf{U}_{h}^{n},d_{t}\xi_{h}^{n}) \\
  =& \int_{\Gamma}\gamma_{f}\mu_{f}h^{-1}(\mathbf{v}_{h}^{n}-\mathbf{v}_{h}^{n-1})\cdot\tau(\mathbf{v}_{h}^{n}-d_{t}\mathbf{U}_{h}^{n})\cdot\tau~dx
   -\int_{\Gamma}\gamma_{f}\mu_{f}h^{-1}(\mathbf{v}_{h}^{n}-\mathbf{v}_{h}^{n-1})\cdot\tau\mathbf{v}_{h}^{n}\cdot\tau~dx\\
   \leq&\frac{\gamma_{f}\mu_{f}}{2h}\|(\mathbf{v}_{h}^{n}-d_{t}\mathbf{U}_{h}^{n})\cdot\tau\|^{2}_{L^{2}(\Gamma)}+\Delta t \frac{\gamma_{f}\mu_{f}}{2h}(\Delta t\|d_{t}\mathbf{v}_{h}^{n}\cdot\tau\|^{2}_{L^{2}(\Gamma)}\\
   &-d_{t}\|\mathbf{v}_{h}^{n}\cdot\tau\|^{2}_{L^{2}(\Gamma)}-\Delta t\|d_{t}\mathbf{v}_{h}^{n}\cdot\tau\|^{2}_{L^{2}(\Gamma)}).
\end{aligned}
\end{eqnarray}

As for $\mathcal{T}_{1,b}^{n}$, we introduce $\mathbf{u}_{h}=\mathbf{v}_{h}-d_{t}\mathbf{U}_{h}$ and obtain
\begin{eqnarray}\label{5.2}
\begin{aligned}
  &\mathcal{T}_{1,b}^{n}(\mathbf{v}_{h}^{n},\mathbf{q}_{h}^{n},d_{t}\mathbf{U}_{h}^{n},d_{t}\xi_{h}^{n})\\
  =& \int_{\Gamma}\gamma_{f}\mu_{f}h^{-1}(\mathbf{u}_{h}^{n}-\mathbf{u}_{h}^{n-1})\cdot\mathbf{n}(\mathbf{u}_{h}^{n}-\mathbf{q}_{h}^{n})\cdot\mathbf{n}~dx
  -\int_{\Gamma}\gamma_{f}\mu_{f}h^{-1}(\mathbf{u}_{h}^{n}-\mathbf{u}_{h}^{n-1})\cdot\mathbf{n}\mathbf{u}_{h}^{n}\cdot\mathbf{n}~dx\\
   \leq& \frac{\gamma_{f}\mu_{f}}{2h}\|(\mathbf{v}_{h}^{n}-\mathbf{q}_{h}^{n}-d_{t}\mathbf{U}_{h}^{n})\cdot\mathbf{n}\|^{2}_{L^{2}(\Gamma)}
   +\frac{\Delta t}{h}\frac{\gamma_{f}\mu_{f}}{2}(\Delta t\|d_{t}\mathbf{u}_{h}^{n}\cdot\mathbf{n}\|^{2}_{L^{2}(\Gamma)}\\
   & -d_{t}\|\mathbf{u}_{h}^{n}\cdot\mathbf{n}\|^{2}_{L^{2}(\Gamma)}-\Delta t\|d_{t}\mathbf{u}_{h}^{n}\cdot\mathbf{n}\|^{2}_{L^{2}(\Gamma)})  \\
   =& \frac{\gamma_{f}\mu_{f}}{2h}\|(\mathbf{v}_{h}^{n}-\mathbf{q}_{h}^{n}-d_{t}\mathbf{U}_{h}^{n})\cdot\mathbf{n}\|^{2}_{L^{2}(\Gamma)}-
   \frac{\Delta t}{h}\frac{\gamma_{f}\mu_{f}}{2}d_{t}\|(\mathbf{v}_{h}^{n}-d_{t}\mathbf{U}_{h}^{n})\cdot\mathbf{n}\|^{2}_{L^{2}(\Gamma)}.
\end{aligned}
\end{eqnarray}
To estimate~$\mathcal{T}_{1,c}^{n}$, we introduce the auxiliary variable~$\mathbf{w}_{h}=\mathbf{v}_{h}-\mathbf{q}_{h}$~at the interface~$\Gamma$ and get
\begin{eqnarray}\label{5.3}
\begin{aligned}
  &\mathcal{T}_{1,c}^{n}(\mathbf{v}_{h}^{n},\mathbf{q}_{h}^{n},d_{t}\mathbf{U}_{h}^{n},d_{t}\xi_{h}^{n})\\
  =& \int_{\Gamma}\gamma_{f}\mu_{f}h^{-1}(\mathbf{w}_{h}^{n}-\mathbf{w}_{h}^{n-1})\cdot\mathbf{n}(\mathbf{w}_{h}^{n}-d_{t}\mathbf{U}_{h}^{n})\cdot\mathbf{n}~dx
   -\int_{\Gamma}\gamma_{f}\mu_{f}h^{-1}(\mathbf{w}_{h}^{n}-\mathbf{w}_{h}^{n-1})\cdot\mathbf{n}\mathbf{w}_{h}^{n}\cdot\mathbf{n}~dx \\
   \leq& \underline{\epsilon}_{f}^{3}
   \frac{\gamma_{f}\mu_{f}}{2h}\|(\mathbf{v}_{h}^{n}-\mathbf{q}_{h}^{n}-d_{t}\mathbf{U}_{h}^{n})\cdot\mathbf{n}\|^{2}_{L^{2}(\Gamma)}
   +\frac{\Delta t}{h}\frac{\gamma_{f}\mu_{f}}{2}(\Delta t(\underline{\epsilon}_{f}^{3})^{-1}\| d_{t}\mathbf{w}_{h}^{n}\cdot\mathbf{n}\|^{2}_{L^{2}(\Gamma)} \\
  & -d_{t}\|\mathbf{w}_{h}^{n}\cdot\mathbf{n}\|^{2}_{L^{2}(\Gamma)}-\Delta t\|d_{t}\mathbf{w}_{h}^{n}\cdot\mathbf{n}\|^{2}_{L^{2}(\Gamma)})  \\
   \leq&\frac{\gamma_{f}\mu_{f}}{2h}(\underline{\epsilon}_{f}^{3}\|(\mathbf{v}_{h}^{n}-\mathbf{q}_{h}^{n}-d_{t}\mathbf{U}_{h}^{n})\cdot\mathbf{n}\|^{2}_{L^{2}(\Gamma)}  +((\underline{\epsilon}_{f}^{3})^{-1}-1)\|(\mathbf{v}_{h}^{n}-\mathbf{v}_{h}^{n-1})\cdot\mathbf{n}\|^{2}_{L^{2}(\Gamma)}\\
  & + ((\underline{\epsilon}_{f}^{3})^{-1}-1)\|(\mathbf{q}_{h}^{n}-\mathbf{q}_{h}^{n-1})\cdot\mathbf{n}\|^{2}_{L^{2}(\Gamma)})
    - \frac{\Delta t}{h}\frac{\gamma_{f}\mu_{f}}{2}d_{t}\|(\mathbf{v}_{h}^{n}-\mathbf{q}_{h}^{n})\cdot\mathbf{n}\|^{2}_{L^{2}(\Gamma)}.
\end{aligned}
\end{eqnarray}
Using (\ref{5.1})-(\ref{5.3}), we have
\begin{eqnarray*}
  &&\Delta t\sum_{n=1}^{N}\mathcal{T}_{1}^{n}(\mathbf{v}_{h}^{n},\mathbf{q}_{h}^{n},d_{t}\mathbf{U}_{h}^{n},d_{t}\xi_{h}^{n}) \\
  &=& \Delta t\sum_{n=1}^{N}
  [\mathcal{T}_{1,a}^{n}+\mathcal{T}_{1,b}^{n}+\mathcal{T}_{1,c}^{n}](\mathbf{v}_{h}^{n},\mathbf{q}_{h}^{n},d_{t}\mathbf{U}_{h}^{n},d_{t}\xi_{h}^{n}) \\
  &\leq& \frac{\gamma_{f}\mu_{f}}{2}\frac{\Delta t}{h}(2\|\mathbf{v}_{h}^{0}\cdot\mathbf{n}\|_{L^{2}(\Gamma)}^{2}+
  \|\mathbf{v}_{h}^{0}\cdot\tau\|_{L^{2}(\Gamma)}^{2}+\|\mathbf{q}_{h}^{0}\cdot\mathbf{n}\|_{L^{2}(\Gamma)}^{2}+\| d_{t}\mathbf{U}_{h}^{0}\cdot\mathbf{n}\|_{L^{2}(\Gamma)}^{2}) \\
 && +\Delta t\frac{\gamma_{f}\mu_{f}}{2h}\sum_{n=1}^{N}[\|(\mathbf{v}_{h}^{n}-d_{t}\mathbf{U}_{h}^{n})\cdot\tau\|^{2}_{L^{2}(\Gamma)}+(1+\underline{\epsilon}_{f}^{3})
 \|(\mathbf{v}_{h}^{n}-\mathbf{q}_{h}^{n}-d_{t}\mathbf{U}_{h}^{n})\cdot\mathbf{n}\|^{2}_{L^{2}(\Gamma)} \\
   && +((\underline{\epsilon}_{f}^{3})^{-1}-1)\|(\mathbf{v}_{h}^{n}-\mathbf{v}_{h}^{n-1})\cdot\mathbf{n}\|^{2}_{L^{2}(\Gamma)}
  + ((\underline{\epsilon}_{f}^{3})^{-1}-1)\|(\mathbf{q}_{h}^{n}-\mathbf{q}_{h}^{n-1})\cdot\mathbf{n}\|^{2}_{L^{2}(\Gamma)}].
\end{eqnarray*}
As for $\mathcal{T}_{2}^{n}$ and $\mathcal{T}_{3}^{n}$, we have
\begin{eqnarray*}
\begin{aligned}
  \mathcal{T}_{2}^{n}(\mathbf{v}_{h}^{n},\mathbf{q}_{h}^{n},d_{t}\mathbf{U}_{h}^{n},d_{t}\xi_{h}^{n}) \leq& \epsilon_{f}^{2}C_{TI}\mu_{f}(\|\mathbf{D}(\mathbf{v}_{h}^{n})\|_{L^{2}(\Omega_{f})}^{2}+\|\mathbf{D}(\mathbf{v}_{h}^{n-1})\|_{L^{2}(\Omega_{f})}^{2}) \\
  & + (\epsilon_{f}^{2})^{-1}h^{-1}\mu_{f}(\|(\mathbf{v}_{h}^{n}-d_{t}\mathbf{U}_{h}^{n})\cdot\tau\|^{2}_{L^{2}(\Gamma)}+
 \|(\mathbf{v}_{h}^{n}-\mathbf{q}_{h}^{n}-d_{t}\mathbf{U}_{h}^{n})\cdot\mathbf{n}\|^{2}_{L^{2}(\Gamma)}),\\
 \mathcal{T}_{3}^{n}(\mathbf{v}_{h}^{n},\mathbf{q}_{h}^{n},d_{t}\mathbf{U}_{h}^{n},d_{t}\xi_{h}^{n}) \leq& (\overline{\epsilon}_{f}^{3})^{-1}\frac{h}{2\gamma_{f}\mu_{f}}
  \| p_{f,h}^{n}-p_{f,h}^{n-1}\|_{L^{2}(\Gamma)}^{2}+\overline{\epsilon}_{f}^{3}\frac{\gamma_{f}\mu_{f}}{2h}\|(\mathbf{v}_{h}^{n}
  -\mathbf{q}_{h}^{n}-d_{t}\mathbf{U}_{h}^{n})\cdot\mathbf{n}\|^{2}_{L^{2}(\Gamma)}.
\end{aligned}
\end{eqnarray*}
To determine an appropriate upper bound for $\mathcal{T}_{4}^{n}$, we get
\begin{eqnarray*}
\begin{aligned}
  \mathcal{T}_{4}^{n}(\mathbf{v}_{h}^{n},\mathbf{q}_{h}^{n},d_{t}\mathbf{U}_{h}^{n},d_{t}\xi_{h}^{n}) =& k_{1}\Delta t\int_{\Omega_{p}}(d_{t}\eta_{h}^{n})(d_{t}\xi_{h}^{n})~d\mathbf{x} \\
  \leq& \frac{k_{1}\Delta t}{2}\|d_{t}\eta_{h}^{n}\|_{L^{2}(\Omega_{p})}^{2}+\frac{k_{1}\Delta t}{2}\|d_{t}\xi_{h}^{n}\|_{L^{2}(\Omega_{p})}^{2}.
\end{aligned}
\end{eqnarray*}
The proof is complete.
\end{proof}

Using Lemma \ref{lemm4.2} and Lemma \ref{lemm4.1}, we can obtain the stability of the loosely-coupled time-stepping method.

\begin{thrm}\label{thm4.1}
For any~$\overline{\epsilon}_{f}^{1},\underline{\epsilon}_{f}^{1},\underline{\epsilon}^{3}_{f},\overline{\epsilon}^{3}_{f},\epsilon_{f}^{2}>0$ with
$
  (2-2(\varsigma+1)\overline{\epsilon}^{1}_{f}\mathcal{C}_{TI}-\frac{\underline{\epsilon}^{1}_{f}}{2}-2\epsilon_{f}^{2}\mathcal{C}_{TI})>0$ and $(2-(\underline{\epsilon}_{f}^{3}+\overline{\epsilon}_{f}^{3}+1))=\delta>0$
provided that the penalty and stabilization parameters are large enough, more precisely
\begin{eqnarray*}
   &&\gamma_{stab}\geq(\overline{\epsilon}_{f}^{3})^{-1},~~~~~\gamma'_{stab}\geq(\underline{\epsilon}_{f}^{3})^{-1}-1  \\
   && \gamma_{f}>\frac{2(\varsigma+1)(\overline{\epsilon}^{1}_{f})^{-1}+2(\epsilon_{f}^{2})^{-1}}{\delta},
\end{eqnarray*}
and if $k_{2}>k_{1},~~k_{3}>k_{1}$  and $\Delta t<Ch$, then there exist positive constants $c_{f}^{1}$ and $C^{1}_{f}$ such that
\begin{eqnarray}
   && \mathbf{E}^{N}_{p,h}+\Delta t\frac{\gamma_{stab}}{2}\frac{h}{\gamma_{f}\mu_{f}}\|p_{f,h}^{N}\|_{L^{2}(\Gamma)}^{2}+\Delta t\frac{\gamma'_{stab}}{2}\frac{\gamma_{f}\mu_{f}}{h}(\|\mathbf{v}_{h}^{N}\cdot\mathbf{n}_{f}\|_{L^{2}(\Gamma)}^{2}+\|\mathbf{q}_{h}^{N}\cdot\mathbf{n}_{p}\|_{L^{2}(\Gamma)}^{2})
   \nonumber\\
   &&+\Delta tc_{f}^{1} \sum^{N}_{n=1}[\mu_{f}\|\mathbf{D}(\mathbf{v}_{h}^{n})\|^{2}_{L^{2}(\Omega_{f})}
   +k^{-1}\|\mathbf{q}_{h}^{n}\|^{2}_{L^{2}(\Omega_{p})} \nonumber\\
   &&  \mu_{f}h^{-1}(\|(\mathbf{v}_{h}^{n}-\mathbf{q}_{h}^{n}-d_{t}\mathbf{U}_{h}^{n})\cdot\mathbf{n}\|^{2}_{L^{2}(\Gamma)}
   +\|(\mathbf{v}_{h}^{n}-d_{t}\mathbf{U}_{h}^{n})\cdot\tau\|^{2}_{L^{2}(\Gamma)})\label{4.6}\\
   &&+\Delta t(2\mu_{p}\|d_{t}\mathbf{D}(\mathbf{U}_{h}^{n})\|^{2}_{L^{2}(\Omega_{p})}+k_{3}\|d_{t}\xi_{h}^{n}\|^{2}_{L^{2}(\Omega_{p})}+k_{2}\| d_{t}\eta_{h}^{n}\|^{2}_{L^{2}(\Omega_{p})})]\nonumber\\
   &\leq&  \mathbf{E}^{0}_{p,h}+C_{f}^{1}\mu_{f}(\|\mathbf{v}_{h}^{0}\cdot\mathbf{n}\|_{L^{2}(\Gamma)}^{2}+
  \|\mathbf{v}_{h}^{0}\cdot\tau\|_{L^{2}(\Gamma)}^{2}+\|\mathbf{q}_{h}^{0}\cdot\mathbf{n}\|_{L^{2}(\Gamma)}^{2}
  +\| d_{t}\mathbf{U}_{h}^{0}\cdot\mathbf{n}\|_{L^{2}(\Gamma)}^{2} \nonumber\\
   &&+\|\mathbf{D}(\mathbf{v}_{h}^{0})\|_{L^{2}(\Omega_{f})}^{2}+\mu_{f}^{-2}\|p_{f,h}^{0}\|_{L^{2}(\Gamma)}^{2})+\Delta t\mu_{f}^{-1}C_{f}^{1}\sum^{N}_{n=1}\|\mathcal{F}(t_{n})\|^{2}\nonumber
\end{eqnarray}
with
\begin{eqnarray*}
  c_{f}^{1} &<& \min\{{(2-2(\varsigma+1)\overline{\epsilon}^{1}_{f}\mathcal{C}_{TI}-\frac{\underline{\epsilon}^{1}_{f}}{2}-2\epsilon_{f}^{2}\mathcal{C}_{TI}),(k_{3}- k_{1}),(k_{2}-k_{1}),}\\
    &&(\gamma_{f}(2-(\underline{\epsilon}_{f}^{3}+\overline{\epsilon}_{f}^{3}+1))-2(\varsigma+1)(\overline{\epsilon}^{1}_{f})^{-1}-2(\epsilon_{f}^{2})^{-1}), (\gamma_{f}-2(\varsigma+1)(\overline{\epsilon}^{1}_{f})^{-1}-2(\epsilon_{f}^{2})^{-1})\},\\
  C_{f}^{1} &>&  \max\{{(2\underline{\epsilon}^{1}_{f})^{-1},\frac{\epsilon_{f}^{2}\mathcal{C}_{TI}}{2},\frac{\gamma_{stab}}{2}\frac{h\Delta t}{\gamma_{f}},\gamma_{f}\frac{\gamma'_{stab}}{2}}\}.
\end{eqnarray*}
\end{thrm}
\begin{proof}
Using (\ref{4.2}) and (\ref{4.3}), we obtain
\begin{eqnarray}
   && \mathbf{E}^{N}_{p,h}+\Delta t\frac{\gamma_{stab}}{2}\frac{h}{\gamma_{f}\mu_{f}}\|p_{f,h}^{N}\|_{L^{2}(\Gamma)}^{2}+\Delta t\frac{\gamma'_{stab}}{2}\frac{\gamma_{f}\mu_{f}}{h}(\|\mathbf{v}_{h}^{N}\cdot\mathbf{n}_{f}\|_{L^{2}(\Gamma)}^{2}
   +\|\mathbf{q}_{h}^{N}\cdot\mathbf{n}_{p}\|_{L^{2}(\Gamma)}^{2})\nonumber\\
   &&+\Delta t\sum^{N}_{n=1}[\mu_{f}(2-2(\varsigma+1)\overline{\epsilon}^{1}_{f}C_{TI}-\frac{\underline{\epsilon}^{1}_{f}}{2}
   -2\epsilon_{f}^{2}C_{TI})\|\mathbf{D}(\mathbf{v}_{h}^{n})\|^{2}_{L^{2}(\Omega_{f})}
   +k^{-1}\|\mathbf{q}_{h}^{n}\|^{2}_{L^{2}(\Omega_{p})} \nonumber\\
   && +\frac{\Delta t}{2}(2\mu_{p}\|d_{t}\mathbf{D}(\mathbf{U}_{h}^{n})\|^{2}_{L^{2}(\Omega_{p})}+(k_{3}-k_{1})\| d_{t}\xi_{h}^{n}\|^{2}_{L^{2}(\Omega_{p})}+(k_{2}- k_{1})\|d_{t}\eta_{h}^{n}\|^{2}_{L^{2}(\Omega_{p})}) \nonumber\\
   && +\frac{1}{2}(\gamma_{f}(2-(\underline{\epsilon}_{f}^{3}+\overline{\epsilon}_{f}^{3}+1))
   -2(\varsigma+1)(\overline{\epsilon}^{1}_{f})^{-1}-2(\epsilon_{f}^{2})^{-1}) \mu_{f}h^{-1}\|(\mathbf{v}_{h}^{n}-\mathbf{q}_{h}^{n}-d_{t}\mathbf{U}_{h}^{n})\cdot\mathbf{n}\|^{2}_{L^{2}(\Gamma)}\nonumber\\
   &&+\frac{1}{2}
   (\gamma_{f}-2(\varsigma+1)(\overline{\epsilon}^{1}_{f})^{-1}-2(\epsilon_{f}^{2})^{-1})\mu_{f}h^{-1}
   \|(\mathbf{v}_{h}^{n}-d_{t}\mathbf{U}_{h}^{n})\cdot\tau\|^{2}_{L^{2}(\Gamma)} \label{4.4}\\
   &&+\frac{1}{2}(\gamma_{stab}-(\overline{\epsilon}_{f}^{3})^{-1})\frac{h}{\gamma_{f}\mu_{f}}\| p_{f,h}^{n}-p_{f,h}^{n-1}\|_{L^{2}(\Gamma)}^{2}\nonumber\\
   &&+\frac{1}{2}(\gamma'_{stab}-(\underline{\epsilon}_{f}^{3})^{-1}+1)\frac{\gamma_{f}\mu_{f}}{h}(\|(\mathbf{v}_{h}^{n}
   -\mathbf{v}_{h}^{n-1})\cdot\mathbf{n}_{f}\|_{L^{2}(\Gamma)}^{2}
   +\|(\mathbf{q}_{h}^{n}-\mathbf{q}_{h}^{n-1})\cdot\mathbf{n}_{p}\|_{L^{2}(\Gamma)}^{2})]\nonumber\\
   &&\leq  \mathbf{E}^{0}_{p,h}+\frac{\gamma_{f}\mu_{f}}{2}\frac{\Delta t}{h}(2\|\mathbf{v}_{h}^{0}\cdot\mathbf{n}\|_{L^{2}(\Gamma)}^{2}+
  \|\mathbf{v}_{h}^{0}\cdot\tau\|_{L^{2}(\Gamma)}^{2}+\|\mathbf{q}_{h}^{0}\cdot\mathbf{n}\|_{L^{2}(\Gamma)}^{2}+\| d_{t}\mathbf{U}_{h}^{0}\cdot\mathbf{n}\|_{L^{2}(\Gamma)}^{2})\nonumber\\
   &&+\frac{\epsilon_{f}^{2}\mathcal{C}_{TI}}{2}\mu_{f}\|\mathbf{D}(\mathbf{v}_{h}^{0})\|_{L^{2}(\Omega_{f})}^{2}+\Delta t\frac{\gamma_{stab}}{2}\frac{h}{\gamma_{f}\mu_{f}}\| p_{f,h}^{0}\|_{L^{2}(\Gamma)}^{2}+\Delta t\frac{\gamma'_{stab}}{2}\frac{\gamma_{f}\mu_{f}}{h}(\|\mathbf{v}_{h}^{0}\cdot\mathbf{n}_{f}\|_{L^{2}(\Gamma)}^{2}
     \nonumber\\
   &&+\|\mathbf{q}_{h}^{0}\cdot\mathbf{n}_{p}\|_{L^{2}(\Gamma)}^{2})+\Delta t\sum^{N}_{n=1}(2\underline{\epsilon}^{1}_{f}\mu_{f})^{-1}\|\mathcal{F}(t_{n})\|^{2}.\nonumber
\end{eqnarray}
Taking
\begin{gather*}
  2-2(\varsigma+1)\overline{\epsilon}^{1}_{f}\mathcal{C}_{TI}-\frac{\underline{\epsilon}^{1}_{f}}{2}-2\epsilon_{f}^{2}\mathcal{C}_{TI}>0,
  ~\gamma_{f}-2(\varsigma+1)(\overline{\epsilon}^{1}_{f})^{-1}-2(\epsilon_{f}^{2})^{-1}>0, \\
  k_{3}- k_{1}>0,~~k_{2}-k_{1}>0,~~\gamma_{stab}-(\overline{\epsilon}_{f}^{3})^{-1}>0,~~\gamma'_{stab}-(\underline{\epsilon}_{f}^{3})^{-1}+1>0,\\
  \gamma_{f}(2-(\underline{\epsilon}_{f}^{3}+\overline{\epsilon}_{f}^{3}+1))-2(\varsigma+1)(\overline{\epsilon}^{1}_{f})^{-1}-2(\epsilon_{f}^{2})^{-1}>0,
\end{gather*}
we imply that (\ref{4.6}) holds. The proof is complete.
\end{proof}

\subsection{Error analysis}

Let $\mathcal{W}_{f,h},~\mathcal{W}_{p,h}$ be the $L^{2}$-projection operators onto~$Q_{h}^{f},Q_{h}^{p}$ satisfying
\begin{eqnarray}
  (p_{f}-\mathcal{W}_{f,h}p_{f},\psi_{f,h})_{\Omega_{f}} &=& 0,~~~~~~\forall\psi_{f,h}\in Q_{h}^{f}, \\
  (p_{p}-\mathcal{W}_{p,h}p_{p},z_{h})_{\Omega_{p}} &=& 0,~~~~~~\forall z_{h}\in Q_{h}^{p}.
\end{eqnarray}
It is well-known that the approximation properties holds (cf. \cite{P.G. Ciarlet}):
\begin{eqnarray}
  \|p_{f}-\mathcal{W}_{f,h}p_{f}\|_{L^{2}(\Omega_{f})} &\leq& Ch^{r_{1}}\|p_{f}\|_{H^{r_{1}}(\Omega_{f})},~~~0\leq r_{1}\leq s_{f}+1, \label{5.14}\\
  \|p_{p}-\mathcal{W}_{p,h}p_{p}\|_{L^{2}(\Omega_{p})} &\leq& Ch^{r_{2}}\|p_{p}\|_{H^{r_{2}}(\Omega_{p})},~~~0\leq r_{2}\leq s_{p}+1,
\end{eqnarray}
where $s_{f}$ and $s_{p}$ the degrees of piecewise polynomials of the spaces $Q_{h}^{f}$ and $Q_{h}^{p}$.

Next, we define a Stokes-like projection operator~$\mathcal{I}_{h}:\mathbf{V}^{f}\rightarrow\mathbf{V}_{h}^{f}\times Q_{h}^{f},~\mathbf{X}^{p}\rightarrow\mathbf{X}^{p}_{h}\times M^{p}_{h}$ for all~$\varphi_{f}\in\mathbf{V}^{f},~\varphi_{p}\in\mathbf{X}^{p}$~by
\begin{eqnarray}
  (\nabla\cdot\varphi_{f},\psi_{f,h})_{\Omega_{f}} &=& (\nabla\cdot\mathcal{I}_{h}\varphi_{f},\psi_{f,h})_{\Omega_{f}},~~~\forall\psi_{f,h}\in Q^{f}_{h}, \\
  (\nabla\cdot\varphi_{p},w_{h})_{\Omega_{p}} &=& (\nabla\cdot\mathcal{I}_{h}\varphi_{p},w_{h})_{\Omega_{p}},~~~~~\forall w_{h}\in M^{p}_{h}.
\end{eqnarray}
From \cite{M.Fernandez}, we know that  there hold the following  estimates
\begin{eqnarray}
  \|\varphi_{f}-\mathcal{I}_{h}\varphi_{f}\|_{H^{1}(\Omega_{f})} &\leq& Ch^{r_{3}}\|\varphi_{f}\|_{H^{r_{3}+1}(\Omega_{f})},~~~0\leq r_{3}\leq k_{f}, \label{5.18}\\
  \|\varphi_{p}-\mathcal{I}_{h}\varphi_{p}\|_{H^{1}(\Omega_{p})} &\leq& Ch^{r_{4}}\|\varphi_{p}\|_{H^{r_{4}+1}(\Omega_{p})},~~~0\leq r_{4}\leq k_{p},\label{5.19}
\end{eqnarray}
where $k_{f}$ and $k_{p}$ the degrees of polynomials in the spaces $\mathbf{V}_{h}^{f}$ and $\mathbf{X}^{p}_{h}$.

Let $\Pi_{h}$ be an interpolation onto~$\mathbf{V}^{p}_{h}$~satisfying for any $\theta>0$ and  $\mathbf{r}\in\mathbf{V}^{p}\cap H^{\theta}(\Omega_{p})$
\begin{eqnarray}
  (\nabla\cdot\Pi_{h}\mathbf{r},z_{h})_{\Omega_{p}} &=& (\nabla\cdot\mathbf{r},z_{h})_{\Omega_{p}},~~~\forall z_{h}\in Q_{h}^{p}, \\
  \langle\Pi_{h}\mathbf{r}\cdot\mathbf{n}_{p},\mathbf{r}_{h}\cdot\mathbf{n}_{p}\rangle_{\Gamma} &=&
  \langle \mathbf{r}\cdot\mathbf{n}_{p},\mathbf{r}_{h}\cdot\mathbf{n}_{p}\rangle_{\Gamma},~~~\forall\mathbf{r}_{h}\in\mathbf{V}^{p}_{h}.
\end{eqnarray}
From \cite{P.G. Ciarlet}, it is well-known that there hold
\begin{eqnarray}
  &&\|\mathbf{r}-\Pi_{h}\mathbf{r}\|_{L^{2}(\Omega_{p})} \leq Ch^{r_{5}}\|\mathbf{r}\|_{H^{r_{5}}(\Omega_{p})},~~~1\leq r_{5}\leq k_{s}+1, \label{5.22}\\
 &&\|\Pi_{h}\mathbf{r}\|_{H(\mathrm{div};\Omega_{p})}\leq C(\|\mathbf{r}\|_{H^{\theta}(\Omega_{p})}+\|\nabla\cdot\mathbf{r}\|_{L^{2}(\Omega_{p})}),
\end{eqnarray}
where $k_{s}$ the degree of piecewise polynomial of the spaces $\mathbf{V}^{p}_{h}$.

For any $z\in H^{1}(\Omega_{p})$, we define its elliptic projection $\mathcal{S}_{h}:H^{1}(\Omega_{p})\rightarrow Q^{p}_{h}$ by
\begin{eqnarray}
  (K\nabla\mathcal{S}_{h}z,\nabla z_{h})_{\Omega_{p}} &=& (K\nabla z,\nabla z_{h})_{\Omega_{p}},~~~\forall z\in Q^{p}_{h}, \\
  (\mathcal{S}_{h}z,1)_{\Omega_{p}}&=& (z,1)_{\Omega_{p}},
\end{eqnarray}
and the projection operator satisfies the approximation property(cf. \cite{P.G. Ciarlet}):
\begin{equation}
  \|\mathcal{S}_{h}z-z\|_{L^{2}(\Omega_{p})}\leq Ch^{r_{6}}\|z\|_{H^{r_{6}+1}(\Omega_{p})},~~~1\leq r_{6}\leq s_{p}+1.\label{5.26}
\end{equation}

Next, we introduce the following notations
\begin{gather}\label{}
  E_{\mathbf{v}}^{n}:=\mathbf{v}(t_{n})-\mathbf{v}_{h}^{n},~~~~E_{p_{f}}^{n}:=p_{f}(t_{n})-p_{f,h}^{n},
  ~~~~E_{\mathbf{U}}^{n}:=\mathbf{U}(t_{n})-\mathbf{U}_{h}^{n},\nonumber \\
  E_{\xi}^{n}:=\xi(t_{n})-\xi_{h}^{n},~~~~E_{\eta}^{n}:=\eta(t_{n})-\eta_{h}^{n},~~~~E_{p_{p}}^{n}:=p_{p}(t_{n})-p_{p,h}^{n},
  ~~~~E_{\mathbf{q}}^{n}:=\mathbf{q}(t_{n})-\mathbf{q}_{h}^{n}.\nonumber
\end{gather}
It is easy to check that
\begin{align*}\label{}
  E_{p_{p}}^{n}=k_{1}E_{\xi}^{n}+k_{2}E_{\eta}^{n},~~~\nabla\cdot E_{\mathbf{U}}^{n}=k_{1}E_{\eta}^{n}-k_{3}E_{\xi}^{n},\\
  E_{\mathbf{q}}^{n}=-k\nabla E_{p_{p}}^{n},~~~E_{\sigma_{f}}^{n}=2\mu_{f}\mathbf{D}(E_{\mathbf{v}}^{n})-E_{p_{f}}^{n}\mathbf{I}.
\end{align*}

Also, we split the errors into two parts and denote as follows:
\begin{eqnarray*}
  &&E_{\mathbf{v}}^{n}:=\mathbf{v}(t_{n})-\mathcal{I}_{h}(\mathbf{v}(t_{n}))+\mathcal{I}_{h}(\mathbf{v}(t_{n}))-\mathbf{v}_{h}^{n}
  :=\Lambda_{\mathbf{v}}^{n}+\Theta_{\mathbf{v}}^{n}, \\
  &&E_{p_{f}}^{n}:=p_{f}(t_{n})-\mathcal{W}_{f,h}(p_{f}(t_{n}))+\mathcal{W}_{f,h}(p_{f}(t_{n}))-p_{f,h}^{n}
  :=\Lambda_{p_{f}}^{n}+\Theta_{p_{f}}^{n},\\
  &&E_{\mathbf{U}}^{n}:=\mathbf{U}(t_{n})-\mathcal{I}_{h}(\mathbf{U}(t_{n}))+\mathcal{I}_{h}(\mathbf{U}(t_{n}))-\mathbf{U}_{h}^{n}
  :=\Lambda_{\mathbf{U}}^{n}+\Theta_{\mathbf{U}}^{n}, \\
  &&E_{\xi}^{n}:=\xi(t_{n})-\mathcal{S}_{h}(\xi(t_{n}))+\mathcal{S}_{h}(\xi(t_{n}))-\xi_{h}^{n}
  :=\Lambda_{\xi}^{n}+\Theta_{\xi}^{n},\\
  &&E_{\xi}^{n}:=\xi(t_{n})-\mathcal{W}_{h}(\xi(t_{n}))+\mathcal{W}_{h}(\xi(t_{n}))-\xi_{h}^{n}
  :=\Psi_{\xi}^{n}+\Phi_{\xi}^{n},\\
  &&E_{\eta}^{n}:=\eta(t_{n})-\mathcal{S}_{h}(\eta(t_{n}))+\mathcal{S}_{h}(\eta(t_{n}))-\eta_{h}^{n}
  :=\Lambda_{\eta}^{n}+\Theta_{\eta}^{n},\\
  &&E_{\eta}^{n}:=\eta(t_{n})-\mathcal{W}_{h}(\eta(t_{n}))+\mathcal{W}_{h}(\eta(t_{n}))-\eta_{h}^{n}
  :=\Psi_{\eta}^{n}+\Phi_{\eta}^{n},\\
  &&E_{p_{p}}^{n}:=p_{p}(t_{n})-\mathcal{W}_{p,h}(p_{p}(t_{n}))+\mathcal{W}_{p,h}(p_{p}(t_{n}))-p_{p,h}^{n}
  :=\Lambda_{p_{p}}^{n}+\Theta_{p_{p}}^{n},\\
  &&E_{\mathbf{q}}^{n}:=\mathbf{q}(t_{n})-\Pi_{h}(\mathbf{q}(t_{n}))+\Pi_{h}(\mathbf{q}(t_{n}))-\mathbf{q}_{h}^{n}
  :=\Lambda_{\mathbf{q}}^{n}+\Theta_{\mathbf{q}}^{n},
\end{eqnarray*}

\begin{thrm}\label{thm5.1}
The solution of the problem (\ref{3.5})-(\ref{3.7}) satisfies the following error estimates:
\begin{eqnarray}
  && \frac{\sqrt{2}}{2}[\sqrt{c}\|\Theta_{\mathbf{U}}^{n}\|_{L^{\infty}(0,T;H^{1}(\Omega_{p}))}
  +\sqrt{k_{3}}\|\Phi_{\xi}^{n}\|_{L^{\infty}(0,T;L^{2}(\Omega_{p}))}+\sqrt{k_{2}}\|\Phi_{\eta}^{n}\|_{L^{\infty}(0,T;L^{2}(\Omega_{p}))}]\nonumber\\
  && +\sqrt{c^{p}}\|\Theta_{\mathbf{q}}^{n}\|_{L^{2}(0,T;L^{2}(\Omega_{p}))}
   +\sqrt{c}\|\Theta_{\mathbf{v}}^{n}\|_{L^{2}(0,T;H^{1}(\Omega_{f}))} \nonumber\\
  && +\|\Theta_{p_{f}}^{n}\|_{L^{2}(0,T;L^{2}(\Omega_{f}))}+\|\Theta_{p_{p}}^{n}\|_{L^{2}(0,T;L^{2}(\Omega_{p}))}
   +\|\Theta_{\xi}^{n}\|_{L^{2}(0,T;L^{2}(\Omega_{p}))}  \nonumber\\
  &&+\sqrt{\gamma_{f}\mu_{f}h^{-1}}\|(\Theta_{\mathbf{v}}^{n}-d_{t}\Theta_{\mathbf{U}}^{n}-\Theta_{\mathbf{q}}^{n})\cdot\mathbf{n}\|_{L^{2}(0,T;L^{2}(\Gamma))}\nonumber\\
  && +\sqrt{\gamma_{f}\mu_{f}h^{-1}}\|(\Theta_{\mathbf{v}}^{n}-d_{t}\Theta_{\mathbf{U}}^{n})\cdot\tau\|_{L^{2}(0,T;L^{2}(\Gamma))} \nonumber\\
  &\leq & C\sqrt{\mathrm{exp}(T)} (h^{r_{3}} \|\mathbf{v}\|_{L^{2}(0,T;H^{r_{3}+1}(\Omega_{f}))}+h^{r_{1}}\|p_{f}\|_{L^{2}(0,T;H^{r_{1}}(\Omega_{f}))}
   +h^{r_{2}}\|p_{p}\|_{L^{2}(0,T;H^{r_{2}}(\Omega_{p}))} \nonumber\\
   && +h^{r_{5}}\|\mathbf{q}\|_{L^{2}(0,T;H^{r_{5}}(\Omega_{p}))}+ h^{r_{4}}\|\mathbf{U}\|_{L^{\infty}(0,T;H^{r_{4}+1}(\Omega_{p}))}+h^{r_{4}}\|d_{t}\mathbf{U}\|_{L^{2}(0,T;H^{r_{4}+1}(\Omega_{p}))} \label{eq210723-1}\\
   && +h^{r_{6}}(\|\xi\|_{L^{2}(0,T;H^{r_{6}+1}(\Omega_{p}))}+\|\xi\|_{L^{\infty}(0,T;H^{r_{6}+1}(\Omega_{p}))}
   +\|d_{t}\xi\|_{L^{2}(0,T;H^{r_{6}+1}(\Omega_{p}))}) ) \nonumber\\
    &&+\sqrt{\frac{(1+\Delta t)\Delta t\gamma_{f}\mu_{f}}{2h\epsilon_{1}}}(\|d_{t}\mathbf{v}_{h}^{n}\cdot\tau\|_{L^{2}(0,T;L^{2}(\Gamma))}
    +\sqrt{2+\frac{\gamma'_{stab}}{1+\Delta t}}\|d_{t}\mathbf{v}_{h}^{n}\cdot\mathbf{n}\|_{L^{2}(0,T;L^{2}(\Gamma))}\nonumber\\
   && +\sqrt{(1+\frac{\gamma'_{stab}}{1+\Delta t})}\|d_{t}\mathbf{q}_{h}^{n}\cdot\mathbf{n}\|_{L^{2}(0,T;L^{2}(\Gamma))}
   +\|d_{tt}\mathbf{U}_{h}^{n}\cdot\mathbf{n}\|_{L^{2}(0,T;L^{2}(\Gamma))} )\nonumber\\
   && +\sqrt{\frac{\Delta t^{2}}{2\epsilon_{1}}+\frac{\gamma_{stab}h\Delta t}{2\gamma_{f}\mu_{f}\epsilon_{1}}}\|d_{t}p_{f,h}^{n}\|_{L^{2}(0,T;L^{2}(\Gamma))}
   +\sqrt{\frac{\Delta tk_{1}^{2}}{2\epsilon_{1}}}\|d_{t}\eta_{h}^{n}\|_{L^{2}(0,T;L^{2}(\Omega_{p}))}\nonumber\\
   &&+\sqrt{\frac{\Delta t^{2}}{3\epsilon_{1}}}\|d_{tt}\eta_{h}^{n}\|_{L^{2}(0,T;H^{1}(\Omega_{p})')}+\sqrt{\frac{\mu_{f}}{\epsilon_{1}}\Delta t^{2}C}\|d_{t}\mathbf{v}_{h}^{n}\|_{L^{2}(0,T;H^{1}(\Omega_{f}))},\nonumber
\end{eqnarray}
where
\begin{gather*}
  0\leq r_{1}\leq s_{f}+1,~~~0\leq r_{2}\leq s_{p}+1,~~~0\leq r_{3}\leq k_{f}, \\
  0\leq r_{4}\leq k_{p},~~~1\leq r_{5}\leq k_{s}+1,~~~1\leq r_{6}\leq s_{p}+1.
\end{gather*}
\end{thrm}
\begin{proof}
Subtracting (\ref{3.5})-(\ref{3.7}) from (\ref{2.15})-(\ref{2.20}), summing the equations, 
using the definition of the projection operators~$\mathcal{W}_{f,h},~\mathcal{W}_{p,h},~\mathcal{I}_{h},~\Pi_{h},~\mathcal{S}_{h},~\mathbf{I}_{h}$ and setting $
	\varphi_{f,h}=\Theta_{\mathbf{v}}^{n}$, $\psi_{f,h}=\Theta_{p_{f}}^{n}$, $\varphi_{p,h}=d_{t}\Theta_{\mathbf{U}}^{n}$, 
	$w_{h}=\Phi_{\xi}^{n}$, $ \mathbf{r}_{h}=\Theta_{\mathbf{q}}^{n}$, $z_{h}=\Theta_{p_{p}}^{n}=k_{1}\Phi_{\xi}^{n}+k_{2}\Phi_{\eta}^{n}$,
we have
\begin{eqnarray}\label{6.23}
   && 2\mu_{p}(\mathbf{D}(\Lambda_{\mathbf{U}}^{n}),\mathbf{D}(\varphi_{p,h}))_{\Omega_{p}}+2\mu_{p}(\mathbf{D}(\Theta_{\mathbf{U}}^{n}),\mathbf{D}(\varphi_{p,h}))_{\Omega_{p}}
   -(\Psi_{\xi}^{n},\nabla\cdot\varphi_{p,h})_{\Omega_{p}}\nonumber\\
   &&-(\Phi_{\xi}^{n},\nabla\cdot\varphi_{p,h})_{\Omega_{p}}+(\nabla\cdot d_t\Theta_{\mathbf{U}}^{n},w_{h})_{\Omega_{p}} 
   +k_{3}(d_t\Phi_{\xi}^{n},w_{h})_{\Omega_{p}}-k_{1}(d_t\Phi_{\eta}^{n},w_{h})_{\Omega_{p}}\nonumber\\
   && -(\Theta_{p_{p}}^{n},\nabla\cdot\mathbf{r}_{h})_{\Omega_{p}}+k^{-1}(\Lambda_{\mathbf{q}}^{n},\mathbf{r}_{h})_{\Omega_{p}}
   +k^{-1}(\Theta_{\mathbf{q}}^{n},\mathbf{r}_{h})_{\Omega_{p}}\nonumber\\
   && +(d_{t}\Phi_{\eta}^{n},z_{h})_{\Omega_{p}}+(\nabla\cdot\Theta_{\mathbf{q}}^{n},z_{h})_{\Omega_{p}}
   +(2\mu_{f}\mathbf{D}(\Lambda_{\mathbf{v}}^{n}),\mathbf{D}(\varphi_{f,h}))_{\Omega_{f}}
   +(2\mu_{f}\mathbf{D}(\Theta_{\mathbf{v}}^{n}),\mathbf{D}(\varphi_{f,h}))_{\Omega_{f}}\nonumber\\
   &&-(\Lambda_{p_{f}}^{n},\nabla\cdot\varphi_{f,h})_{\Omega_{f}}-(\Theta_{p_{f}}^{n},\nabla\cdot\varphi_{f,h})_{\Omega_{f}}
   +(\nabla\cdot \Theta_{\mathbf{v}}^{n},\psi_{f,h})_{\Omega_{f}}\nonumber\\
   && -s_{f,p}(d_{t}p_{f,h}^{n},\psi_{f,h})
   -s_{f,q}(d_{t}\mathbf{q}^{n}_{h}\cdot\mathbf{n},\mathbf{r}_{h}\cdot\mathbf{n})-
   s_{f,v}(d_{t}\mathbf{v}^{n}_{h}\cdot\mathbf{n},\varphi_{f,h}\cdot\mathbf{n})\nonumber\\
   &&-\int_{\Gamma}\mathbf{n}\cdot \Theta_{\sigma_{f}}^{n}\mathbf{n}(\varphi_{f,h}-\varphi_{p,h}-\mathbf{r}_{h})\cdot\mathbf{n}~dx \nonumber\\ &&-\int_{\Gamma}\mathbf{n}\cdot\sigma_{f,h}(\varsigma\varphi_{f,h},-\psi_{f,h})\mathbf{n}(\Theta_{\mathbf{v}}^{n}-d_{t}\Theta_{\mathbf{U}}^{n}-\Theta_{\mathbf{q}}^{n})\cdot\mathbf{n}~dx\nonumber\\
   &&-\int_{\Gamma}\tau\cdot \Theta_{\sigma_{f}}^{n}\mathbf{n}(\varphi_{f,h}-\varphi_{p,h})\cdot\tau~dx -\int_{\Gamma}\tau\cdot\sigma_{f,h}(\varsigma\varphi_{f,h},-\psi_{f,h})\mathbf{n}(\Theta_{\mathbf{v}}^{n}-d_{t}\Theta_{\mathbf{U}}^{n})\cdot\tau~dx\nonumber\\
   &&-\int_{\Gamma}\mathbf{n}\cdot \Lambda_{\sigma_{f}}^{n}\mathbf{n}(\varphi_{f,h}-\varphi_{p,h}-\mathbf{r}_{h})\cdot\mathbf{n}~dx
   -\int_{\Gamma}\tau\cdot \Lambda_{\sigma_{f}}^{n}\mathbf{n}(\varphi_{f,h}-\varphi_{p,h})\cdot\tau~dx \\
   &&+\int_{\Gamma}\gamma_{f}\mu_{f}h^{-1}(\Theta_{\mathbf{v}}^{n}-d_{t}\Theta_{\mathbf{U}}^{n}-\Theta_{\mathbf{q}}^{n})\cdot\mathbf{n}(\varphi_{f,h}-\varphi_{p,h}-\mathbf{r}_{h})\cdot\mathbf{n}~dx\nonumber\\
   &&+\int_{\Gamma}\gamma_{f}\mu_{f}h^{-1}(\Theta_{\mathbf{v}}^{n}-d_{t}\Theta_{\mathbf{U}}^{n})\cdot\tau(\varphi_{f,h}-\varphi_{p,h})\cdot\tau~dx\nonumber\\
   &=&
   -\int_{\Gamma}\gamma_{f}\mu_{f}h^{-1}((\mathbf{v}_{h}^{n}-\mathbf{v}_{h}^{n-1})-d_{t}(\mathbf{U}_{h}^{n}-\mathbf{U}_{h}^{n-1}))\cdot\mathbf{n}(-\mathbf{r}_{h})\cdot\mathbf{n}~dx\nonumber\\
   &&-\int_{\Gamma}\gamma_{f}\mu_{f}h^{-1}((\mathbf{v}_{h}^{n}-\mathbf{v}_{h}^{n-1})-(\mathbf{q}_{h}^{n}-\mathbf{q}_{h}^{n-1}))
   \cdot\mathbf{n}(-\varphi_{p,h})\cdot\mathbf{n}~dx\nonumber\\
   &&-\int_{\Gamma}\gamma_{f}\mu_{f}h^{-1}(\mathbf{v}_{h}^{n}-\mathbf{v}_{h}^{n-1})\cdot\tau(-\varphi_{p,h})\cdot\tau~dx\nonumber\\
   &&+\int_{\Omega_{p}} k_{1}(\eta_{h}^{n}-\eta_{h}^{n-1})w_{h}~d\mathbf{x}+\int_{\Omega_{p}}\mathbf{R}_{h}^{n}\cdot z_{h}~d\mathbf{x}\nonumber\\
   &&+\int_{\Gamma}2\mu_{f}(\mathbf{D}(\mathbf{v}_{h}^{n})-\mathbf{D}(\mathbf{v}_{h}^{n-1}))\mathbf{n}\cdot [\mathbf{n}(\varphi_{f,h}-\mathbf{r}_{h}-\varphi_{p,h})\cdot\mathbf{n}+\tau(\varphi_{f,h}-\varphi_{p,h})\cdot\tau]~dx\nonumber\\
   &&-\int_{\Gamma}\mathbf{n}\cdot(p_{f,h}^{n}-p_{f,h}^{n-1})\mathbf{n}(\varphi_{f,h}-\mathbf{r}_{h}-\varphi_{p,h})\cdot\mathbf{n}~dx,\nonumber
\end{eqnarray}
where $
\mathbf{R}_{h}^{n}:=-\frac{1}{\Delta t}\int_{t_{n-1}}^{t_{n}}(t-t_{n-1})d_{tt}\eta_{h}(t)~dt$.

Integrating (\ref{6.23}) over $[0, t]$, we get
\begin{eqnarray}\label{6.24}
   &&\varepsilon_{h}^{t}+c^{p}\|\Theta_{\mathbf{q}}^{n}\|^{2}_{L^{2}(0,t;L^{2}(\Omega_{p}))}+c\|\Theta_{\mathbf{v}}^{n}\|^{2}_{L^{2}(0,t;H^{1}(\Omega_{f}))}\nonumber\\
   &&+\gamma_{f}\mu_{f}h^{-1}\|(\Theta_{\mathbf{v}}^{n}-d_{t}\Theta_{\mathbf{U}}^{n}-\Theta_{\mathbf{q}}^{n})\cdot\mathbf{n}\|^{2}_{L^{2}(0,t;L^{2}(\Gamma))}
   +\gamma_{f}\mu_{f}h^{-1}\|(\Theta_{\mathbf{v}}^{n}-d_{t}\Theta_{\mathbf{U}}^{n})\cdot\tau\|^{2}_{L^{2}(0,t;L^{2}(\Gamma))}\nonumber\\
   &&+\frac{\Delta t}{2}(c\|d_{t}\Theta_{\mathbf{U}}^{n}\|^{2}_{L^{2}(0,t;H^{1}(\Omega_{p}))}+k_{3}\|d_{t}\Phi_{\xi}^{n}\|^{2}_{L^{2}(0,t;L^{2}(\Omega_{p}))}
   +k_{2}\|d_{t}\Phi_{\eta}^{n}\|^{2}_{L^{2}(0,t;L^{2}(\Omega_{p}))})\nonumber\\
   &\leq& \varepsilon_{h}^{0} +\int_{0}^{t}(s_{f,p}(d_{t}p_{f,h}^{n},\Theta_{p_{f}}^{n})
   +s_{f,q}(d_{t}\mathbf{q}^{n}_{h}\cdot\mathbf{n},\Theta_{\mathbf{q}}^{n}\cdot\mathbf{n})+
   s_{f,v}(d_{t}\mathbf{v}^{n}_{h}\cdot\mathbf{n},\Theta_{\mathbf{v}}^{n}\cdot\mathbf{n}))~dt\nonumber\\
   &&+\int_{0}^{t}((2\mu_{f}\mathbf{D}(\Lambda_{\mathbf{v}}^{n}),\mathbf{D}(\Theta_{\mathbf{v}}^{n}))_{\Omega_{f}}
   +(\Lambda_{p_{f}}^{n},\nabla\cdot\Theta_{\mathbf{v}}^{n})_{\Omega_{f}}+k^{-1}(\Lambda_{\mathbf{q}}^{n},\Theta_{\mathbf{q}}^{n})_{\Omega_{p}})~dt \nonumber\\
   && +\int_{0}^{t}(2\mu_{p}(\mathbf{D}(\Lambda_{\mathbf{U}}^{n}),d_{t}\mathbf{D}(\Theta_{\mathbf{U}}^{n}))_{\Omega_{p}}
   +(\Lambda_{\xi}^{n},d_{t}\nabla\cdot\Theta_{\mathbf{U}}^{n})_{\Omega_{p}})~dt \nonumber\\
   &&+\int_{0}^{t}\int_{\Gamma}\mathbf{n}\cdot \Theta_{\sigma_{f}}^{n}\mathbf{n}(\Theta_{\mathbf{v}}^{n}-d_{t}\Theta_{\mathbf{U}}^{n}-\Theta_{\mathbf{q}}^{n})\cdot\mathbf{n}~dxdt \nonumber\\ &&+\int_{0}^{t}\int_{\Gamma}\mathbf{n}\cdot\sigma_{f,h}(\varsigma\Theta_{\mathbf{v}}^{n},-\Theta_{p_{f}}^{n})\mathbf{n}(\Theta_{\mathbf{v}}^{n}
   -d_{t}\Theta_{\mathbf{U}}^{n}-\Theta_{\mathbf{q}}^{n})\cdot\mathbf{n}~dxdt\\
   &&+\int_{0}^{t}\int_{\Gamma}\mathbf{n}\cdot \Lambda_{\sigma_{f}}^{n}\mathbf{n}(\Theta_{\mathbf{v}}^{n}-d_{t}\Theta_{\mathbf{U}}^{n}-\Theta_{\mathbf{q}}^{n})\cdot\mathbf{n}~dxdt
   +\int_{0}^{t}\int_{\Gamma}\tau\cdot \Lambda_{\sigma_{f}}^{n}\mathbf{n}(\Theta_{\mathbf{v}}^{n}-d_{t}\Theta_{\mathbf{U}}^{n})\cdot\tau~dxdt \nonumber\\
   &&+\int_{0}^{t}\int_{\Gamma}\tau\cdot \Theta_{\sigma_{f}}^{n}\mathbf{n}(\Theta_{\mathbf{v}}^{n}-d_{t}\Theta_{\mathbf{U}}^{n})\cdot\tau~dxdt +\int_{0}^{t}\int_{\Gamma}\tau\cdot\sigma_{f,h}(\varsigma\Theta_{\mathbf{v}}^{n},-\Theta_{p_{f}}^{n})\mathbf{n}(\Theta_{\mathbf{v}}^{n}
   -d_{t}\Theta_{\mathbf{U}}^{n})\cdot\tau~dxdt\nonumber\\
   && - \int_{0}^{t}\int_{\Gamma}\gamma_{f}\mu_{f}h^{-1}(\mathbf{v}_{h}^{n}-\mathbf{v}_{h}^{n-1})\cdot\tau(-d_{t}\Theta_{\mathbf{U}}^{n})\cdot\tau~dxdt \nonumber\\
   && -\int_{0}^{t}\int_{\Gamma}\gamma_{f}\mu_{f}h^{-1}((\mathbf{v}_{h}^{n}-\mathbf{v}_{h}^{n-1})
   -d_{t}(\mathbf{U}_{h}^{n}-\mathbf{U}_{h}^{n-1}))\cdot\mathbf{n}(-\Theta_{\mathbf{q}}^{n})\cdot\mathbf{n}~dxdt \nonumber\\
   &&-\int_{0}^{t}\int_{\Gamma}\gamma_{f}\mu_{f}h^{-1}((\mathbf{v}_{h}^{n}-\mathbf{v}_{h}^{n-1})-(\mathbf{q}_{h}^{n}-\mathbf{q}_{h}^{n-1}))
   \cdot\mathbf{n}(-d_{t}\Theta_{\mathbf{U}}^{n})\cdot\mathbf{n}~dxdt \nonumber\\
   &&+\int_{0}^{t}\int_{\Omega_{p}} k_{1}(\eta_{h}^{n}-\eta_{h}^{n-1})d_{t}\Theta_{\xi}^{n}~d\mathbf{x}dt
   +\int_{0}^{t}\int_{\Omega_{p}}\mathbf{R}_{h}^{n}\cdot\Theta_{p_{p}}^{n}~d\mathbf{x}dt \nonumber\\
   &&+\int_{0}^{t}\int_{\Gamma}2\mu_{f}(\mathbf{D}(\mathbf{v}_{h}^{n})-\mathbf{D}(\mathbf{v}_{h}^{n-1}))\mathbf{n}\cdot [\mathbf{n}(\Theta_{\mathbf{v}}^{n}-d_{t}\Theta_{\mathbf{U}}^{n}-\Theta_{\mathbf{q}}^{n})\cdot\mathbf{n}
   +\tau(\Theta_{\mathbf{v}}^{n}-d_{t}\Theta_{\mathbf{U}}^{n})\cdot\tau] ~dxdt \nonumber\\
   &&-\int_{0}^{t}\int_{\Gamma}\mathbf{n}\cdot(p_{f,h}^{n}-p_{f,h}^{n-1})\mathbf{n}(\Theta_{\mathbf{v}}^{n}-d_{t}\Theta_{\mathbf{U}}^{n}-\Theta_{\mathbf{q}}^{n})\cdot\mathbf{n}~dxdt,\nonumber
\end{eqnarray}
where  $\Theta_{p_{p}}^{n}:=k_{1}\Phi_{\xi}^{n}+k_{2}\Phi_{\eta}^{n}$ and $\varepsilon_{h}^{t}:=\frac{1}{2}[c\|\Theta_{\mathbf{U}}^{t}\|_{H^{1}(\Omega_{p})}^{2}+k_{3}\|\Phi_{\xi}^{t}\|_{L^{2}(\Omega_{p})}^{2}
	+k_{2}\|\Phi_{\eta}^{t}\|_{L^{2}(\Omega_{p})}^{2}]$.

Using Cauchy-Schwarz inequality and Young inequality, we have
\begin{eqnarray}\label{6.25}
   && (2\mu_{f}\mathbf{D}(\Lambda_{\mathbf{v}}^{n}),\mathbf{D}(\Theta_{\mathbf{v}}^{n}))_{\Omega_{f}}
   +(\Lambda_{p_{f}}^{n},\nabla\cdot\Theta_{\mathbf{v}}^{n})_{\Omega_{f}}+k^{-1}(\Lambda_{\mathbf{q}}^{n},\Theta_{\mathbf{q}}^{n})_{\Omega_{p}}\nonumber \\
   &\leq&C\epsilon_{1}^{-1}(\|\Lambda_{\mathbf{v}}^{n}\|^{2}_{H^{1}(\Omega_{f})}+\|\Lambda_{p_{f}}^{n}\|^{2}_{L^{2}(\Omega_{f})}
   +\|\Lambda_{\mathbf{q}}^{n}\|^{2}_{L^{2}(\Omega_{p})})\nonumber \\
   && +\epsilon_{1}(\|\Theta_{\mathbf{v}}^{n}\|^{2}_{H^{1}(\Omega_{f})}+\|\nabla\cdot\Theta_{\mathbf{v}}^{n}\|^{2}_{L^{2}(\Omega_{f})}
   +\|\Theta_{\mathbf{q}}^{n}\|^{2}_{L^{2}(\Omega_{p})})  \\
   &\leq&C\epsilon_{1}^{-1}(\|\Lambda_{\mathbf{v}}^{n}\|^{2}_{H^{1}(\Omega_{f})}+\|\Lambda_{p_{f}}^{n}\|^{2}_{L^{2}(\Omega_{f})}
   +\|\Lambda_{\mathbf{q}}^{n}\|^{2}_{L^{2}(\Omega_{p})})\nonumber \\
   && +\epsilon_{1}(\|\Theta_{\mathbf{v}}^{n}\|^{2}_{H^{1}(\Omega_{f})}+\|\Theta_{\mathbf{v}}^{n}\|^{2}_{H^{1}(\Omega_{f})}
   +\|\Theta_{\mathbf{q}}^{n}\|^{2}_{L^{2}(\Omega_{p})}).\nonumber
\end{eqnarray}
Using integration by parts in time, we obtain
\begin{eqnarray}\label{6.27}
  &&\int_{0}^{t}2\mu_{p}(\mathbf{D}(\Lambda_{\mathbf{U}}^{n}),d_{t}\mathbf{D}(\Theta_{\mathbf{U}}^{n}))_{\Omega_{p}}~dt
   +\int_{0}^{t}(\Lambda_{\xi}^{n},d_{t}\nabla\cdot\Theta_{\mathbf{U}}^{n})_{\Omega_{p}}~dt \nonumber\\
  &\leq&  C(\epsilon_{1}^{-1}\|\Lambda_{\mathbf{U}}^{t}\|^{2}_{H^{1}(\Omega_{p})}+\|d_{t}\Lambda_{\mathbf{U}}^{n}\|^{2}_{L^{2}(0,t;H^{1}(\Omega_{p}))}  +\epsilon_{1}^{-1}\|\Lambda_{\xi}^{t}\|^{2}_{L^{2}(\Omega_{p})} \nonumber\\
  && +\|d_{t}\Lambda_{\xi}^{n}\|^{2}_{L^{2}(0,t;L^{2}(\Omega_{p}))}  )(\epsilon_{1}\|\Theta_{\mathbf{U}}^{t}\|^{2}_{H^{1}(\Omega_{p})}+\|\Theta_{\mathbf{U}}^{n}\|^{2}_{L^{2}(0,t;H^{1}(\Omega_{p}))}).\nonumber
\end{eqnarray}
the interface terms can be bounded as follows,
\begin{eqnarray}\label{6.28}
  && \int_{\Gamma}\mathbf{n}\cdot \Theta_{\sigma_{f}}^{n}\mathbf{n}(\Theta_{\mathbf{v}}^{n}
  -d_{t}\Theta_{\mathbf{U}}^{n}-\Theta_{\mathbf{q}}^{n})\cdot\mathbf{n}~dx \nonumber\\
  &&+\int_{\Gamma}\mathbf{n}\cdot\sigma_{f,h}(\varsigma\Theta_{\mathbf{v}}^{n},-\Theta_{p_{f}}^{n})\mathbf{n}(\Theta_{\mathbf{v}}^{n}
   -d_{t}\Theta_{\mathbf{U}}^{n}-\Theta_{\mathbf{q}}^{n})\cdot\mathbf{n}~dx  \nonumber\\
  &=& (1+\varsigma)\int_{\Gamma}\mathbf{n}\cdot (2\mu_{f}\mathbf{D}(\Theta_{\mathbf{v}}^{n}))\mathbf{n}(\Theta_{\mathbf{v}}^{n}
  -d_{t}\Theta_{\mathbf{U}}^{n}-\Theta_{\mathbf{q}}^{n})\cdot\mathbf{n}~dx \\
  &\leq&  2\mu_{f}(1+\varsigma)\|\mathbf{D}(\Theta_{\mathbf{v}}^{n})\cdot\mathbf{n}\|_{L^{2}(\Gamma)}\|(\Theta_{\mathbf{v}}^{n}
  -d_{t}\Theta_{\mathbf{U}}^{n}-\Theta_{\mathbf{q}}^{n})\cdot\mathbf{n}\|_{L^{2}(\Gamma)} \nonumber\\
  &\leq& \mu_{f}(1+\varsigma)\epsilon_{1}C_{TI}C\|\Theta_{\mathbf{v}}^{n}\|_{H^{1}(\Omega_{f})}^{2}
  +\mu_{f}(1+\varsigma)(\epsilon_{1}h)^{-1}\|(\Theta_{\mathbf{v}}^{n}
  -d_{t}\Theta_{\mathbf{U}}^{n}-\Theta_{\mathbf{q}}^{n})\cdot\mathbf{n}\|_{L^{2}(\Gamma)}^{2}.\nonumber
\end{eqnarray}
As for the interface terms, we have
\begin{eqnarray}\label{6.29}
\begin{aligned}
  & \int_{\Gamma}\tau\cdot \Theta_{\sigma_{f}}^{n}\mathbf{n}(\Theta_{\mathbf{v}}^{n}
  -d_{t}\Theta_{\mathbf{U}}^{n})\cdot\tau~dx
  +\int_{\Gamma}\tau\cdot\sigma_{f,h}(\varsigma\Theta_{\mathbf{v}}^{n},-\Theta_{p_{f}}^{n})\mathbf{n}(\Theta_{\mathbf{v}}^{n}
   -d_{t}\Theta_{\mathbf{U}}^{n})\cdot\tau~dx  \\
  \leq& \mu_{f}(1+\varsigma)\epsilon_{1}C_{TI}C\|\Theta_{\mathbf{v}}^{n}\|_{H^{1}(\Omega_{f})}^{2}
  +\mu_{f}(1+\varsigma)(\epsilon_{1}h)^{-1}\|(\Theta_{\mathbf{v}}^{n}
  -d_{t}\Theta_{\mathbf{U}}^{n})\cdot\tau\|_{L^{2}(\Gamma)}^{2}.
\end{aligned}
\end{eqnarray}
\begin{eqnarray}\label{6.30}
\begin{aligned}
  &\int_{\Gamma}\mathbf{n}\cdot \Lambda_{\sigma_{f}}^{n}\mathbf{n}(\Theta_{\mathbf{v}}^{n}
  -d_{t}\Theta_{\mathbf{U}}^{n}-\Theta_{\mathbf{q}}^{n})\cdot\mathbf{n}~dx \\
  \leq& \mu_{f}\epsilon_{1}C_{TI}C\|\Lambda_{\mathbf{v}}^{n}\|_{H^{1}(\Omega_{f})}^{2}
  +\mu_{f}(\epsilon_{1}h)^{-1}\|(\Theta_{\mathbf{v}}^{n}-d_{t}\Theta_{\mathbf{U}}^{n}-\Theta_{\mathbf{q}}^{n})\cdot\mathbf{n}\|_{L^{2}(\Gamma)}^{2} \\
  & + C\epsilon_{1}\|\Lambda_{p_{f}}^{n}\|_{L^{2}(\Omega_{f})}^{2}
  +\epsilon_{1}^{-1}\|(\Theta_{\mathbf{v}}^{n}-d_{t}\Theta_{\mathbf{U}}^{n}-\Theta_{\mathbf{q}}^{n})\cdot\mathbf{n}\|_{L^{2}(\Gamma)}^{2},
\end{aligned}
\end{eqnarray}
\begin{eqnarray}\label{6.31}
\begin{aligned}
  & \int_{\Gamma}\tau\cdot \Lambda_{\sigma_{f}}^{n}\mathbf{n}(\Theta_{\mathbf{v}}^{n}-d_{t}\Theta_{\mathbf{U}}^{n})\cdot\tau~dx \\
  \leq& \mu_{f}\epsilon_{1}C_{TI}C\|\Lambda_{\mathbf{v}}^{n}\|_{H^{1}(\Omega_{f})}^{2}
  +\mu_{f}(\epsilon_{1}h)^{-1}\|(\Theta_{\mathbf{v}}^{n}-d_{t}\Theta_{\mathbf{U}}^{n})\cdot\tau\|_{L^{2}(\Gamma)}^{2} \\
  &  + C\epsilon_{1}\|\Lambda_{p_{f}}^{n}\|_{L^{2}(\Omega_{f})}^{2}
  +\epsilon_{1}^{-1}\|(\Theta_{\mathbf{v}}^{n}-d_{t}\Theta_{\mathbf{U}}^{n})\cdot\tau\|_{L^{2}(\Gamma)}^{2},
\end{aligned}
\end{eqnarray}
\begin{eqnarray}\label{6.32}
\begin{aligned}
  &\int_{0}^{t}\int_{\Gamma}\gamma_{f}\mu_{f}h^{-1}(\mathbf{v}_{h}^{n}-\mathbf{v}_{h}^{n-1})\cdot\tau(-d_{t}\Theta_{\mathbf{U}}^{n})\cdot\tau~dxdt \\
  \leq& \frac{1}{2}\gamma_{f}\mu_{f}h^{-1}[\epsilon_{1}^{-1}(1+\Delta t^{-1})
  \|(\mathbf{v}_{h}^{n}-\mathbf{v}_{h}^{n-1})\cdot\tau\|_{L^{2}(0,t;L^{2}(\Gamma))}^{2} \\
   & +\epsilon_{1}(\|(\Theta_{\mathbf{v}}^{n}-d_{t}\Theta_{\mathbf{U}}^{n})\cdot\tau\|_{L^{2}(0,t;L^{2}(\Gamma))}^{2}
   +\Delta t\|\Theta_{\mathbf{v}}^{n}\cdot\tau\|_{L^{2}(0,t;L^{2}(\Gamma))}^{2})] \\
   \leq& \frac{1}{2}\gamma_{f}\mu_{f}h^{-1}[\epsilon_{1}^{-1}(1+\Delta t^{-1})\Delta t^{2}
  \|d_{t}\mathbf{v}_{h}^{n}\cdot\tau\|_{L^{2}(0,t;L^{2}(\Gamma))}^{2} \\
   & +\epsilon_{1}(\|(\Theta_{\mathbf{v}}^{n}-d_{t}\Theta_{\mathbf{U}}^{n})\cdot\tau\|_{L^{2}(0,t;L^{2}(\Gamma))}^{2}
   +\Delta tC\|\Theta_{\mathbf{v}}^{n}\|_{L^{2}(0,t;H^{1}(\Omega_{f}))}^{2})],
\end{aligned}
\end{eqnarray}
\begin{eqnarray}\label{6.33}
\begin{aligned}
  &\int_{0}^{t}\int_{\Gamma}\gamma_{f}\mu_{f}h^{-1}((\mathbf{v}_{h}^{n}-\mathbf{v}_{h}^{n-1})
   -d_{t}(\mathbf{U}_{h}^{n}-\mathbf{U}_{h}^{n-1}))\cdot\mathbf{n}(-\Theta_{\mathbf{q}}^{n})\cdot\mathbf{n}~dxdt \\
  \leq&\frac{1}{2}\gamma_{f}\mu_{f}h^{-1}
  [\epsilon^{-1}(1+\Delta t^{-1})\int_{0}^{t}(\|(\mathbf{v}_{h}^{n}-\mathbf{v}_{h}^{n-1})\cdot\mathbf{n}\|_{L^{2}(\Gamma)}^{2} \\
   & +\|d_{t}(\mathbf{U}_{h}^{n}-\mathbf{U}_{h}^{n-1})\cdot\mathbf{n}\|_{L^{2}(\Gamma)}^{2})~dt+\epsilon\Delta t\int_{0}^{t}\|d_{t}\Theta_{\mathbf{U}}^{n}\cdot\mathbf{n}\|_{L^{2}(\Gamma)}^{2}~dt] \\
   \leq& \frac{1}{2}\gamma_{f}\mu_{f}(\epsilon_{1}h)^{-1}(1+\Delta t^{-1})\Delta t^{2}(\|d_{t}\mathbf{v}_{h}^{n}\cdot\mathbf{n}\|^{2}_{L^{2}(0,t;L^{2}(\Gamma))}
   +\|d_{tt}\mathbf{U}_{h}^{n}\cdot\mathbf{n}\|_{L^{2}(0,t;L^{2}(\Gamma))}^{2})  \\
   &+\frac{1}{2}\gamma_{f}\mu_{f}h^{-1}\epsilon_{1}\Delta t\|d_{t}\Theta_{\mathbf{U}}^{n}\|_{L^{2}(0,t;H^{1}(\Omega_{p}))}^{2},
\end{aligned}
\end{eqnarray}
\begin{eqnarray} \label{6.34}
\begin{aligned}
  &\int_{0}^{t}\int_{\Gamma}\gamma_{f}\mu_{f}h^{-1}((\mathbf{v}_{h}^{n}-\mathbf{v}_{h}^{n-1})-(\mathbf{q}_{h}^{n}-\mathbf{q}_{h}^{n-1}))
   \cdot\mathbf{n}(-d_{t}\Theta_{\mathbf{U}}^{n})\cdot\mathbf{n}~dxdt \\
  \leq&\frac{1}{2}\gamma_{f}\mu_{f}h^{-1}
  [\epsilon_{1}^{-1}(1+\Delta t^{-1})\int_{0}^{t}(\|(\mathbf{v}_{h}^{n}-\mathbf{v}_{h}^{n-1})\cdot\mathbf{n}\|_{L^{2}(\Gamma)}^{2}
   + \|(\mathbf{q}_{h}^{n}-\mathbf{q}_{h}^{n-1})\cdot\mathbf{n}\|_{L^{2}(\Gamma)}^{2})~dt \\
   &+\epsilon_{1}\int_{0}^{t}(\|(\Theta_{\mathbf{v}}^{n}-\Theta_{\mathbf{q}}^{n}-d_{t}\Theta_{\mathbf{U}}^{n})\cdot\mathbf{n}\|_{L^{2}(\Gamma)}^{2}
    +\Delta t\|\Theta_{\mathbf{v}}^{n}\cdot\mathbf{n}\|_{L^{2}(\Gamma)}^{2}+\Delta t\|\Theta_{\mathbf{q}}^{n}\cdot\mathbf{n}\|_{L^{2}(\Gamma)}^{2})~dt] \\
   \leq& \frac{1}{2}\gamma_{f}\mu_{f}h^{-1}
  [\epsilon_{1}^{-1}(1+\Delta t^{-1})\Delta t^{2}(\|d_{t}\mathbf{v}_{h}^{n}\cdot\mathbf{n}\|^{2}_{L^{2}(0,t;L^{2}(\Gamma))}
    +\|d_{t}\mathbf{q}_{h}^{n}\cdot\mathbf{n}\|^{2}_{L^{2}(0,t;L^{2}(\Gamma))})
      \\
   &  +C\Delta t\|\Theta_{\mathbf{v}}^{n}\|_{L^{2}(0,t;H^{1}(\Omega_{f}))}^{2}+\epsilon_{1}(\|(\Theta_{\mathbf{v}}^{n}-\Theta_{\mathbf{q}}^{n}-d_{t}\Theta_{\mathbf{U}}^{n})\cdot\mathbf{n}\|_{L^{2}(0,t;L^{2}(\Gamma))}^{2}\\
   &+
  \frac{\mathrm{C}_{\mathrm{TI}}\Delta t}{h}\|\Theta_{\mathbf{q}}^{n}\|_{L^{2}(0,t;L^{2}(\Omega_{p}))}^{2})],
\end{aligned}
\end{eqnarray}
\begin{eqnarray}\label{6.35}
\begin{aligned}
  &\int_{0}^{t}\int_{\Gamma}2\mu_{f}(\mathbf{D}(\mathbf{v}_{h}^{n})-\mathbf{D}(\mathbf{v}_{h}^{n-1}))\mathbf{n}\cdot [\mathbf{n}(\Theta_{\mathbf{v}}^{n}-d_{t}\Theta_{\mathbf{U}}^{n}-\Theta_{\mathbf{q}}^{n})\cdot\mathbf{n}
   +\tau(\Theta_{\mathbf{v}}^{n}-d_{t}\Theta_{\mathbf{U}}^{n})\cdot\tau] ~dxdt  \\
  \leq& \frac{\mu_{f}}{\epsilon_{1}}\int_{0}^{t}\|\mathbf{D}(\mathbf{v}_{h}^{n}-\mathbf{v}_{h}^{n-1})\|_{L^{2}(\Gamma)}^{2}~dt \\
  & +\epsilon_{1}\mu_{f}(\|(\Theta_{\mathbf{v}}^{n}-\Theta_{\mathbf{q}}^{n}-d_{t}\Theta_{\mathbf{U}}^{n})\cdot\mathbf{n}\|_{L^{2}(0,t;L^{2}(\Gamma))}^{2}
  +\|(\Theta_{\mathbf{v}}^{n}-d_{t}\Theta_{\mathbf{U}}^{n})\cdot\tau\|_{L^{2}(0,t;L^{2}(\Gamma))}^{2}) \\
   \leq& \frac{\mu_{f}}{\epsilon_{1}}\Delta t^{2}C\|d_{t}\mathbf{v}_{h}^{n}\|_{L^{2}(0,t;H^{1}(\Omega_{f}))}^{2} \\
  & +\epsilon_{1}\mu_{f}(\|(\Theta_{\mathbf{v}}^{n}-\Theta_{\mathbf{q}}^{n}-d_{t}\Theta_{\mathbf{U}}^{n})\cdot\mathbf{n}\|_{L^{2}(0,t;L^{2}(\Gamma))}^{2}
  +\|(\Theta_{\mathbf{v}}^{n}-d_{t}\Theta_{\mathbf{U}}^{n})\cdot\tau\|_{L^{2}(0,t;L^{2}(\Gamma))}^{2}),
\end{aligned}
\end{eqnarray}
\begin{eqnarray}\label{6.36}
\begin{aligned}
  & \int_{0}^{t}\int_{\Gamma}\mathbf{n}\cdot(p_{f,h}^{n}-p_{f,h}^{n-1})\mathbf{n}(\Theta_{\mathbf{v}}^{n}
  -d_{t}\Theta_{\mathbf{U}}^{n}-\Theta_{\mathbf{q}}^{n})\cdot\mathbf{n}~dxdt \\
   \leq& \frac{\Delta t^{2}}{2\epsilon_{1}}\|d_{t}p_{f,h}^{n}\|^{2}_{L^{2}(0,t;L^{2}(\Gamma))}
   +\frac{\epsilon_{1}}{2}\|(\Theta_{\mathbf{v}}^{n}-\Theta_{\mathbf{q}}^{n}-d_{t}\Theta_{\mathbf{U}}^{n})\cdot\mathbf{n}\|_{L^{2}(0,t;L^{2}(\Gamma))}^{2},
\end{aligned}
\end{eqnarray}
\begin{eqnarray}\label{6.37}
\begin{aligned}
  & \int_{0}^{t}\int_{\Omega_{p}} k_{1}(\eta_{h}^{n}-\eta_{h}^{n-1})d_{t}\Theta_{\xi}^{n}~d\mathbf{x}dt \\
  \leq& \frac{\Delta tk_{1}^{2}}{2\epsilon_{1}}\|d_{t}\eta_{h}^{n}\|^{2}_{L^{2}(0,t;L^{2}(\Omega_{p}))}
  +\frac{\epsilon_{1}\Delta t}{2} \|d_{t}\Theta_{\xi}^{n}\|^{2}_{L^{2}(0,t;L^{2}(\Omega_{p}))},
\end{aligned}
\end{eqnarray}
\begin{eqnarray}\label{6.38}
\begin{aligned}
  & \int_{0}^{t}\int_{\Omega_{p}}\mathbf{R}_{h}^{n}\cdot\Theta_{p_{p}}^{n}~d\mathbf{x}dt \\
  \leq& \int_{0}^{t}\|\mathbf{R}_{h}^{n}\|_{H^{1}(\Omega_{p})'}\|\nabla\Theta_{p_{p}}^{n}\|_{L^{2}(\Omega_{p})}~dt \\
  \leq& \int_{0}^{t}(\frac{1}{\epsilon_{1}}\|\mathbf{R}_{h}^{n}\|^{2}_{H^{1}(\Omega_{p})'}
  +\frac{\epsilon_{1}}{4}\|\Theta_{\mathbf{q}}^{n}\|^{2}_{L^{2}(\Omega_{p})})~dt \\
  \leq& \frac{\Delta t^{2}}{3\epsilon_{1}}\|d_{tt}\eta_{h}^{n}\|^{2}_{L^{2}(0,t;H^{1}(\Omega_{p})')}
   +\frac{\epsilon_{1}}{4}\|\Theta_{\mathbf{q}}^{n}\|^{2}_{L^{2}(0,t;L^{2}(\Omega_{p}))},
\end{aligned}
\end{eqnarray}
For any $\epsilon_{1}>0$, the stabilization terms $s_{f,p}(\cdot,\cdot),~s_{f,v}(\cdot,\cdot),~s_{f,q}(\cdot,\cdot)$ satisfy the following upper bounds, respectively,
\begin{eqnarray}\label{6.39}
\begin{aligned}
  & \int_{0}^{t}s_{f,p}(d_{t}p_{f,h}^{n},\Theta_{p_{f}}^{n})~dt \\
  \leq&\frac{\gamma_{stab}}{2}\frac{h\Delta t}{\gamma_{f}\mu_{f}}
  [\epsilon_{1}^{-1}\|d_{t}p_{f,h}^{n}\|^{2}_{L^{2}(0,t;L^{2}(\Gamma))}
  +\epsilon_{1}\frac{\mathrm{C}_{\mathrm{TI}}\Delta t}{h}\|\Theta_{p_{f}}^{n}\|^{2}_{L^{2}(0,t;L^{2}(\Omega_{f}))}],\\
\end{aligned}
\end{eqnarray}
\begin{eqnarray}\label{6.40}
\begin{aligned}
  & \int_{0}^{t}s_{f,v}(d_{t}\mathbf{v}_{h}^{n}\cdot\mathbf{n},\Theta_{\mathbf{v}}^{n}\cdot\mathbf{n})~dt  \\ \leq&\frac{\gamma'_{stab}}{2}\gamma_{f}\mu_{f}\frac{\Delta t}{h}[\epsilon_{1}^{-1}\|d_{t}\mathbf{v}_{h}^{n}\cdot\mathbf{n}\|^{2}_{L^{2}(0,t;L^{2}(\Gamma))}+\epsilon_{1}C\Delta t\|\Theta_{\mathbf{v}}^{n}\|_{L^{2}(0,t;H^{1}(\Omega_{f}))}^{2}],
\end{aligned}
\end{eqnarray}
\begin{eqnarray}\label{6.41}
\begin{aligned}
  &  \int_{0}^{t}s_{f,q}(d_{t}\mathbf{q}_{h}^{n}\cdot\mathbf{n},\Theta_{\mathbf{q}}^{n}\cdot\mathbf{n})~dt  \\
  \leq&\frac{\gamma'_{stab}}{2}\gamma_{f}\mu_{f}\frac{\Delta t}{h}[\epsilon_{1}^{-1}\|d_{t}\mathbf{q}_{h}^{n}\cdot\mathbf{n}\|^{2}_{L^{2}(0,t;L^{2}(\Gamma))}+\epsilon_{1}\frac{C_{TI}\Delta t}{h} \|\Theta_{\mathbf{q}}^{n}\|_{L^{2}(0,t;L^{2}(\Omega_{p}))}^{2}],
\end{aligned}
\end{eqnarray}
It is easy to check that
\begin{align*}
   & \|(\Theta_{p_{f}}^{n},\Theta_{p_{p}}^{n}),\Theta_{\xi}^{n}\|_{Q\times M} \nonumber\\
   \leq& C\sup\limits_{((\varphi_{f,h},\mathbf{r}_{h}),\varphi_{p,h})\in\mathbf{V}_{h}\times\mathbf{X}_{h}^{p}}
   (\frac{(\Theta_{p_{f}}^{n},\nabla\cdot\varphi_{f,h})_{\Omega_{f}}+(\Theta_{p_{p}}^{n},\nabla\cdot\mathbf{r}_{h})_{\Omega_{p}}
   +(\Theta_{\xi}^{n},\nabla\cdot\varphi_{p,h})_{\Omega_{p}}}
   {\|((\varphi_{f,h},\mathbf{r}_{h}),\varphi_{p,h})\|_{\mathbf{V}\times\mathbf{X}^{p}} }  \nonumber\\
   &+\frac{\langle\mathbf{n}\Theta_{\sigma_{f}}^{n}\mathbf{n},(\varphi_{f,h}-\varphi_{p,h}-\mathbf{r}_{h})\mathbf{n}\rangle_{\Gamma}}
   {\|((\varphi_{f,h},\mathbf{r}_{h}),\varphi_{p,h})\|_{\mathbf{V}\times\mathbf{X}^{p}} }) \nonumber\\
    \leq& C\sup\limits_{((\varphi_{f,h},\mathbf{r}_{h}),\varphi_{p,h})\in\mathbf{V}_{h}\times\mathbf{X}_{h}^{p}}
    (\frac{-2\mu_{f}(\mathbf{D}(E_{\mathbf{v}}^{n}),\mathbf{D}(\varphi_{f,h}))_{\Omega_{f}}-2\mu_{p}(\mathbf{D}(E_{\mathbf{U}}^{n}),\mathbf{D}(\varphi_{p,h}))_{\Omega_{p}}}
    {\|((\varphi_{f,h},\mathbf{r}_{h}),\varphi_{p,h})\|_{\mathbf{V}\times\mathbf{X}^{p}}}\nonumber\\
   &+\frac{-k^{-1}(E_{\mathbf{q}}^{n},\mathbf{r}_{h})_{\Omega_{p}}-\langle\tau E_{\sigma_{f}}^{n}\mathbf{n},(\varphi_{f,h}-\varphi_{p,h})\tau\rangle_{\Gamma}}
    {\|((\varphi_{f,h},\mathbf{r}_{h}),\varphi_{p,h})\|_{\mathbf{V}\times\mathbf{X}^{p}}}
    +\frac{-\langle\mathbf{n}\Lambda_{\sigma_{f}}^{n}\mathbf{n},(\varphi_{f,h}-\varphi_{p,h}-\mathbf{r}_{h})\mathbf{n}\rangle_{\Gamma}}
   {\|((\varphi_{f,h},\mathbf{r}_{h}),\varphi_{p,h})\|_{\mathbf{V}\times\mathbf{X}^{p}} } \nonumber\\
   &+\frac{-(\Lambda_{p_{f}}^{n},\nabla\cdot\varphi_{f,h})_{\Omega_{f}}-(\Lambda_{p_{p}}^{n},\nabla\cdot\mathbf{r}_{h})_{\Omega_{p}}
   -(\Lambda_{\xi}^{n},\nabla\cdot\varphi_{p,h})_{\Omega_{p}}}
   {\|((\varphi_{f,h},\mathbf{r}_{h}),\varphi_{p,h})\|_{\mathbf{V}\times\mathbf{X}^{p}} })
\end{align*}
 and
\begin{eqnarray*}
\begin{aligned}
   & \epsilon_{2}(\|\Theta_{p_{f}}^{n}\|^{2}_{L^{2}(0,t;L^{2}(\Omega_{f}))}+\|\Theta_{p_{p}}^{n}\|^{2}_{L^{2}(0,t;L^{2}(\Omega_{p}))}
   +\|\Theta_{\xi}^{n}\|^{2}_{L^{2}(0,t;L^{2}(\Omega_{p}))}) \\
   \leq& C\epsilon_{2}(\|\Theta_{\mathbf{v}}^{n}\|^{2}_{L^{2}(0,t;H^{1}(\Omega_{p}))}+\|\Theta_{\mathbf{U}}^{n}\|^{2}_{L^{2}(0,t;H^{1}(\Omega_{p}))}
   +\|\Theta_{\mathbf{q}}^{n}\|^{2}_{L^{2}(0,t;L^{2}(\Omega_{p}))} \\
\end{aligned}
\end{eqnarray*}
\begin{eqnarray}\label{6.42}
\begin{aligned}
   & +\|\Lambda_{\mathbf{v}}^{n}\|^{2}_{L^{2}(0,t;H^{1}(\Omega_{p}))}+\|d_{t}\Lambda_{\mathbf{U}}^{n}\|^{2}_{L^{2}(0,t;H^{1}(\Omega_{p}))}
   +\|\Lambda_{\mathbf{q}}^{n}\|^{2}_{L^{2}(0,t;L^{2}(\Omega_{p}))}  \\
   & +\|\Lambda_{p_{f}}^{n}\|^{2}_{L^{2}(0,t;L^{2}(\Omega_{f}))}+\|\Lambda_{p_{p}}^{n}\|^{2}_{L^{2}(0,t;L^{2}(\Omega_{p}))}
   +\|\Lambda_{\xi}^{n}\|^{2}_{L^{2}(0,t;L^{2}(\Omega_{p}))} \\
   & +\|(\Theta_{\mathbf{v}}^{n}-d_{t}\Theta_{\mathbf{U}}^{n})\cdot\tau\|_{L^{2}(0,t;L^{2}(\Gamma))}
   +\|(\Theta_{\mathbf{v}}^{n}-d_{t}\Theta_{\mathbf{U}}^{n}-\Theta_{\mathbf{q}}^{n})\cdot\mathbf{n}\|_{L^{2}(0,t;L^{2}(\Gamma))}).
\end{aligned}
\end{eqnarray}

Substituting (\ref{6.25})-(\ref{6.42}) into (\ref{6.24}) and taking~$\epsilon_{1},~\epsilon_{2}$~small enough, we get
\begin{eqnarray}\label{}
  && \frac{1}{2}[c\|\Theta_{\mathbf{U}}^{t}\|_{H^{1}(\Omega_{p})}^{2}+k_{3}\|\Phi_{\xi}^{t}\|_{L^{2}(\Omega_{p})}^{2}
   +k_{2}\|\Phi_{\eta}^{t}\|_{L^{2}(\Omega_{p})}^{2}]+c^{p}\|\Theta_{\mathbf{q}}^{n}\|^{2}_{L^{2}(0,t;L^{2}(\Omega_{p}))}
   \nonumber \\
  && +c\|\Theta_{\mathbf{v}}^{n}\|^{2}_{L^{2}(0,t;H^{1}(\Omega_{f}))}+\|\Theta_{p_{f}}^{n}\|^{2}_{L^{2}(0,t;L^{2}(\Omega_{f}))}
  +\|\Theta_{p_{p}}^{n}\|^{2}_{L^{2}(0,t;L^{2}(\Omega_{p}))}+\|\Theta_{\xi}^{n}\|^{2}_{L^{2}(0,t;L^{2}(\Omega_{p}))} \nonumber \\
  &&+\gamma_{f}\mu_{f}h^{-1}\|(\Theta_{\mathbf{v}}^{n}-d_{t}\Theta_{\mathbf{U}}^{n}-\Theta_{\mathbf{q}}^{n})\cdot\mathbf{n}\|^{2}_{L^{2}(0,t;L^{2}(\Gamma))}
   +\gamma_{f}\mu_{f}h^{-1}\|(\Theta_{\mathbf{v}}^{n}-d_{t}\Theta_{\mathbf{U}}^{n})\cdot\tau\|^{2}_{L^{2}(0,t;L^{2}(\Gamma))} \nonumber\\
  &\leq & C ( \|\Lambda_{\mathbf{v}}^{n}\|^{2}_{L^{2}(0,t;H^{1}(\Omega_{f}))}+\|\Lambda_{p_{f}}^{n}\|^{2}_{L^{2}(0,t;L^{2}(\Omega_{f}))}
   +\|\Lambda_{\mathbf{q}}^{n}\|^{2}_{L^{2}(0,t;L^{2}(\Omega_{p}))}\nonumber\\
   && +\|\Lambda_{p_{p}}^{n}\|^{2}_{L^{2}(0,t;L^{2}(\Omega_{p}))}+ \|\Lambda_{\mathbf{U}}^{t}\|^{2}_{H^{1}(\Omega_{p})}+\|\Lambda_{\xi}^{t}\|^{2}_{L^{2}(\Omega_{p})}  \\
   &&+\|d_{t}\Lambda_{\mathbf{U}}^{n}\|^{2}_{L^{2}(0,t;H^{1}(\Omega_{p}))}+\|\Lambda_{\xi}^{n}\|^{2}_{L^{2}(0,t;L^{2}(\Omega_{p}))}
   +\|d_{t}\Lambda_{\xi}^{n}\|^{2}_{L^{2}(0,t;L^{2}(\Omega_{p}))} ) \nonumber\\
   &&+ \frac{(1+\Delta t)\Delta t\gamma_{f}\mu_{f}}{2h\epsilon_{1}}(\|d_{t}\mathbf{v}_{h}^{n}\cdot\tau\|_{L^{2}(0,t;L^{2}(\Gamma))}^{2}
    +(2+\frac{\gamma'_{stab}}{1+\Delta t})\|d_{t}\mathbf{v}_{h}^{n}\cdot\mathbf{n}\|_{L^{2}(0,t;L^{2}(\Gamma))}^{2} \nonumber\\
   && +\|d_{tt}\mathbf{U}_{h}^{n}\cdot\mathbf{n}\|_{L^{2}(0,t;L^{2}(\Gamma))}^{2}+(1+\frac{\gamma'_{stab}}{1+\Delta t})\|d_{t}\mathbf{q}_{h}^{n}\cdot\mathbf{n}\|_{L^{2}(0,t;L^{2}(\Gamma))}^{2} )\nonumber\\
   && +\frac{\mu_{f}}{\epsilon_{1}}\Delta t^{2}C\|d_{t}\mathbf{v}_{h}^{n}\|_{L^{2}(0,t;H^{1}(\Omega_{f}))}^{2}+(\frac{\Delta t^{2}}{2\epsilon_{1}}+\frac{\gamma_{stab}h\Delta t}{2\gamma_{f}\mu_{f}\epsilon_{1}})\|d_{t}p_{f,h}^{n}\|^{2}_{L^{2}(0,t;L^{2}(\Gamma))}\nonumber\\
   &&+\frac{\Delta tk_{1}^{2}}{2\epsilon_{1}}\|d_{t}\eta_{h}^{n}\|^{2}_{L^{2}(0,t;L^{2}(\Omega_{p}))}
   +\frac{\Delta t^{2}}{3\epsilon_{1}}\|d_{tt}\eta_{h}^{n}\|^{2}_{L^{2}(0,t;H^{1}(\Omega_{p})')}.\nonumber
\end{eqnarray}
Using (\ref{5.14}), (\ref{5.18}), (\ref{5.19}), (\ref{5.22}) and (\ref{5.26}), we see that (\ref{eq210723-1}) holds. The proof is complete.
\end{proof}

\begin{thrm}\label{thm5.2}
The solution of the problem (\ref{3.5})-(\ref{3.7}) satisfies the following error estimates
\begin{eqnarray}
  && \frac{\sqrt{2}}{2}[\sqrt{c}\|\mathbf{U}(t_{n})-\mathbf{U}^{n}_{h}\|_{L^{\infty}(0,T;H^{1}(\Omega_{p}))}
  +\sqrt{k_{3}}\|\xi(t_{n})-\xi^{n}_{h}\|_{L^{\infty}(0,T;L^{2}(\Omega_{p}))}\nonumber\\
  && +\sqrt{k_{2}}\|\eta(t_{n})-\eta^{n}_{h}\|_{L^{\infty}(0,T;L^{2}(\Omega_{p}))}]\nonumber\\
  && +\sqrt{c^{p}}\|\mathbf{q}(t_{n})-\mathbf{q}^{n}_{h}\|_{L^{2}(0,T;L^{2}(\Omega_{p}))}
   +\sqrt{c}\|\mathbf{v}(t_{n})-\mathbf{v}^{n}_{h}\|_{L^{2}(0,T;H^{1}(\Omega_{f}))} \nonumber\\
  && +\|p_{f}(t_{n})-p_{f,h}^{n}\|_{L^{2}(0,T;L^{2}(\Omega_{f}))}+\|p_{p}(t_{n})-p_{p,h}^{n}\|_{L^{2}(0,T;L^{2}(\Omega_{p}))}
   +\|\xi(t_{n})-\xi^{n}_{h}\|_{L^{2}(0,T;L^{2}(\Omega_{p}))}  \nonumber\\
  &&+\sqrt{\gamma_{f}\mu_{f}h^{-1}}\|[(\mathbf{v}(t_{n})-d_{t}\mathbf{U}(t_{n})-\mathbf{q}(t_{n}))
  -(\mathbf{v}^{n}_{h}-d_{t}\mathbf{U}^{n}_{h}-\mathbf{q}^{n}_{h})]\cdot\mathbf{n}\|_{L^{2}(0,T;L^{2}(\Gamma))}  \nonumber\\
  && +\sqrt{\gamma_{f}\mu_{f}h^{-1}}\|[(\mathbf{v}(t_{n})-d_{t}\mathbf{U}(t_{n}))
  -(\mathbf{v}^{n}_{h}-d_{t}\mathbf{U}^{n}_{h})]\cdot\tau\|_{L^{2}(0,T;L^{2}(\Gamma))} \nonumber\\
   &\leq & C\sqrt{\mathrm{exp}(T)} (h^{r_{3}} \|\mathbf{v}\|_{L^{2}(0,T;H^{r_{3}+1}(\Omega_{f}))}+h^{r_{1}}\|p_{f}\|_{L^{2}(0,T;H^{r_{1}}(\Omega_{f}))}
   +h^{r_{5}}\|\mathbf{q}\|_{L^{2}(0,T;H^{r_{5}}(\Omega_{p}))} \nonumber\\
   && + h^{r_{4}}\|\mathbf{U}\|_{L^{\infty}(0,T;H^{r_{4}+1}(\Omega_{p}))}+h^{r_{4}}\|d_{t}\mathbf{U}\|_{L^{2}(0,T;H^{r_{4}+1}(\Omega_{p}))}
   +h^{r_{6}}\|\eta\|_{L^{\infty}(0,T;H^{r_{6}+1}(\Omega_{p}))} \label{eq210516-1}\\
   &&+h^{r_{6}}(\|\xi\|_{L^{2}(0,T;H^{r_{6}+1}(\Omega_{p}))}+\|\xi\|_{L^{\infty}(0,T;H^{r_{6}+1}(\Omega_{p}))}
   +\|d_{t}\xi\|_{L^{2}(0,T;H^{r_{6}+1}(\Omega_{p}))})  \nonumber\\
   && +h^{r_{2}}\|p_{p}\|_{L^{2}(0,T;H^{r_{2}}(\Omega_{p}))})\nonumber\\
   && +\sqrt{\frac{(1+\Delta t)\Delta t\gamma_{f}\mu_{f}}{2h\epsilon_{1}}}(\|d_{t}\mathbf{v}_{h}^{n}\cdot\tau\|_{L^{2}(0,T;L^{2}(\Gamma))}
    +\sqrt{2+\frac{\gamma'_{stab}}{1+\Delta t}}\|d_{t}\mathbf{v}_{h}^{n}\cdot\mathbf{n}\|_{L^{2}(0,T;L^{2}(\Gamma))}\nonumber\\
   && +\sqrt{(1+\frac{\gamma'_{stab}}{1+\Delta t})}\|d_{t}\mathbf{q}_{h}^{n}\cdot\mathbf{n}\|_{L^{2}(0,T;L^{2}(\Gamma))}
   +\|d_{tt}\mathbf{U}_{h}^{n}\cdot\mathbf{n}\|_{L^{2}(0,T;L^{2}(\Gamma))} )\nonumber\\
   && +\sqrt{\frac{\Delta t^{2}}{2\epsilon_{1}}+\frac{\gamma_{stab}h\Delta t}{2\gamma_{f}\mu_{f}\epsilon_{1}}}\|d_{t}p_{f,h}^{n}\|_{L^{2}(0,T;L^{2}(\Gamma))}
   +\sqrt{\frac{\Delta tk_{1}^{2}}{2\epsilon_{1}}}\|d_{t}\eta_{h}^{n}\|_{L^{2}(0,T;L^{2}(\Omega_{p}))}\nonumber\\
   &&+\sqrt{\frac{\Delta t^{2}}{3\epsilon_{1}}}\|d_{tt}\eta_{h}^{n}\|_{L^{2}(0,T;H^{1}(\Omega_{p})')}+\sqrt{\frac{\mu_{f}}{\epsilon_{1}}\Delta t^{2}C}\|d_{t}\mathbf{v}_{h}^{n}\|_{L^{2}(0,T;H^{1}(\Omega_{f}))},\nonumber
\end{eqnarray}
where
\begin{gather*}
  0\leq r_{1}\leq s_{f}+1,~~~0\leq r_{2}\leq s_{p}+1,~~~0\leq r_{3}\leq k_{f}, \\
  0\leq r_{4}\leq k_{p},~~~1\leq r_{5}\leq k_{s}+1,~~~1\leq r_{6}\leq s_{p}+1.
\end{gather*}
\end{thrm}

\begin{proof}

The above estimates following immediately from an application of the triangle inequality on

\begin{eqnarray*}
  \mathbf{v}(t_{n})-\mathbf{v}_{h}^{n}=\Lambda_{\mathbf{v}}^{n}+\Theta_{\mathbf{v}}^{n}, && p_{f}(t_{n})-p_{f,h}^{n}=\Lambda_{p_{f}}^{n}+\Theta_{p_{f}}^{n}, \\
  \mathbf{U}(t_{n})-\mathbf{U}_{h}^{n}=\Lambda_{\mathbf{U}}^{n}+\Theta_{\mathbf{U}}^{n}, && \xi(t_{n})-\xi_{h}^{n}=\Lambda_{\xi}^{n}+\Theta_{\xi}^{n},\\
  \eta(t_{n})-\eta_{h}^{n}=\Lambda_{\eta}^{n}+\Theta_{\eta}^{n}, && p_{p}(t_{n})-p_{p,h}^{n}=\Lambda_{p_{p}}^{n}+\Theta_{p_{p}}^{n},\\
  \mathbf{q}(t_{n})-\mathbf{q}_{h}^{n}=\Lambda_{\mathbf{q}}^{n}+\Theta_{\mathbf{q}}^{n}. &&
\end{eqnarray*}

Using (\ref{5.14}), (\ref{5.18}), (\ref{5.19}), (\ref{5.22}), (\ref{5.26}) and Theorem \ref{thm5.1}, we see that (\ref{eq210516-1}) holds. The proof is complete.
\end{proof}


\section{Numerical tests}\label{sec4}

{\bf Test  1.}
Taking the domain by $\Omega=[0,1]\times[-1,1]$, and  we associate the upper half with the Stokes flow, while the
lower half represents the poroelastic structure. The appropriate interface conditions are enforced along the interface $y=0$. 

The source functions~$\mathbf{f},~g,~\mathbf{h}$~and~$s$~are given by
\begin{equation*}
  \mathbf{f}=\begin{pmatrix}\pi e^{t}\cos(\frac{\pi y}{2})\cos(\pi x)+\pi\mu_{f}\cos(y)\cos(\pi t)\\-\frac{\pi}{2}e^{t}\sin(\pi x)\sin(\frac{\pi y}{2}) \end{pmatrix},
\end{equation*}
\begin{equation*}
  \mathbf{h}=\begin{pmatrix}\alpha\pi e^{t}\cos(\frac{\pi y}{2})\cos(\pi x)+\mu_{p}\cos(y)\sin(\pi t)\\-\frac{\pi}{2}\alpha e^{t}\sin(\pi x)\sin(\frac{\pi y}{2}) \end{pmatrix},
\end{equation*}
\begin{eqnarray*}
  g=-2\pi\cos(\pi t), \ \
  s=(s_{0}-\frac{3}{4}\pi^{2})e^{t}\sin(\pi x)\cos(\frac{\pi y}{2})-2\alpha\pi\cos(\pi t).
\end{eqnarray*}
The boundary conditions are
\begin{eqnarray*}
  \mathbf{U}=0,~~\mathbf{q}=0,~~\xi=0,~~\eta=0,~~&&(\mathbf{x},t)\in\Gamma_{p}^{in}\cup\Gamma_{p}^{out}\cup\Gamma_{p}^{ext}\times(0,T], \\
  \mathbf{v}=0,~~&&(\mathbf{x},t)\in\Gamma_{f}^{in}\cup\Gamma_{f}^{out}\times(0,T],\\
  p_{f}=0,~~&&(\mathbf{x},t)\in\Gamma_{f}^{out}\times(0,T].
\end{eqnarray*}
The initial values are
\begin{equation*}
\mathbf{U}^{0}=0,~~\mathbf{v}^{0}=0,~~\xi^{0}=0,~~\eta^{0}=0.
\end{equation*}
 And we take the total simulation time for this test case is $T=10^{-3}$ and the time step is $\Delta t=10^{-4}$, the mesh space discretization step with $h=0.014$ as a exact solution. We use $\mathbb{P}_2-\mathbb{P}_1$ element approximations for velocity and pressure in the fluid, combined with $\mathbb{P}_2-\mathbb{P}_1$ element approximation of the relative velocity and pressure of the fluid within the porous structure and $\mathbb{P}_2$ element approximation of the structure displacement.
 
 \begin{table}[H]
 	\caption{Values of parameters}\label{table2}
 	\centering
 	\begin{tabular}{ccc}
 		\hline
 		Parameters& Notation & Values \\
 		\hline
 		Viscosity & $\mu_{f}$ & 0.01 \\
 		Lem\'{e} constant & $\mu_{p}$ & $1\times10^{8}$ \\
 		Lem\'{e} constant & $\lambda_{p}$ & $4.28\times10^{6}$\\
 		Mass storage coefficient & $s_{0}$ & $5\times10^{-6}$\\
 		Biot-Willis constant & $\alpha$ & 1\\
 		Hydraulic conductivity & k & 1\\
 		\hline
 	\end{tabular}
 \end{table}
 
\begin{table}[H]
  \caption{Error indicators and convergence order}\label{table1}
  \centering
  \begin{tabular}{ccccccccc}
  \hline
  $h$ & $\varepsilon_{f}$ & Rate & $\varepsilon_{p}$ & Rate & $\varepsilon_{fp}$ & Rate & $\varepsilon_{pp}$ & Rate\\
  \hline
  $h=0.27$ & 1.8E-02 &   & 4.6E-07 &      & 6.7E-04 &      & 7.4E-04 &     \\
  $h/2$ & 1.1E-02 & 0.71 & 2.5E-07 & 0.88 & 2.7E-04 & 1.31 & 3.1E-04 & 1.26\\
  $h/4$ & 6.8E-03 & 0.69 & 1.3E-07 & 0.94 & 1.1E-04 & 1.30 & 1.3E-04 & 1.25\\
  $h/8$ & 4.0E-03 & 0.77 & 4.1E-08 & 1.66 & 3.9E-05 & 1.50 & 4.5E-05 & 1.53\\
  \hline
  \end{tabular}
\end{table}
Table \ref{table1} gives the error and the corresponding order of convergence between numerical solution and exact solutions in different mesh~$h$, which shows that the smaller error of the velocity $\mathbf{v}$ and pressure~$p_{f}$, the displacement~$\mathbf{U}$, pressure~$p_{p}$ in different regions, it improved the calculation precision, where 
\begin{align*}
	\varepsilon_{f}:=\|\mathbf{v}-\mathbf{v}_{h}\|_{L^{2}(H^{1})}, & ~~~~~ \varepsilon_{fp}:=\|p_{f}-p_{f,h}\|_{L^{2}(L^{2})}, \\
	\varepsilon_{p}:=\|\mathbf{U}-\mathbf{U}_{h}\|_{L^{\infty}(H^{1})}, & ~~~~~ \varepsilon_{pp}:=\|p_{p}-p_{p,h}\|_{L^{2}(L^{2})}.
\end{align*}
The convergence order of velocity $\mathbf{v}$ is greater than $0.5$, and the convergence order of displacement $\mathbf{U}$ and pressure ~$p_{f}$,~$p_{p}$ can reach the first convergence gradually. 

\begin{figure}[H]
\begin{minipage}[t]{0.5\textwidth}
	\centering
	\includegraphics[width=8cm]{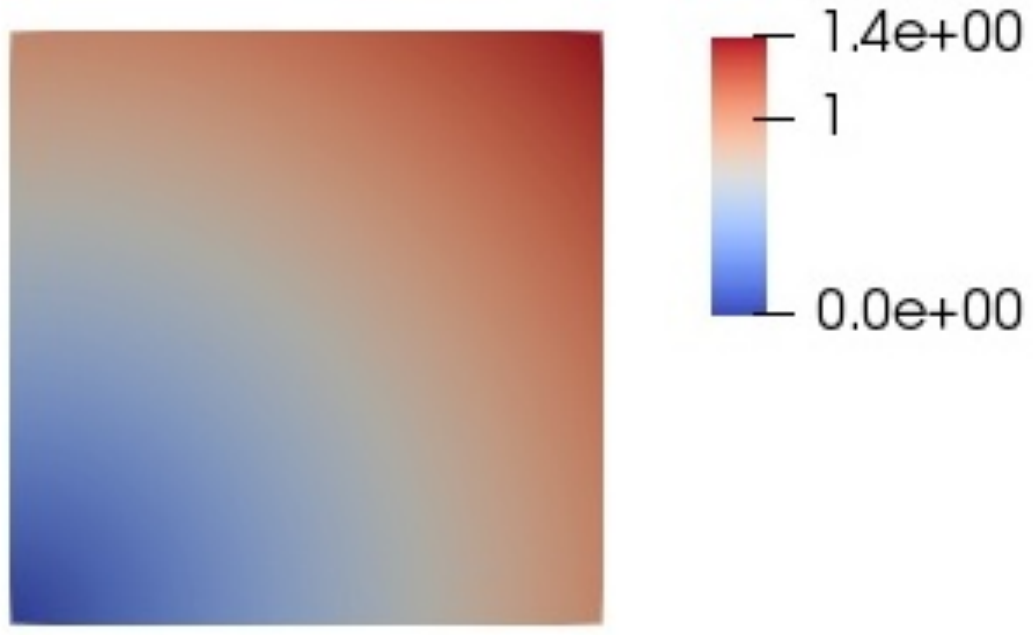}
	\caption{Numerical solution of
		$\mathbf{v}$ in\ $t=0.001$}\label{fig v}
\end{minipage}
\begin{minipage}[t]{0.5\textwidth}
\centering
\includegraphics[width=8cm]{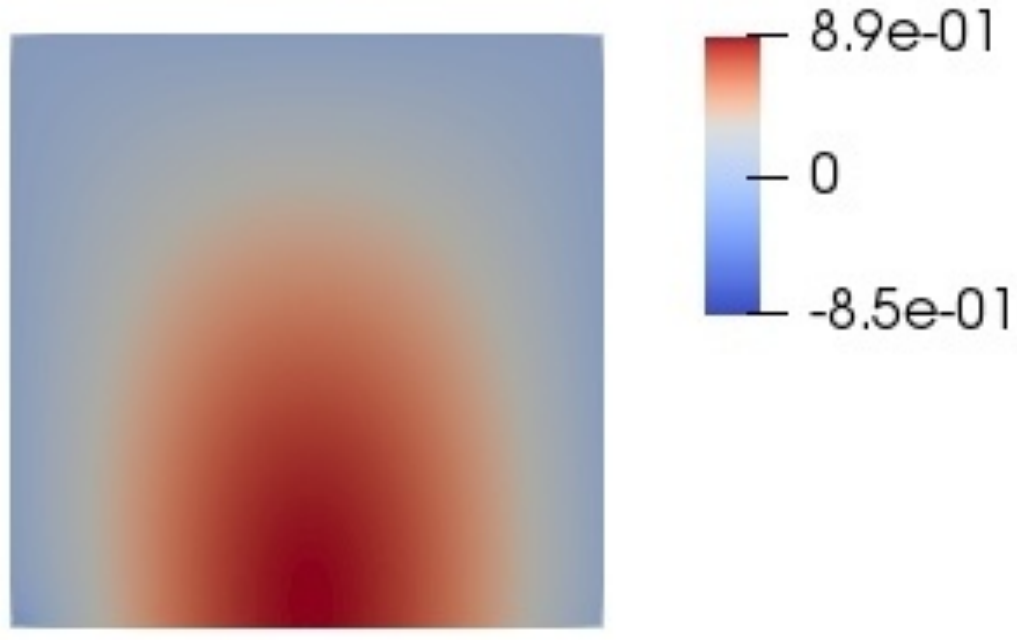}
\caption{Numerical solution of
	$p_{f}$ in\ $t=0.001$}\label{fig pF}
\end{minipage}
\end{figure}

\begin{figure}[H]
\begin{minipage}[t]{0.5\textwidth}
\centering
\includegraphics[width=8cm]{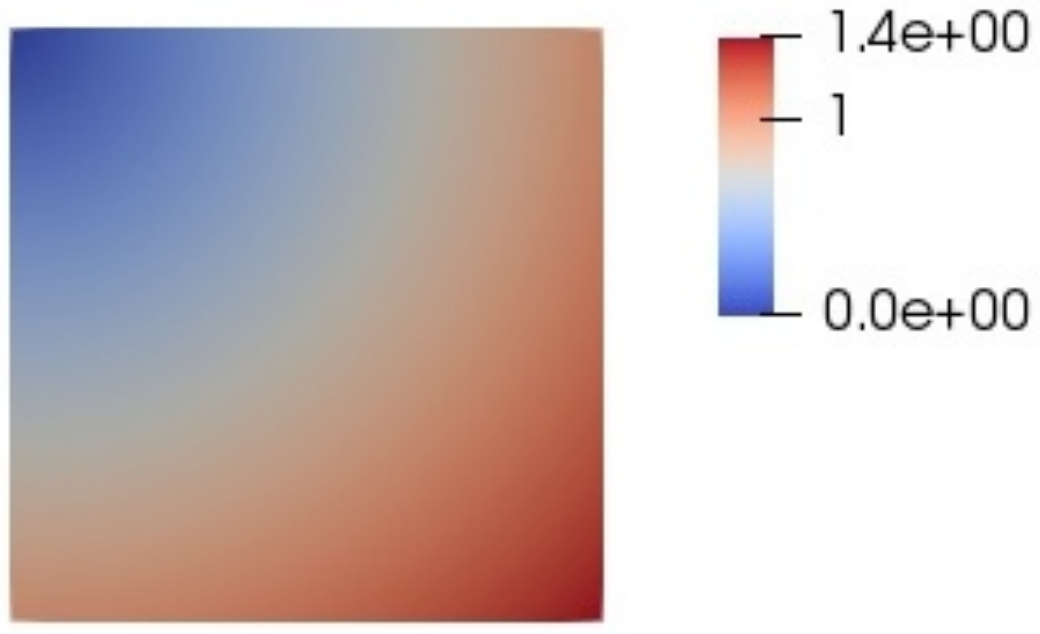}
\caption{The numerical solution of
 $\mathbf{U}$ in\ $t=0.001$}\label{fig U}
\end{minipage}
\begin{minipage}[t]{0.5\textwidth}
\centering
\includegraphics[width=8cm]{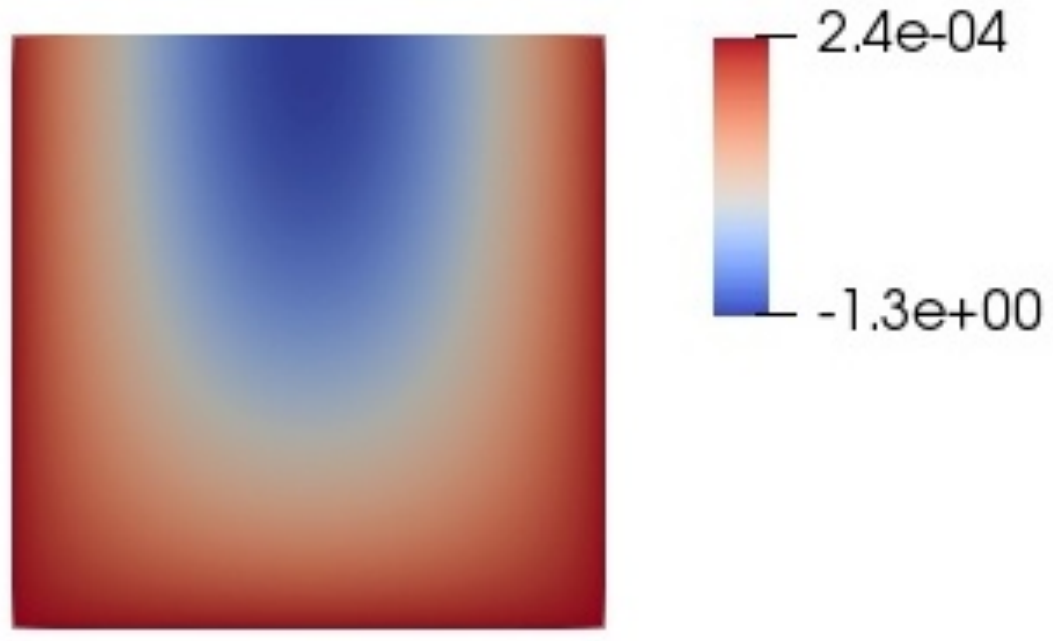}
\caption{The numerical solution of
 $p_{p}$ in\ $t=0.001$}\label{fig pP}
\end{minipage}
\end{figure}


From Figure \ref{fig v}- Figure \ref{fig pP}, we find out that our loosely-coupled time-stepping method  has good numerical stability and no "locking" phenomenon.

{\bf Test  2.}
In this example, we focus on fluid¨Cstructure interaction in the context of modeling the interaction between a stationary fracture filled with fluid and the surrounding poroelastic medium. We consider a general case where the hydraulic conductivity is a tensor given by $k=\frac{\mathbf{K}}{\mu_{f}}$, where $\mathbf{K}$ is the permeability tensor. The interface condition is the Beavers¨CJoseph¨CSaffman condition (\ref{1.15}) with the coefficient $
  \beta=\frac{\alpha\mu_{f}\sqrt{3}}{\sqrt{tr(\mathbf{K})}}$. The reference domain is a square $[-100m,100m]^{2}$. A fracture is positioned in the middle of the square, whose boundary is given by
\begin{equation*}
  \hat{y}^{2}=0.8^{2}(\hat{x}-35)(\hat{x}+35).
\end{equation*}
The fracture represents the reference fluid domain $\hat{\Omega}^{f}$, while the reference poroelastic structure domain is given as
$\hat{\Omega}^{p}=\hat{\Omega}\backslash\hat{\Omega}^{f}$, one can see Figure \ref{fig 3}. To obtain a more realistic domain, we transform the reference domain~$\hat{\Omega}$~onto the physical domain~$\Omega$~via the mapping
\begin{equation*}\label{7.2}
  \begin{pmatrix}x\\y\end{pmatrix}(\hat{x},\hat{y})=\begin{pmatrix}\hat{x}\\
  5\cos(\frac{\hat{x}+\hat{y}}{100})\cos(\frac{\pi\hat{x}+\hat{y}}{100})^{2}+\hat{y}/5-\hat{x}/10
  \end{pmatrix}.
\end{equation*}
\begin{figure}[H]
\centering
\includegraphics[width=14cm]{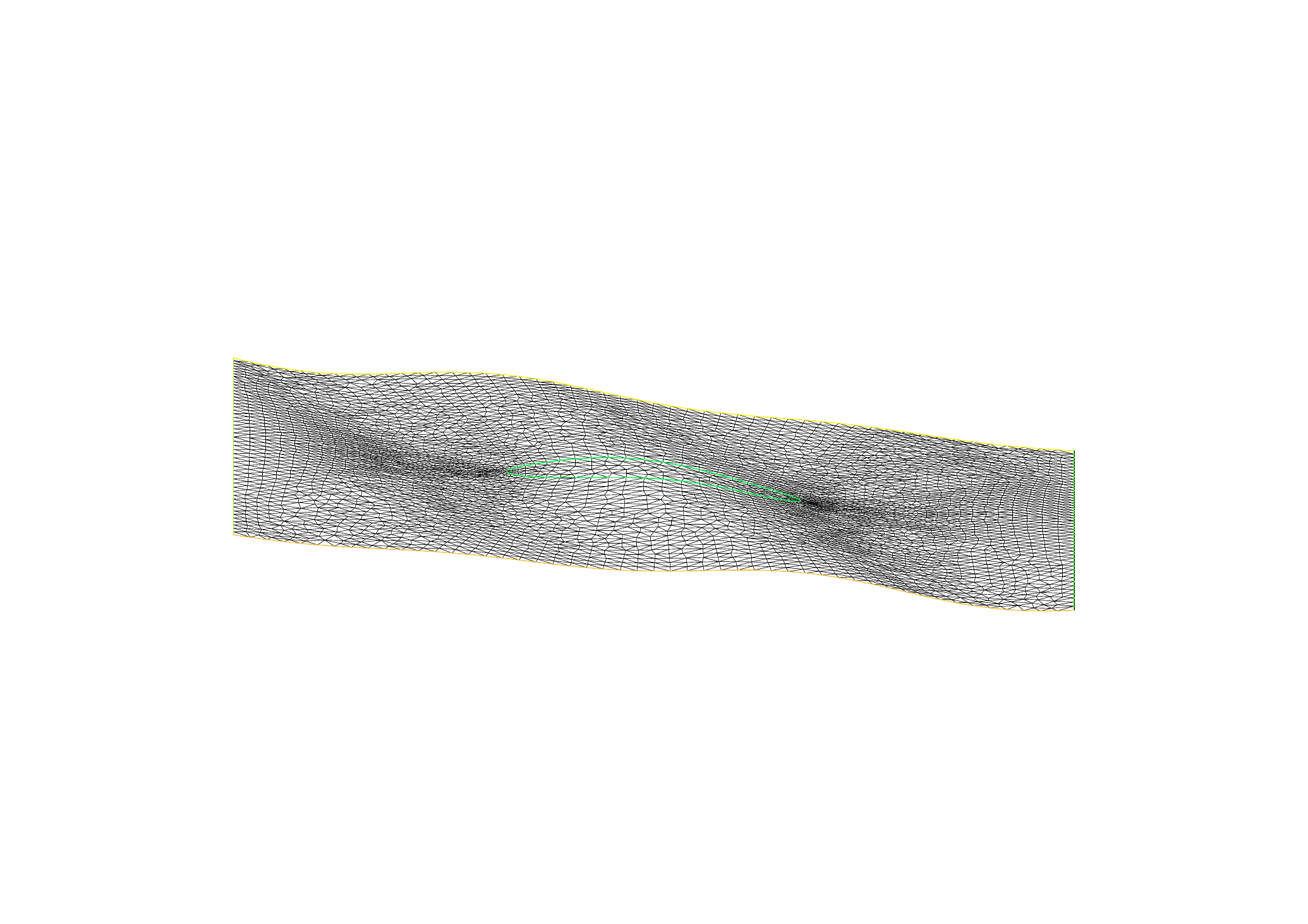}                  
\caption{
 Physical~computational~domain~$\Omega=\Omega^{f}\bigcup\Omega^{p}$}\label{fig 3}
\end{figure}

The flow is driven by the injection of the fluid into the fracture with the constant rate $g = 25 kg/s$. On all external boundaries, we prescribe
the no flow condition $\mathbf{q}\cdot\mathbf{n}=0$, zero normal displacement $\mathbf{U}\cdot\mathbf{n}=0$, and zero shear traction~$\tau\cdot\sigma_{p}\mathbf{n}=0$. The simulation time is $T=10s$. The problem was solved using the time step $\Delta t=0.1s$. The remaining parameters are given in
Table \ref{table3}. We adopt the $\mathbb{P}_1-\mathbb{P}_1$ element approximation for the fluid velocity and pressure, complemented with the pressure stabilization $s_{p}(p_{f,h},\psi_{f,h})$. For the relative velocity and pressure approximation of the fluid in the poroelastic medium, as for the structure
 displacement, we again use the $\mathbb{P}_1-\mathbb{P}_1$ finite elements. However, due to the large time-step in this
example, we are numerically closer to the divergence-free regime. Thus, we add two pseudo-pressure stabilization given
by
\begin{equation*}
  s_{q}(\xi_{h},w_{h}):=\gamma_{q}h^{2}k_{1}\int_{\Omega_{p}}\nabla \xi_{h}\cdot\nabla w_{h}d\mathbf{x},~~~
  s_{q}(\eta_{h},z_{h}):=\gamma_{q}h^{2}k_{2}\int_{\Omega_{p}}\nabla \eta_{h}\cdot\nabla z_{h}d\mathbf{x},
\end{equation*}
where the stabilization parameter is selected as $\gamma_{q}=10^{-3}$, and the choice of the penalty and stabilization parameters are $\gamma_{f}=1500,~\gamma_{stab}=1,~\gamma_{stab}'=0$.
\begin{table}[H]
  \caption{Values of parameters}\label{table3}
  \centering
  \begin{tabular}{ccc}
  \hline
  Parameters& Denotations & Values \\
  \hline
  Viscosity & $\mu_{f}$ & $10^{-3}$ \\
  Lem\'{e} constant & $\mu_{p}$ & $2.92\times10^{8}$ \\
  Lem\'{e} constant & $\lambda_{p}$ & $1.94\times10^{10}$\\
  Mass storage coefficient & $s_{0}$ & $6.9\times10^{-5}$\\
  Biot-Willis constant & $\alpha$ & 1\\
  Hydraulic conductivity & k & $1\times10^{-8}$\\
  Beavers¨CJoseph¨CSaffman coefficient  & $\beta$ & $3.47\times10^{3}$\\
  \hline
  \end{tabular}
\end{table}

\begin{figure}[H]
\begin{minipage}[t]{0.5\textwidth}	
\centering
\includegraphics[width=8cm]{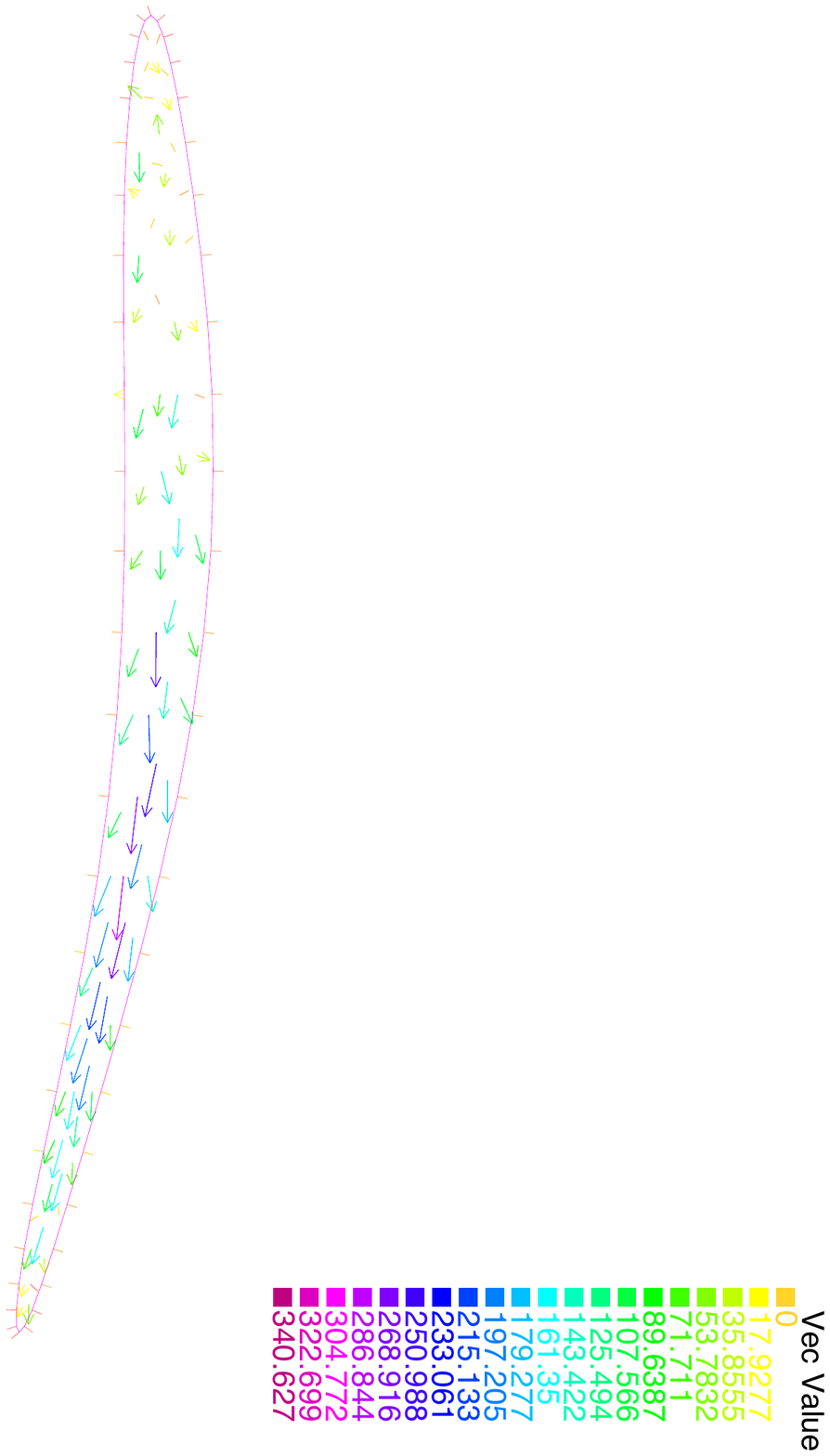}
\caption{
 $Fluid~velocity$}\label{fig 4}
\end{minipage}
\begin{minipage}[t]{0.5\textwidth}
\centering
\includegraphics[width=8cm]{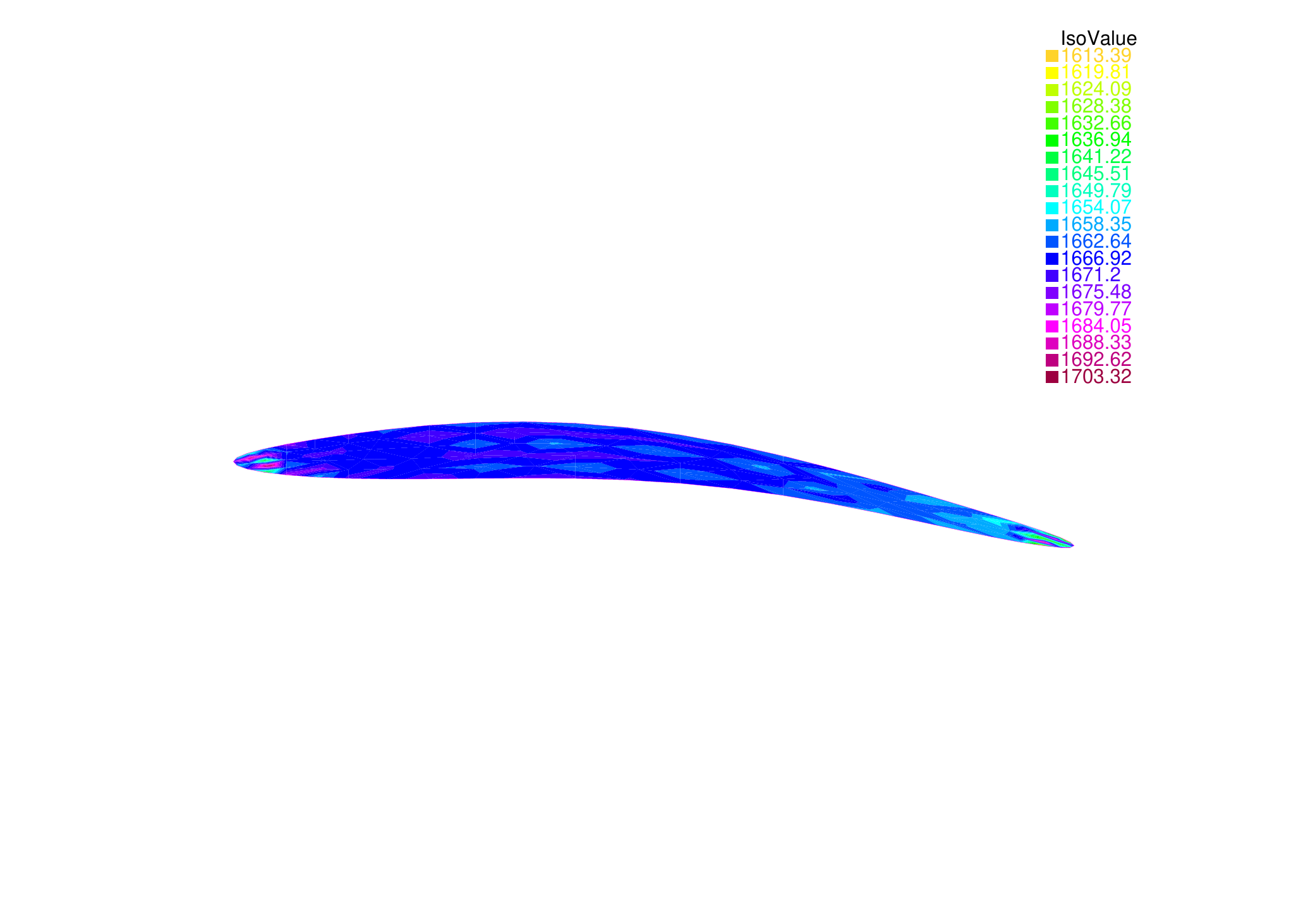}
\caption{
 $Fluid~pressure$}\label{fig 5}
\end{minipage}
\end{figure}

Figure \ref{fig 4} and Figure \ref{fig 5} show the pressure and the velocity in the fracture at final time. At the beginning of the process the
simulations capture the expected local pressure increase in the region of fluid injection, while at the final time, the fluid pressure is the largest at the tip of the crack. The pressure and displacement are shown in Figure \ref{fig 6} and Figure \ref{fig 7}.

\begin{figure}[H]
\begin{minipage}[t]{0.5\textwidth}	
\centering
\includegraphics[width=8cm]{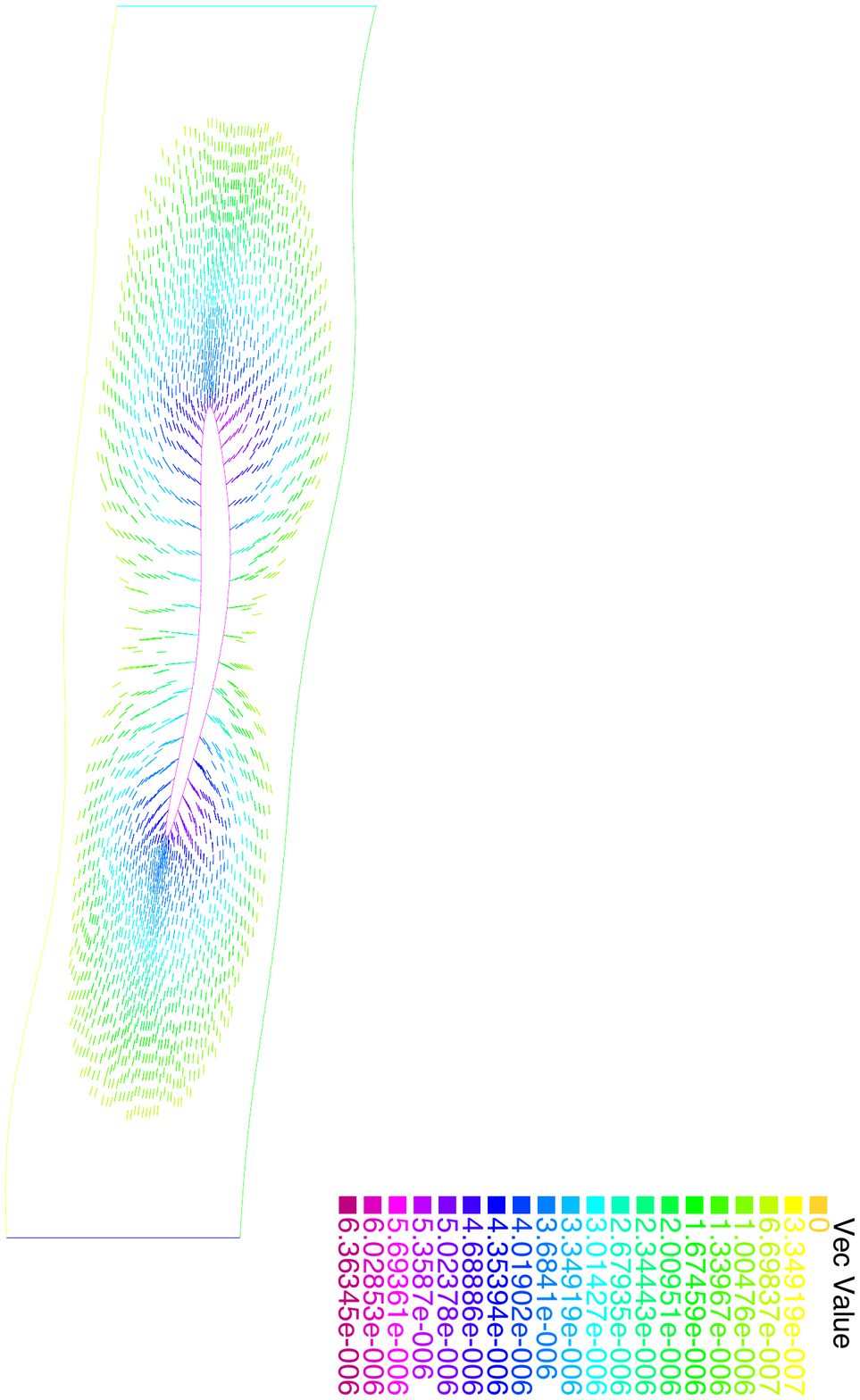}
\caption{
 Fluid~displacement~in~surrounding medium}\label{fig 6}
\end{minipage}
\begin{minipage}[t]{0.5\textwidth}	
\centering
\includegraphics[width=8cm]{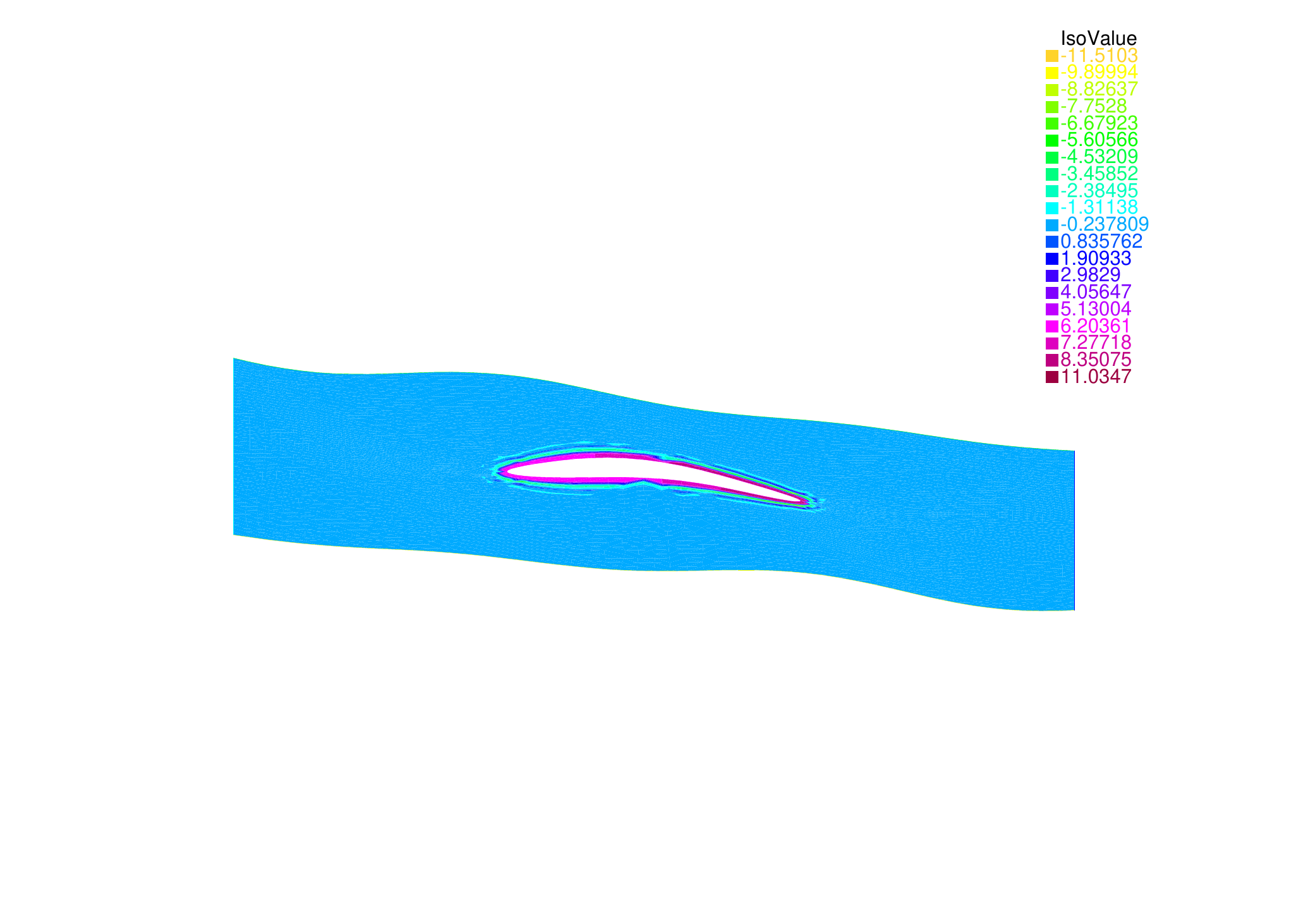}
\caption{
 Fluid~pressure~in~surrounding medium}\label{fig 7}
\end{minipage}
\end{figure}

This test demonstrates the ability of our algorithm to handle complex two-dimensional simulations in
different applications. Our model includes Darcy equations in the mixed formulation which
are necessary to compute accurately the Darcy velocity in the surrounding rock.

\section{Conclusions}
In this paper, we propose a multiphysics finite element method based on Nitsche's technique for Stokes-poroelasticity problem. To better describe the multiphysics process of deformation and diffusion for Stokes Stokes-poroelasticity problem, we first present a reformulation of the original problem by introducing two pseudo-pressures, which reveals the underlying deformation and diffusion multiphysics process. Then, we define the weak solution of the reformulated problem and prove the existence and uniqueness of weak solution of the original problem and the reformulated problem. We use Nitsche's technique to approximate the coupling condition at the interface and analyze the stability of the reformulated full-coupling problem. For the reformulated full-coupling problem, we propose a loosely-coupled time-stepping method, which decouples the problem into three sub-problems at each time step. And we prove the loosely-coupled time-stepping method with good stability and no "locking" phenomenon. Also, we give the error estimates of the loosely-coupled time-stepping method. A practical advantage of the time-stepping algorithm allows one to use any convergent Stokes solver together with any convergent diffusion equation solver to solve the Stokes-poroelasticity problem. From the point of calculation, this method is very effective and accurate. To the best of our knowledge, it is first time to use the multiphysics finite element method with Nitsche's technique to solve the Stokes-poroelasticity problem and give the error estimates.


\end{document}